\documentclass[letterpaper,11pt,oneside,reqno]{amsart}

\usepackage[sorting=nyt,style=alphabetic,backend=bibtex,hyperref=true,doi=false,maxbibnames=9,maxcitenames=4]{biblatex}
\makeatletter
\def\blx@maxline{77}
\makeatother
\usepackage{amsmath,amssymb,amsthm,amsfonts}
\usepackage{graphicx,color}
\usepackage{upgreek}
\usepackage[mathscr]{euscript}
\usepackage{hyperref}

\allowdisplaybreaks
\numberwithin{equation}{section}

\usepackage{tikz}
\usetikzlibrary{decorations.markings,calc}

\usepackage{array}
\usepackage{adjustbox}
\usepackage{cleveref}
\usepackage{enumerate}

\usepackage[DIV12]{typearea}

\newcommand{\XProcess}{\mathbf{X}}
\newcommand{\VertexXProcess}{\mathbf{X}_{\mathrm{hc}}}
\newcommand{\DiscreteVertexXProcess}{\mathbf{{X}}_{\mathrm{v}}}

\newcommand{\RateLambda}{{\uplambda}}
\newcommand{\Speed}{{\boldsymbol\upxi}}
\newcommand{\RoadblockSet}{\mathbf{B}}
\newcommand{\RoadblockProb}{\mathbf{p}}
\newcommand{\HeightFunction}{\mathfrak{h}}

\newcommand{\LimitShape}{\mathscr{H}}

\newcommand{\HahnDistribution}{{\boldsymbol{\upvarphi}}^{\mathrm{hc}}}
\newcommand{\ContinuousHahnDistribution}{{\boldsymbol\upvarphi}}

\newcommand{\SpeedEssRange}{\boldsymbol{\Xi}}
\newcommand{\SpeedEssRangeCirc}{\boldsymbol{\Xi}^\circ}

\renewcommand{\Re}{\mathop{\mathrm{Re}}}
\renewcommand{\Im}{\mathop{\mathrm{Im}}}

\newtheorem{proposition}{Proposition}[section]
\newtheorem{lemma}[proposition]{Lemma}

\newtheorem{theorem}[proposition]{Theorem}
\theoremstyle{definition}
\newtheorem{definition}[proposition]{Definition}
\newtheorem{remark}[proposition]{Remark}

\begin{document}
\title{Inhomogeneous exponential jump model}

\author[A. Borodin]{Alexei Borodin}\address{A. Borodin, Department of Mathematics, Massachusetts Institute of Technology, 77 Massachusetts ave., Cambridge, MA 02139, USA, and Institute for Information Transmission Problems, Bolshoy Karetny per. 19, Moscow, 127994, Russia}\email{borodin@math.mit.edu}

\author[L. Petrov]{Leonid Petrov}\address{L. Petrov, University of Virginia, Department of Mathematics, 141 Cabell Drive, Kerchof Hall, P.O. Box 400137, Charlottesville, VA 22904, USA, and Institute for Information Transmission Problems, Bolshoy Karetny per. 19, Moscow, 127994, Russia}\email{lenia.petrov@gmail.com}

\date{}

\begin{abstract}
	We introduce and study the inhomogeneous exponential jump model --- an
	integrable stochastic interacting particle system on the continuous half line
	evolving in continuous time.
	An important feature of the system is the presence of arbitrary spatial
	inhomogeneity on the half line which does not break the integrability.
	We completely characterize the macroscopic limit shape and asymptotic
	fluctuations of the height function (=~integrated current) in the model.
	In particular, we explain how the presence of inhomogeneity may lead to macroscopic
	phase transitions in the limit shape such as shocks or traffic jams. 
	Away from these singularities the asymptotic fluctuations of the height
	function around its macroscopic limit shape are governed by the GUE
	Tracy--Widom distribution.
	A surprising result is that while the limit shape is discontinuous at a
	traffic jam caused by a macroscopic slowdown in the inhomogeneity, fluctuations on both
	sides of such a traffic jam still have the GUE Tracy--Widom distribution (but
	with different non-universal normalizations).

	The integrability of the model comes from the fact that it is a degeneration
	of the inhomogeneous stochastic higher spin six vertex models studied
	earlier in \cite{BorodinPetrov2016inhom}. 
	Our results on fluctuations are obtained via an asymptotic analysis of
	Fredholm determinantal formulas arising from contour integral expressions for
	the $q$-moments in the stochastic higher spin six vertex model.
	We also discuss ``product-form'' translation invariant stationary
	distributions of the exponential jump model which lead to an alternative
	hydrodynamic-type heuristic derivation of the macroscopic limit shape.
\end{abstract}

\maketitle

\setcounter{tocdepth}{1}
\tableofcontents
\setcounter{tocdepth}{3}

\section{Introduction}
\label{sec:introduction}

\subsection{Background}

The study of nonequilibrium stochastic interacting particle systems in one space
dimension (together with applications to traffic models and other settings) has
been successful for several past decades
\cite{macdonald1968bioASEP},
\cite{Spitzer1970}, 
\cite{Liggett1985}, 
\cite{spohn1991large},
\cite{helbing2001trafficSurvey}.
A prototypical example of a particle system modeling traffic on a one-lane road
is TASEP (Totally Asymmetric Simple Exclusion Process) in which particles
evolve on $\mathbb{Z}$.
Important questions about interacting particle systems include describing their
asymptotic (long-time and large-scale) behavior, and understanding how this
behavior depends on the initial condition.
Since early days hydrodynamic-type methods 
have been applied to answer these questions
(%
	e.g., 
	\cite{Andjel1984}, 
	\cite{Rezakhanlou1991hydrodynamics},
	\cite{seppalainen1999existence}%
), 
which allowed to establish laws of large numbers for asymptotic particle
locations and integrated particle currents.

The introduction of exact algebraic (``integrable'') techniques into the study
of interacting particle systems pioneered in \cite{johansson2000shape} brought
results on asymptotics of fluctuations (i.e., the next order of asymptotics
after the law of large numbers).
In particular, Johansson \cite{johansson2000shape} showed that the fluctuations
in the TASEP started from a densely packed (``step'') initial configuration
are governed by the GUE Tracy--Widom distribution from the random matrix theory
\cite{tracy_widom1994level_airy}.
This and related fluctuation results contribute to a general belief that driven
interacting particle systems with an exclusion mechanism belong (under mild
conditions) to the Kardar--Parisi--Zhang universality class 
\cite{FerrariSpohnHandbook},
\cite{CorwinKPZ}, 
\cite{QuastelSpohnKPZ2015}.

Studying asymptotics of interacting particle systems by means of exact formulas
have brought much progress over the last two decades. 
At the same time, these methods have certain limitations
even when applied to a single particular particle system such as TASEP.
For example, proving asymptotic results for general (arbitrary) initial data
is typically quite hard
(cf.{} however the case of TASEP in \cite{matetski2017kpz}).
Another type of questions which has been evading integrable methods
is the asymptotic behavior of systems in spatially inhomogeneous environment
(also referred to as systems with defects or site-disordered systems).
Spatial inhomogeneity should be contrasted with another type of inhomogeneity
under which each particle has its own speed parameter (such as the jump rate in
TASEP). 
Systems with particle-dependent inhomogeneity
often\footnote{%
	But not always: a notable open problem is to find an integrable structure in
	ASEP (a two-sided generalization of TASEP) with particle-dependent jump rates.%
} 
possess the same 
integrable structure as their homogeneous counterparts
(%
	in the case of TASEP this structure comes from Schur processes 
	\cite{okounkov2001infinite},
	\cite{okounkov2003correlation},
	\cite{BorFerr2008DF}%
).
This integrability
then leads to fluctuation results for systems with particle-dependent inhomogeneity
(%
	e.g., 
	\cite{BaikBBPTASEP}, 
	\cite{BorodinEtAl2009TwoSpeed}, 
	\cite{Duits2011GFF}, 
	\cite{barraquand2015phase}%
).

Interacting particle systems with spatial inhomogeneity are connected to
situations naturally arising in traffic models, and have been the subject
of numerical simulations
(%
	e.g., 
	\cite{krug1996phase},
	\cite{krug1999simulation},
	\cite{krug2000phase},
	\cite{dong2008understanding}, 
	see also \cite{helbing2001trafficSurvey}%
).
Moreover, such particle systems were extensively studied by hydrodynamic
methods and other techniques, which has led to law of large numbers type
results for various models including TASEP
\cite{Landim1996hydrodynamics}, 
\cite{seppalainen1999existence},
\cite{Seppalainen_Discont_TASEP_2010},
\cite{blank2011exclusion}, 
\cite{blank2012discrete}.
A notable hard problem in this direction is the effect 
of a ``slow bond'' (i.e., a selected site jumping over which requires longer waiting time)
on global characteristics of the system such as the current.
In one case this problem was recently resolved in \cite{Basuetal2014_slowbond}
(%
	see also 
	\cite{seppalainen2001slow_bond},
	\cite{costin2012blockage}%
).

\subsection{Inhomogeneous exponential jump model}

We introduce and study 
the inhomogeneous exponential jump model ---
an integrable interacting particle system on $\mathbb{R}_{>0}$
with a rather general spatial
inhomogeneity governed by a piecewise continuous speed function $\Speed(x)$, $x\in \mathbb{R}_{\ge0}$. 
Let us briefly describe this system (see \Cref{sub:definition_of_the_model} below
for a detailed definition in full generality). 
Initially the configuration of particles $\XProcess(0)$ in $\mathbb{R}_{>0}$ is empty,
and at any time $t\ge0$ the particle configuration $\XProcess(t)$ on $\mathbb{R}_{>0}$
is a finite collection of finite particle stacks.
That is, one location on $\mathbb{R}_{>0}$ can be occupied by several particles.
In continuous time one particle can become active and 
leave a stack at a given location $x$ with rate 
$\Speed(x)(1-q^{\#\left\{ \textnormal{number of particles in the stack} \right\}})$
(here $0<q<1$ is a fixed parameter),
or a new active particle can be added to the system at location $0$ at rate 
$\Speed(0)$.\footnote{%
	Here and below we say that a certain event has rate $\mu>0$ if it repeats
	after independent random time intervals which have exponential distribution
	with rate $\mu$ (and mean $\mu^{-1}$). 
	That is, $\mathbb{P}(\textnormal{time between occurrences}>\Delta t) =e^{-\mu
	\Delta t}$. These independent exponentially distributed random times are also assumed
	independent from the rest of the system.%
}
In continuous time almost surely only one particle can become active.
The active particle desires to travel to the right by an exponentially distributed random
distance with mean $1/\RateLambda$ (where $\RateLambda>0$ is a parameter of the model),
but it may be stopped by other particles it encounters along the way. 
Namely, the active particle jumps over each sitting particle with probability $q$, 
or else is stopped and joins the corresponding stack of particles.
See \Cref{fig:expon_model_def} for an illustration 
(and for now assume that $\RoadblockProb(b)=1$ in that picture).

Let the height function
$\HeightFunction_{\XProcess(t)}(x)$
be the number of particles in our configuration 
at time $t\ge0$ which are weakly to the right of the location $x\in\mathbb{R}_{>0}$.
For each $t$ this is a random nonincreasing function in $x$.
Moreover, 
$\HeightFunction_{\XProcess(t)}(x)$
weakly increases in $t$ for fixed $x$.

The inhomogeneous exponential jump model is integrable in the sense that 
we are able to explicitly compute observables 
$\mathbb{E}\bigl((1-\zeta H)^{-1}(1-\zeta Hq)^{-1}(1-\zeta Hq^{2})^{-1}\ldots \bigr)$,
$H=q^{\HeightFunction_{\XProcess(t)}(x)}$
(where $x,t>0$ are arbitrary and $\zeta\in \mathbb{C}\setminus\mathbb{R}_{\ge0}$
is a parameter of the observable),
in a Fredholm determinantal form (see \Cref{thm:exponential_Fredholm} below).
This Fredholm determinant is amenable to asymptotics and 
opens a way to study law of large numbers and fluctuations
of the inhomogeneous exponential jump model.

\subsection{Asymptotic behavior}

We are interested in describing the asymptotic behavior of the 
inhomogeneous exponential jump model in terms of its height function
$\HeightFunction_{\XProcess(t)}(x)$ as $\RateLambda\to+\infty$
and the time scales as $t=\tau\RateLambda$, $\tau>0$.
We assume that $q\in(0,1)$ and the speed function $\Speed(\cdot)$ are fixed.
When $\RateLambda$ grows, the expected distance of individual jumps of the particles
goes to zero, while more and more particles are added to the system with time. 
Our asymptotic results are the following:
\begin{enumerate}[\quad\bf1.]
	\item We show that there exists a limit shape $\LimitShape(\tau,x)$ such that
		$\lim_{\RateLambda\to+\infty}
		\RateLambda^{-1}\HeightFunction_{\XProcess(\tau\RateLambda)}(x)
		=\LimitShape(\tau,x)$ in probability for any $\tau,x>0$.
		The limit shape is determined
		by $\Speed(\cdot)$ and $q$ in terms of integrals of
		$q$-polygamma functions and their inverses.

	\item We also present an informal hydrodynamic-type argument
		(%
			relying on constructing a one-parameter family of translation invariant stationary
			distributions for the homogeneous exponential jump model
			with arbitrary density $\rho>0$, and computing 
			the associated particle current $j(\rho)$%
		)
		suggesting that the limit shape $\LimitShape(\tau,x)$ should satisfy a
		partial differential equation in $\tau$ and $x$. 
		This equation is similar to
		the one considered for TASEP in inhomogeneous environment,
		see \cite{Seppalainen_Discont_TASEP_2010} and references therein.
		We then explicitly verify that $\LimitShape(\tau,x)$ described in terms of $q$-polygamma
		functions satisfies this equation.
		
	\item Our main result is that the inhomogeneous exponential jump model
		belongs to the Kardar--Parisi--Zhang universality class, that is, the fluctuations
		of the random height function around the limit shape have scale $\RateLambda^{\frac{1}{3}}$
		and are governed by the GUE Tracy--Widom distribution $F_2$:
		\begin{equation*}
				\lim_{\RateLambda\to+\infty}
				\mathbb{P}
				\left( 
					\frac{\HeightFunction_{\XProcess(\tau\RateLambda)}(x)-
					\RateLambda \LimitShape(\tau,x)}
					{
						\RateLambda^{\frac{1}{3}}
						c(\tau,x)
					}
					\ge -r
				\right)=F_2(r),\qquad r\in\mathbb{R}.
		\end{equation*}
		Here $x\in(0,x_e(\tau))$, where $x_e(\tau)$ is the asymptotic location
		of the rightmost particle in the system, i.e.,
		$\LimitShape(\tau,x)=0$
		for all $x\ge x_e(\tau)$.
\end{enumerate}

We obtain the limit shape and the fluctuation results 
simultaneously by analyzing the Fredholm
determinant of \Cref{thm:exponential_Fredholm} 
by the steepest descent method.
Because the model depends on an arbitrary speed function $\Speed(\cdot)$, 
we had to make sure that this analysis does not impose unnecessary 
restrictions on the generality of this function.
This presented the main technical challenges of our work.

One of the most striking features of our asymptotic results is a new type of
phase transitions caused by a sufficient decrease in the speed function
$\Speed(\cdot)$ on an interval.
Namely, at the beginning of such a decrease the limit shape
$\LimitShape(\tau,x)$ becomes discontinuous (leading to what we call a traffic jam), 
but the asymptotic fluctuations of the height
function on both sides of this traffic jam 
(and at the location of the traffic jam itself!) have scale
$\RateLambda^{\frac{1}{3}}$ and the GUE Tracy--Widom distribution,
but with different non-universal normalizations.
Computer simulations also suggest that these Tracy--Widom fluctuations on both sides
of a traffic jam are independent. 
A finer analysis of the fluctuation behavior in a neighborhood 
of a traffic jam will be the subject of a future work.

In fact, we also consider a slightly more general situation 
when the inhomogeneous space might contain deterministic roadblocks
which capture particles with some fixed probabilities.
The presence of these roadblocks leads to 
shocks in the limit shape and
phase transitions in the fluctuation exponent and fluctuation distribution
of Baik--Ben Arous--P\'ech\'e type \cite{BBP2005phase}.
This phase transition is known to appear in
interacting particle systems with particle-dependent inhomogeneity
(e.g., see \cite{BaikBBPTASEP}, \cite{barraquand2015phase})
and in other related situations (cf. \cite{AmolBorodin2016Phase}).
We refer to \Cref{thm:main_theorem_on_fluctuations} below for a complete
formulation of the results on asymptotics of fluctuations, and to
\Cref{sub:phase_transitions_short_discussion} for more discussion and examples
of phase transitions arising for various choices of the speed function
$\Speed(\cdot)$ and the configuration of the roadblocks.

Let us emphasize that the ability to analyze 
an interacting particle system with spatial inhomogeneity 
to the point of asymptotic fluctuations 
comes from new integrable tools
developed in \cite{BorodinPetrov2016inhom}
for the inhomogeneous six vertex model.
It seems that earlier methods of Integrable Probability
are not directly applicable to such particle systems 
with spatial inhomogeneity.

\subsection{Remark. Model for $q=0$}
The inhomogeneous exponential jump model has a natural degeneration for
$q=0$ which changes the behavior of the particle system in two aspects.
First, for $q=0$ particles leave each stack and become active at rate
$\Speed(x)$ (where $x$ is the location of this stack), independently of the
number of particles in the stack.
Second, moving particles cannot fly over sitting particles, so one can say that
the particles are ordered and the process preserves this ordering.

Although the $q=0$ process is simpler than the one for $q\in(0,1)$, the
methods of the present paper are not directly applicable to rigorously obtaining
asymptotics of fluctuations in the $q=0$ case.
However, the $q\in(0,1)$ results in the present paper have natural $q=0$
degenerations which are proven in a companion paper \cite{KnizelPetrov2017}
using a different approach based on a connection with Schur measures
(which in turn follows from the coupling construction of \cite{OrrPetrov2016}).

This need for a different approach for $q=0$ should be compared to the
situation of ASEP and $q$-TASEP vs TASEP.
Namely, the asymptotic analysis of ASEP or $q$-TASEP by means of Fredholm
determinants (see \cite{TW_ASEP2} and \cite{FerrariVeto2013}, respectively)
does not survive the limit transition to the TASEP.
On the other hand, some of the ASEP and $q$-TASEP results (most notably, on
the GUE Tracy--Widom fluctuations of the height function) remain valid in the
TASEP case and were established earlier in 
\cite{johansson2000shape} 
by a different method
which can also be traced to Schur measures.

\subsection{Outline}
In \Cref{sec:model_and_results} we define the inhomogeneous exponential jump
model in full generality, and describe the main results of the paper.
In \Cref{sec:from_vertex_to_expon} we show how formulas for the stochastic
higher spin six vertex model from \cite{BorodinPetrov2016inhom} lead to a
Fredholm determinantal formula for the $q$-Laplace transform of the height
function of the exponential model.
In \Cref{sec:asymptotic_analysis} we perform the asymptotic analysis of the
Fredholm determinant and prove the main results.
Necessary formulas pertaining to $q$-digamma and $q$-polygamma functions are
given in \Cref{sec:appendix_q}.
In \Cref{sec:appendix_stationary_distributions} we discuss translation
invariant stationary distributions of our particle systems, and perform
computations leading to an alternative hydrodynamic-type heuristic derivation
of the macroscopic limit shape in the inhomogeneous exponential jump model.

\subsection{Acknowledgments}

The authors are grateful to discussions with
Guillaume Barraquand,
Ivan Corwin,
Michael Blank,
Tomohiro Sasamoto, 
Herbert Spohn,
Kazumasa Takeuchi,
and
Jon Warren.
The research was carried out in part during the authors' participation in the
Kavli Institute for Theoretical Physics program ``New approaches to
non-equilibrium and random systems: KPZ integrability, universality,
applications and experiments'', and consequently was partially supported by the
National Science Foundation PHY11-25915.
A.~B.\ is supported by the National Science Foundation grant
DMS-1607901 and by Fellowships of the Radcliffe Institute for Advanced Study and the Simons
Foundation.

\section{Model and main results}
\label{sec:model_and_results}

\subsection{Definition of the model}
\label{sub:definition_of_the_model}

The \emph{inhomogeneous exponential jump model} is a continuous time Markov
process $\left\{ \XProcess(t) \right\}_{t\ge0}$ on the space of finite particle
configurations in $\mathbb{R}_{>0}:=\left\{ y\in\mathbb{R}\colon y>0 \right\}$, that is,
\begin{equation*}
	\mathrm{Conf}_{\bullet}(\mathbb{R}_{>0}):=\left\{ \left( x_1\ge x_2\ge \ldots \ge x_k \right)\colon 
	\textnormal{$x_i\in \mathbb{R}_{>0}$ and $k\in\mathbb{Z}_{\ge0}$ is arbitrary} \right\}.
\end{equation*}
Note that several particles can occupy the same point of $\mathbb{R}_{>0}$.
Denote by $\varnothing \in\mathrm{Conf}_{\bullet}(\mathbb{R}_{>0})$ the empty particle
configuration (having $k=0$), and let the initial configuration of the
exponential jump model be $\XProcess(0)=\varnothing$. 
For later convenience, let us also assume that there is an infinite stack of
particles at location $0$.

For $X\in\mathrm{Conf}_{\bullet}(\mathbb{R}_{>0})$ and $x\in\mathbb{R}_{\ge0}$ denote
by $\eta(x)=\eta_X(x)\in\mathbb{Z}_{\ge0}$ the number of particles of the
configuration $X$ at location $x$. Next, define the \emph{height function}
associated with $X\in\mathrm{Conf}_{\bullet}(\mathbb{R}_{>0})$ by
\begin{equation}
	\HeightFunction(x)=\HeightFunction_X(x):=
	\#\left\{ \textnormal{particles $x_i\in X$ which are $\ge x$} \right\}
	\in\mathbb{Z}_{\ge0}.
	\label{height_function_definition}
\end{equation}
The height function is weakly decreasing in $x$,
$\HeightFunction_{X}(0)=+\infty$ (due to the infinite stack at $0$), and
$\lim_{x\to+\infty}\HeightFunction_{X}(x)=0$.

The inhomogeneous exponential jump model $\XProcess(t)$ on
$\mathrm{Conf}_{\bullet}(\mathbb{R}_{>0})$ depends on the following data:
\begin{equation}
	\label{Expmodel_data_assumptions}
	\hspace{-5pt}
	\parbox{.88\textwidth}{
		\begin{itemize}
		\item Main ``\emph{quantization}'' parameter $q\in(0,1)$;
		\item \emph{Jumping distance} parameter $\RateLambda>0$;
		\item \emph{Speed function} $\Speed(x)$, $x\in\mathbb{R}_{\ge0}$, which is
			positive, piecewise continuous, has left and right limits in
			$\mathbb{R}_{\ge0}$, and 
			is bounded away from $0$ and~$+\infty$.
		\item Discrete set $\RoadblockSet\subset\mathbb{R}_{>0}$ without
			accumulation points in $\mathbb{R}_{\ge0}$ (however, $\RoadblockSet$ can be
			infinite) and a function $\RoadblockProb\colon \RoadblockSet\to(0,1)$. Points
			belonging to $\RoadblockSet$ will be referred to as \emph{roadblocks}.
	\end{itemize}
	}
\end{equation}
Under the Markov process $\XProcess(t)$ particles randomly jump to the right in
continuous time. Let us begin by describing how particles ``wake up'' and start
moving. First, new particles enter the system (leaving the infinite stack at location
$0$) at rate $\Speed(0)$.
Next, if at some time $t>0$ there are particles at a location
$x\in\mathbb{R}_{>0}$, then one particle decides to leave this location at
rate $\Speed(x)(1-q^{\eta_{\XProcess(t)}(x)})$. Almost surely at each
time moment at most one particle can start moving, and 
waking up events at different locations are independent.

Each particle which wakes up at some time moment $t\ge0$ instantaneously jumps
to the right by a random distance according to the distribution
\begin{multline}\label{jumpmodel_traveling_particle}
	\mathbb{P}
	\left( 
		\textnormal{%
			the moving particle travels 
			distance $\ge y$
			$\big\vert$
			it started at $x\in\mathbb{R}_{\ge0}$%
		}
	\right)
	\\=
	e^{- \RateLambda y}q^{\HeightFunction_{\XProcess(t)}(x+)-\HeightFunction_{\XProcess(t)}(x+y)}
	\prod_{b\in\RoadblockSet,\,x<b<x+y}\RoadblockProb(b),
\end{multline}
where $y>0$ is arbitrary and the height function above corresponds to the
configuration $\XProcess(t)$ before the moving particle started its jump.  
In words, distribution \eqref{jumpmodel_traveling_particle} means that the
moving particle's desired travel distance has exponential distribution with
rate $\RateLambda$ (and mean $\RateLambda^{-1}$). 
However, in the process of its move the particle has a chance to be stopped by
other particles or by roadblocks it flies over. 
Namely, the moving particle flies past each sitting particle with probability
$q$ per particle (indeed,
$\HeightFunction_{\XProcess(t)}(x+)-\HeightFunction_{\XProcess(t)}(x+y)$ is the
number of such sitting particles strictly between $x$ and $x+y$), and flies past each
roadblock $b\in\RoadblockSet$ with probability $\RoadblockProb(b)$. 
Note that to fly past a roadblock at $b\in\RoadblockSet$ the moving particle
must also pass all particles possibly sitting at $b$, with probability $q$ per
particle. 
See \Cref{fig:expon_model_def} for an illustration.

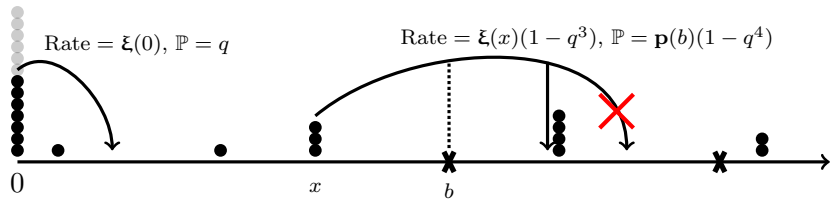
\begin{figure}[htbp]
	\begin{tikzpicture}
		[scale=.9, very thick]
		\def\eps{.17}
		\draw[->] (0,0) -- (12,0);
		\node[below] at (0,0) {$0$};
		\draw[fill] (0,\eps) circle (2pt);
		\draw[fill] (0,2*\eps) circle (2pt);
		\draw[fill] (0,3*\eps) circle (2pt);
		\draw[fill] (0,4*\eps) circle (2pt);
		\draw[fill] (0,5*\eps) circle (2pt);
		\draw[fill] (0,6*\eps) circle (2pt);
		\draw[fill] (0,7*\eps) circle (2pt);
		\draw[fill, opacity=.2] (0,8*\eps) circle (2pt);
		\draw[fill, opacity=.2] (0,9*\eps) circle (2pt);
		\draw[fill, opacity=.2] (0,10*\eps) circle (2pt);
		\draw[fill, opacity=.2] (0,11*\eps) circle (2pt);
		\draw[fill, opacity=.2] (0,12*\eps) circle (2pt);
		\draw[fill, opacity=.2] (0,13*\eps) circle (2pt);
		\draw[fill] (3,\eps) circle (2pt);
		\draw[fill] (4.4,2*\eps) circle (2pt);
		\draw[fill] (4.4,3*\eps) circle (2pt);
		\draw[fill] (.6,\eps) circle (2pt);
		\draw[fill] (4.4,\eps) circle (2pt);
		\node [below] at (4.4,-.15) {\scriptsize $x$};
		\draw[fill] (8,\eps) circle (2pt);
		\draw[fill] (8,2*\eps) circle (2pt);
		\draw[fill] (8,3*\eps) circle (2pt);
		\draw[fill] (8,4*\eps) circle (2pt);
		\draw[fill] (11,\eps) circle (2pt);
		\draw[fill] (11,2*\eps) circle (2pt);
		\draw[->] (0,8*\eps) to [in=90,out=40] (1.4,\eps) node[right, yshift=40pt, xshift=-30pt]
		{\scriptsize$\textnormal{Rate}=\Speed(0)$,
		$\mathbb{P}=q$};
		\draw[->] (4.4,4*\eps) to [in=90,out=40] (9,\eps);
		\draw[->] (8-\eps,8.7*\eps) to (8-\eps,\eps) node[above, yshift=34pt, xshift=15pt]
		{\scriptsize$\textnormal{Rate}=\Speed(x)(1-q^{3})$, $\mathbb{P}=\RoadblockProb(b)(1-q^{4})$};
		\draw[color=red, ultra thick] (8.6,.5)--++(.5,.5);
		\draw[color=red, ultra thick] (8.6,1)--++(.5,-.5);
		\draw[line width=2] (6.3,-.15)--++(.15,.3);
		\draw[line width=2] (6.45,-.15)--++(-.15,.3);
		\draw[line width=2] (10.3,-.15)--++(.15,.3);
		\draw[line width=2] (10.45,-.15)--++(-.15,.3);
		\node [below] at (6.375,-.15) {\scriptsize $b$};
		\draw[densely dotted] (6.375,1.2*\eps)--++(0,1.3);
	\end{tikzpicture}
	\caption{%
		Two possible jumps in the inhomogeneous exponential jump model
		$\XProcess(t)$ (but only one can occur at an instance of continuous time). 
		The left jump has rate $\Speed(0)$ and the desired travel distance has
		exponential distribution with parameter $\RateLambda$. 
		The moving particle flies over one sitting particle with probability~$q$. 
		The right jump has rate $\Speed(x)(1-q^{3})$, the desired travel distance
		has the same exponential distribution, and the moving particle joins the
		stack of $4$ other particles with probability $\RoadblockProb(b)(1-q^{4})$. 
		Here $\RoadblockProb(b)$ corresponds to flying over the roadblock at $b$,
		and $1-q^{4}$ is the complementary probability to flying over the stack of
		$4$ sitting particles.%
	}
	\label{fig:expon_model_def}
\end{figure}

\begin{remark}
	\label{rmk:lambda_parameter_depends_on_x}
	The parameter $\RateLambda$ can in fact depend on the location $x$, too.
	However, this dependence can be eliminated by a change of variables (cf.
	\Cref{rmk:enough_lambda_constant}), and so without loss of generality we can
	and will consider constant $\RateLambda$.
\end{remark}

\subsection{Hydrodynamics}
\label{sub:intro_hydrodynamics}

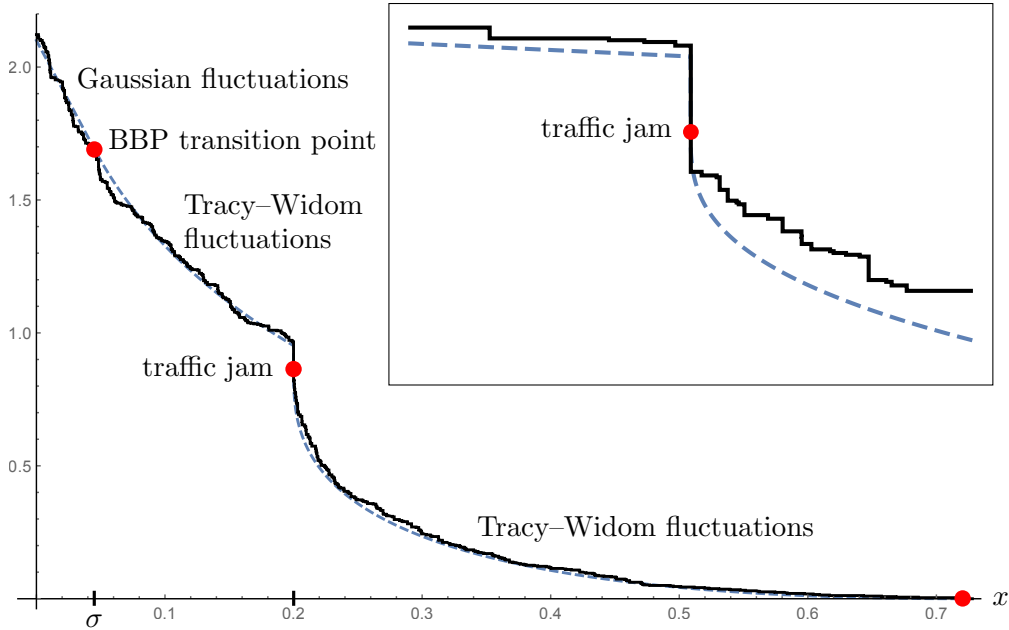
\begin{figure}[htpb]
	\centering
	\begin{tikzpicture}
		[scale=1]
		\node[anchor=north east] at (0,0) {\includegraphics[width=.8\textwidth]{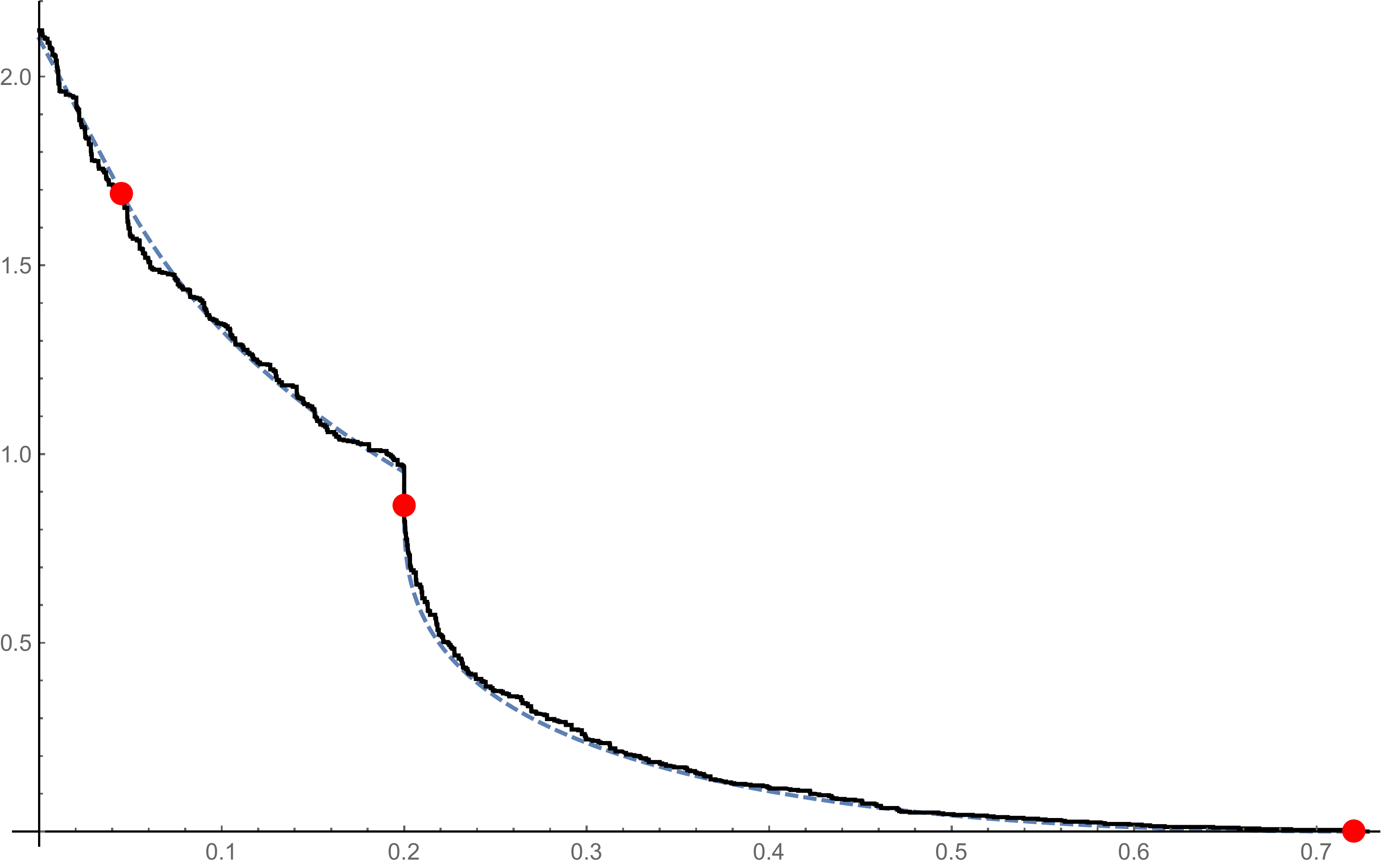}};
		\node[anchor=north east,rectangle,draw=black] at (0,0) {\includegraphics[width=.48\textwidth]{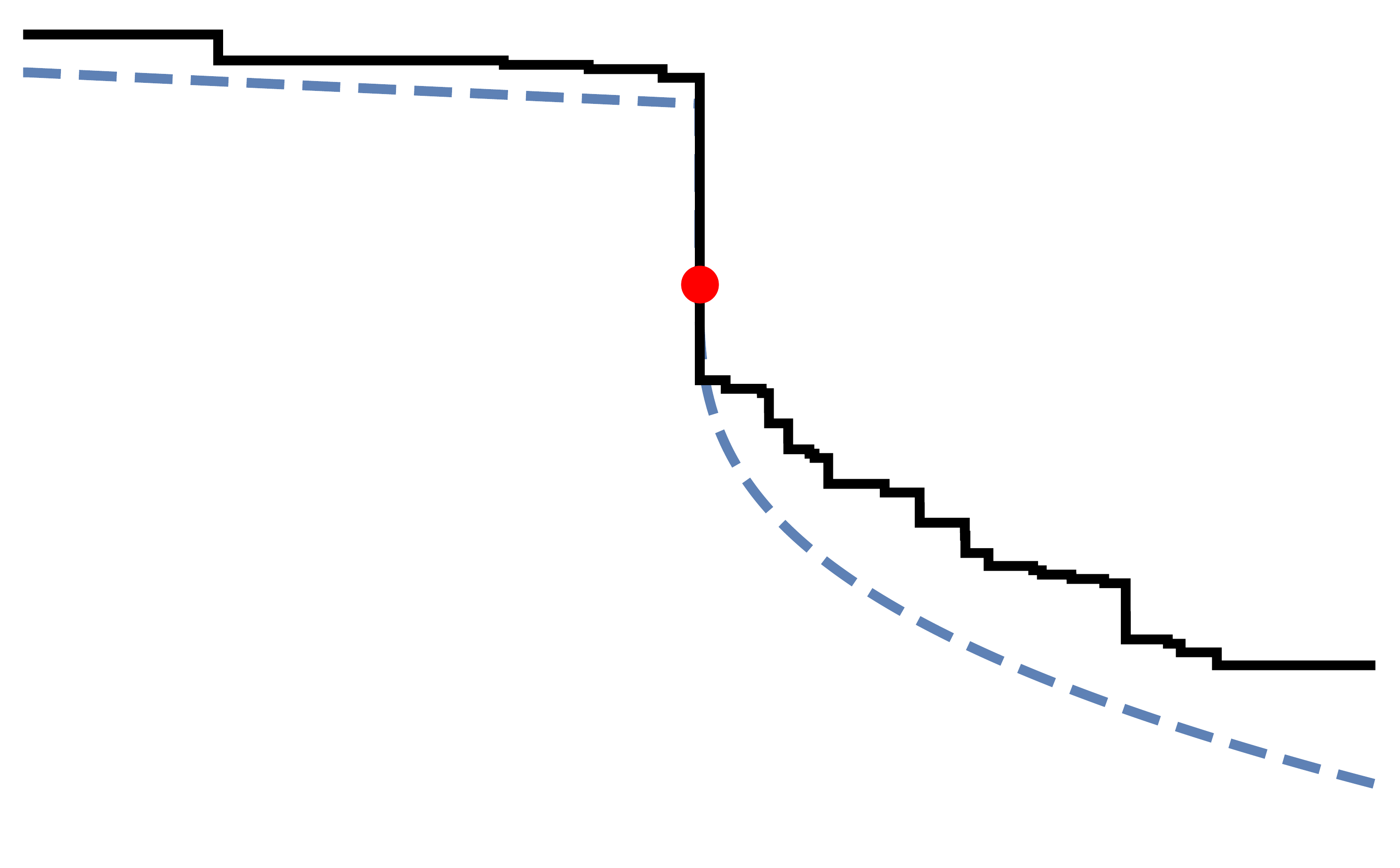}};
		\node[anchor=west] at (-12.32,-1) {Gaussian fluctuations};
		\node[anchor=west] at (-11.9,-1.85) {BBP transition point};
		\node[anchor=west] at (-10.9,-2.9) {\parbox{.16\textwidth}{Tracy--Widom fluctuations}};
		\node[anchor=west] at (-11.45,-4.87) {traffic jam};
		\node[anchor=west] at (-6.15,-1.69) {traffic jam};
		\node[anchor=west] at (-7,-7) {Tracy--Widom fluctuations};
		\draw[very thick] (-11.953,-7.82)--++(0,-.2) node [below, yshift=1] {$\sigma$};
		\draw[very thick] (-9.302,-7.82)--++(0,-.2);
		\node at (0.1,-7.93) {$x$};
	\end{tikzpicture}
	\caption[]{Example of a random height function (solid),
		its limit shape approximation $\LimitShape(\tau,x)$
		(dashed), and a zoom at $x=0.2$ showing the behavior around the traffic jam.
		The parameters are $q=\frac{1}{2}$, $\tau=3$, and the speed function is
		$
			\Speed(x)
			=
			0.7\cdot \mathbf{1}_{x=0}
			+
			\mathbf{1}_{0<x<0.2}
			+
			0.4\cdot\mathbf{1}_{x\ge0.2}
		$
		(%
			here and below $\mathbf{1}_{A}$ stands for the indicator function of an
			event $A$%
		).
		Solid red dots indicate points where a curved shape is tangent to 
		the adjacent linear part.
		\\
		The value $\Speed(0)=0.7<1$ plays the role of a roadblock and creates a
		singularity, namely,
		the fluctuation exponents and 
		the fluctuation behavior undergoes a
		Baik--Ben Arous--P\'ech\'e phase transition 
		at $x=\sigma\approx 0.045$.
		Moreover, $\LimitShape(\tau,x)$ is linear for $0\le x\le \sigma$.
		The discontinuous decrease in $\Speed(x)$
		at $x=0.2$ leads to a traffic jam, namely,
		the limit shape is discontinuous at $x=0.2$ but 
		fluctuations on both sides of $0.2$ are governed by the GUE Tracy--Widom distribution.
	}
	\label{fig:intro_limit_shape}
\end{figure}

The main goal of this paper is to perform an asymptotic analysis of the
inhomogeneous exponential jump model $\XProcess(t)$ in the regime as
$\RateLambda\to+\infty$ and the continuous time grows as $t=\tau\RateLambda$,
where $\tau\ge0$ is the fixed rescaled time.
In this regime the mean desired travel distance of the jumping particles
goes to zero, while more and more particles enter the system
as time grows.
In this regime one expects that the rescaled height function 
\eqref{height_function_definition} converges (in probability) to a deterministic
\emph{limit shape}, 
$
	\RateLambda^{-1}\HeightFunction_{\XProcess(\tau\RateLambda)}(x)
	\to
	\LimitShape(\tau,x)
$.
An example of such a limit shape is given in \Cref{fig:intro_limit_shape}.

At least in the case of no roadblocks, a partial differential equation for the
limit shape $\LimitShape(\tau,x)$ can be written down by looking at the
hydrodynamic behavior similarly to the treatment of driven interacting
particle systems in one space dimension in, e.g., 
\cite{Andjel1984},
\cite{Rezakhanlou1991hydrodynamics}, 
\cite{Landim1996hydrodynamics},
\cite{Seppalainen_Discont_TASEP_2010}.
To write such an equation for our model, we need the following notation:
\begin{equation*}
	\phi_n(w):=\sum_{k=1}^{\infty}\frac{k^nw^k}{1-q^k},\qquad |w|<1,
	\qquad 
	n\in\mathbb{Z}_{\ge0}.
\end{equation*}
This function can be analytically continued to a meromorphic function of $w\in\mathbb{C}$,
see \Cref{sec:appendix_q} for details.
\begin{theorem}
	\label{thm:PDE}
	Let there be no roadblocks,
	and $\Speed(x)$ be continuous at $x=0$.
	Then the limit shape $\LimitShape(\tau,x)$
	(whose existence follows from \Cref{thm:main_theorem_on_fluctuations}
	below)
	satisfies the following
	equation:
	\begin{equation}
		\frac{\partial\LimitShape(\tau,x)}{\partial x}
		=
		-
		\phi_1\left( \frac{1}{\Speed(x)}
		\frac{\partial\LimitShape(\tau,x)}{\partial \tau}\right),
		\label{intro_PDE_height_function}
	\end{equation}
	with initial condition
	$\LimitShape(0,x)=0$ ($x>0$)
	and boundary condition
	$\LimitShape(\tau,0)=+\infty$ ($\tau\ge0$).
\end{theorem}
We require that $\Speed(x)$ is continuous at $x=0$ to avoid the
singularity near $x=0$ as in \Cref{fig:intro_limit_shape}.
\begin{proof}[Heuristic argument for \Cref{thm:PDE}]
	We first present a heuristic hydrodynamic-type argument leading to equation
	\eqref{intro_PDE_height_function}. 
	Later in \Cref{sub:intro_limit_shape} we will verify this equation using an
	explicit expression for $\LimitShape(\tau,x)$ arising from asymptotic
	analysis of the Fredholm determinantal formula for the $q$-Laplace transform
	of the height function $\HeightFunction_{\XProcess(t)}$ of the exponential
	jump model.
	Details and necessary computations pertaining to the hydrodynamic approach 
	may be found in \Cref{sec:appendix_stationary_distributions}.
	
	The hydrodynamic argument is based on the following assertions:
	\begin{enumerate}[\quad$\bullet$]
		\item (\emph{existence of limit shape})
			The limit (in probability)
			$\LimitShape(\tau,x)=\lim_{\RateLambda\to+\infty}
			\RateLambda^{-1}\HeightFunction_{\XProcess(\tau\RateLambda)}(x)$
			exists, and
			$\rho(\tau,x):=-\frac{\partial}{\partial x}\LimitShape(\tau,x)\in[0,+\infty]$ 
			exists for any $x\in\mathbb{R}_{\ge0}$.
			Clearly, $\rho(\tau,x)$ has the meaning of the limiting density (at
			location $x\in\mathbb{R}_{\ge0}$ at scaled time $\tau\ge0$) of particles in
			the inhomogeneous exponential jump model.

		\item (\emph{local stationarity})
			Locally at each $x\in\mathbb{R}_{>0}$ where $\rho(\tau,x)< +\infty$
			the asymptotic behavior of our particle system (%
				as $\RateLambda\to+\infty$ and under the rescaling of the space around
				$x$ by $\RateLambda$%
			) is described by 
			a translation invariant stationary distribution\footnote{%
				By stationary we mean distributions which do not change under the
				corresponding stochastic evolution, and translation invariance means
				invariance under spatial translations of $\mathbb{R}$.%
			}
			on $\mathrm{Conf}^{\sim}_{\bullet}(\mathbb{R})$, the space of (possibly
			countably infinite) particle configurations in $\mathbb{R}$ with multiple
			particles per location allowed.

		\item (\emph{classification of translation invariant stationary distributions})
			All distributions on $\mathrm{Conf}^{\sim}_{\bullet}(\mathbb{R})$ which are
			translation invariant and stationary under the homogeneous version of the
			exponential jump model with speed $\Speed(x)\equiv \Speed$ and
			$\RateLambda=1$ depend on one real parameter $c\ge0$ and are 
			given by the marked Poisson processes $\mathfrak{m}_{c,1}$
			defined in \Cref{sub:appendix_expon_model}.
			That is, a random configuration under $\mathfrak{m}_{c,1}$
			is obtained by taking the standard Poisson process on $\mathbb{R}$
			of intensity $\phi_0(c)$, and independently 
			putting $j\ge1$ particles at each point of this Poisson process 
			with probability
			$\ContinuousHahnDistribution_c(j)=
			\frac{1}{\phi_0(c)}\frac{c^j}{1-q^j}$.
	\end{enumerate}

	We prove the first assertion about the limit shape
	using exact formulas (\Cref{thm:main_theorem_on_fluctuations}),
	and do not prove local stationarity.
	We also do not prove the full
	classification of translation invariant stationary
	distributions, but establish
	its weaker version (\Cref{statement:cont_invariant}) that the marked Poisson
	process $\mathfrak{m}_{c,1}$ is indeed stationary under the 
	homogeneous exponential jump model on $\mathbb{R}$
	(which exists for a certain class of initial configurations, see
	\Cref{sub:appendix_preliminaries}). 

	One can compute (\Cref{sub:appendix_expon_model})
	the particle density and the particle
	current (sometimes also called the particle flux)
	associated with $\mathfrak{m}_{c,1}$, they have the form:
	\begin{equation*}
		\rho(c)=\phi_1(c),\qquad j(c)=\Speed c.
	\end{equation*}
	From \Cref{statement:Phi_n_behavior} it readily follows that 
	$\phi_1\colon[0,1]\to[0,+\infty]$ is one to one and increasing, and so the 
	particle current in the local stationary regime 
	depends on the density as 
	\begin{equation*}
		j(\rho)=\Speed \phi_1^{-1}(\rho).
	\end{equation*}
	One then expects that the limiting density satisfies
	\begin{equation}
		\frac{\partial}{\partial \tau}\rho(\tau,x)
		+
		\frac{\partial}{\partial x}j\bigl( \rho(\tau,x) \bigr)
		=0
		\quad 
		\Leftrightarrow
		\quad 
		\frac{\partial}{\partial \tau}\rho(\tau,x)
		+
		\frac{\partial}{\partial x}
		\Bigl[
			\Speed(x)\phi_1^{-1}(\rho(\tau,x))
		\Bigr]=0,
	\label{intro_PDE_density}
	\end{equation}
	with initial condition $\rho(0,x)=0$ ($x>0$) and boundary condition
	$\rho(\tau,0)=+\infty$ ($\tau\ge0$). 

	Using the fact that
	\begin{equation*}
		\LimitShape(\tau,x)=\int_{x}^{+\infty}\rho(\tau,u)du,
	\end{equation*}
	and integrating \eqref{intro_PDE_density} from $x$ to $+\infty$,
	we obtain
	\begin{equation*}
		\frac{\partial\LimitShape(\tau,x)}{\partial\tau}=\Speed(x)\phi_1^{-1}\bigl(\rho(\tau,x)\bigr).
	\end{equation*}
	Dividing this by $\Speed(x)$ and applying $\phi_1$ to both sides leads to the desired equation
	\eqref{intro_PDE_height_function}.
\end{proof}

\subsection{Limit shape}
\label{sub:intro_limit_shape}

Let us now present an explicit expression for the limit shape $\LimitShape(\tau,x)$ 
of the height function in the inhomogeneous exponential jump model
in full generality of \Cref{sub:definition_of_the_model}, i.e., 
with possible roadblocks.
We start with some notation. 

\begin{definition}[Essential ranges]
	\label{def:essential_ranges}
	Denote for $x>0$:
	\begin{equation}
		\SpeedEssRange_x:=
		\{\Speed(0)\}
		\cup\bigcup_{b\in\RoadblockSet\colon 0<b<x}\{\Speed(b)\}
		\cup
		\mathsf{EssRange}\{\Speed(y)\colon 0< y <x\},
		\qquad 
		\mathcal{W}_x:=\min \SpeedEssRange_x,
		\label{range_of_Speed_notation}
	\end{equation}
	where $\mathsf{EssRange}$ stands for the \emph{essential range}, i.e., the set of all
	points for which the preimage of any neighborhood under $\Speed$ has positive
	Lebesgue measure. 
	Note that we include values of $\Speed$ corresponding to $0$ and the roadblocks
	even if they do not belong to the essential range. 
	These latter values play a special role because 
	there are infinitely many particles at $0$, and each of the locations
	$b\in\RoadblockSet$ contains at least one particle with nonzero probability. 
	Let also
	\begin{equation}
		\SpeedEssRangeCirc_x:=\mathsf{EssRange}\{\Speed(y)\colon 0< y <x\},
		\qquad 
		\mathcal{W}_x^\circ:=\min\SpeedEssRangeCirc_x.
		\label{essential_range_of_Speed_notation}
	\end{equation}
	Clearly, $\mathcal{W}_x\le \mathcal{W}_x^\circ$ for all $x$.
\end{definition}

\begin{definition}[Edge]
	\label{def:edge_of_the_limit_shape}
	Fix $\tau\ge0$, and let $x_e=x_e(\tau)\ge0$ be the unique solution to the equation
	\begin{equation}
		\tau=\int_{0}^{x_e}\frac{dy}{(1-q)\Speed(y)}.
		\label{edge_of_the_limit_shape_equation}
	\end{equation}
	This solution is well-defined since the integrand is positive and bounded
	away from zero. 
	Clearly, $x_e(0)=0$, and $x_e(\tau)$ increases with $\tau$. 
	We call $x_e$ the \emph{edge of the limit shape}.

	This name can be informally justified as follows. Instead of looking at the
	rightmost particle in our exponential jump model, consider the model with
	just one particle. Then $(1-q)\Speed(y)$ is the rate with which this particle
	decides to leave a location $y$, and $\frac1{(1-q)\Speed(y)}$ is the mean
	time this single particle spends at $y$. In the limit as $\RateLambda\to+\infty$
	(i.e., as the travel distance goes to zero), the integral in the right-hand
	side of \eqref{edge_of_the_limit_shape_equation} 
	represents the scaled time it takes to reach
	location $x_e$. Equating this time with $\tau$ determines the asymptotic
	location of this single particle. 

	Let us also denote for all $x>0$:
	\begin{equation}
		\tau_e(x):=
		\int_0^x\frac{dy}{(1-q)\Speed(y)};
		\label{edge_of_the_limit_shape_tau_of_x}
	\end{equation}
	this is the time at which the location $x$ becomes the edge.
\end{definition}

\begin{proposition}
	\label{statement:TW_root_exists_and_is_unique}
	Fix $\tau>0$. For any $x\in(0,x_e)$, the equation
	\begin{equation}
		\label{critical_points_second_equation}
		\tau w=\int_0^x \phi_2\left( \frac{w}{\Speed(y)} \right)dy
	\end{equation}
	on $w\in(0,\mathcal{W}_x^\circ)$
	has a unique root which we denote by $\upomega^\circ_{\tau,x}$.
	For a fixed $x$ the function $\tau\mapsto \upomega^\circ_{\tau,x}$
	is strictly increasing, and
	$\upomega^\circ_{\tau_e(x),x}=0$,
	$\lim_{\tau\to+\infty}\upomega^\circ_{\tau,x}=\mathcal{W}_x^\circ$.
	Moreover, for a fixed $\tau$ the function $x\mapsto \upomega^\circ_{\tau,x}$
	is strictly decreasing, and
	$\upomega^\circ_{\tau,x_e(\tau)}=0$.
\end{proposition}
Note that equation \eqref{critical_points_second_equation} 
can have other roots outside the interval $w\in(0,\mathcal{W}_x^\circ)$.
We prove \Cref{statement:TW_root_exists_and_is_unique} in 
\Cref{sub:limit_shape_formulas}.
We are now in a position to describe the limit shape:
\begin{definition}[Limit shape]
	\label{def:limit_shape_definition_H_tau_x}
	The limit shape $\LimitShape(\tau,x)$ for $(\tau,x)\in\mathbb{R}_{\ge0}^2$ 
	is defined as follows:
	\begin{equation}
		\label{limit_shape_definition_H_tau_x}
		\LimitShape(\tau,x)
		:=
		\begin{cases}
			+\infty,  & \textnormal{if $x=0$ and $\tau\ge0$};
               \\
			0,        & \textnormal{if $x\ge x_e(\tau)$};
               \\
			\displaystyle
			\tau\min\bigl(\upomega^\circ_{\tau,x},\mathcal{W}_x\bigr)
			-
			\int_{0}^x\phi_1\biggl( \frac{\min (\upomega^\circ_{\tau,x},\mathcal{W}_x )}{\Speed(y)} \biggl)dy
			,
			&\textnormal{if $\tau>0$ and $0<x<x_e(\tau)$}.
		\end{cases}
	\end{equation}
\end{definition}

From the very definition of $\LimitShape(\tau,x)$ it is possible
to deduce the following properties one naturally expects of a limiting
height function (see \Cref{sub:limit_shape_formulas}
for the proof of \Cref{statement:properties_of_limit_shape_after_definition}):

\begin{proposition}
	\label{statement:properties_of_limit_shape_after_definition}
	For any fixed $\tau>0$, the function $x\mapsto \LimitShape(\tau,x)$ of
	\Cref{def:limit_shape_definition_H_tau_x} is left continuous, decreasing for
	$x\in\mathbb{R}_{>0}$, strictly decreasing for $x\in(0,x_e(\tau))$, 
	and $\LimitShape(\tau,x_e(\tau))=0$.
\end{proposition}

The law of large numbers stating that $\LimitShape(\tau,x)$ of
\Cref{def:limit_shape_definition_H_tau_x} is indeed the limit of
the rescaled random height function
$\RateLambda^{-1}\HeightFunction_{\XProcess(\tau\RateLambda)}(x)$ as
$\RateLambda\to+\infty$ would follow from \Cref{thm:main_theorem_on_fluctuations}
below. Modulo this law of large numbers, we can check that the
limit shape satisfies the partial differential equation
\eqref{intro_PDE_height_function} explained above
via hydrodynamic-type arguments:

\begin{proof}[Proof of \Cref{thm:PDE} modulo \Cref{thm:main_theorem_on_fluctuations}]
	When there are no roadblocks, we have
	$\mathcal{W}_x=\mathcal{W}_x^\circ>\upomega^\circ_{\tau,x}$, and so the limit
	shape is given for $0<x<x_e(\tau)$ by 
	\begin{equation*}
		\LimitShape(\tau,x)=\tau
		\upomega^\circ_{\tau,x} -\int_0^x
		\phi_1\biggl(
			\frac{\upomega^\circ_{\tau,x}}{\Speed(y)}
		\biggr)dy. 
	\end{equation*}
	One can directly check by differentiating this expression 
	(see \Cref{statement:PDE_verification} for details on computations) 
	that this function satisfies \eqref{intro_PDE_height_function}, as desired.
\end{proof}

\begin{remark}
	\label{rmk:Gaussian_PDE_remark}
	When $\upomega^\circ_{\tau,x}>\mathcal{W}_x$, one can
	write down partial differential equations for
	$\LimitShape(\tau,x)$ different from \eqref{intro_PDE_height_function}
	using the explicit formula \eqref{limit_shape_definition_H_tau_x}. 
	Fix $x\in\mathbb{R}_{>0}$ and assume that there are no roadblocks at $x$. 
	In particular, this implies that $\mathcal{W}_x$ is constant in a neighborhood of $x$.
	Then in this neighborhood we have:
	\begin{equation*}
		\frac{\partial\LimitShape(\tau,x)}{\partial\tau}
		=
		\mathcal{W}_x,\qquad 
		\frac{\partial\LimitShape(\tau,x)}{\partial x}
		=
		-
		\phi_1
		\left( \frac{\mathcal{W}_x}{\Speed(x)} \right).
	\end{equation*}
	These equations are simpler than 
	\eqref{intro_PDE_height_function} and
	in particular imply that the speed of growth of $\LimitShape(\tau,x)$
	at $x$ is constant. Moreover, if $\Speed$ is constant in a neighborhood of $x$,
	then the limit shape is linear in this neighborhood
	(cf. the leftmost part of the limit shape in \Cref{fig:intro_limit_shape}).
\end{remark}

\subsection{Asymptotic fluctuations}
\label{sub:asymptotics_intro}

To formulate our main result
on fluctuations of the random height function we need 
to recall the standard notation
of limiting fluctuation distributions.
We define the \emph{Fredholm determinant} $\det(1+K)$ corresponding to a
kernel $K(z,w)$ on a certain contour $\boldsymbol\gamma$ in the complex plane
via the expansion
\begin{equation}\label{Fredholm_expansion}
	\det(1+K)
	=
	1
	+
	\sum_{M=1}^{\infty}
	\frac{1}{M!}
	\int\limits_{\boldsymbol\gamma}\frac{dz_1}{2\pi\mathbf{i}}
	\ldots\int\limits_{\boldsymbol\gamma}\frac{dz_M}{2\pi\mathbf{i}}
	\,\mathop{\mathrm{det}}\limits_{i,j=1}^{M}\bigl[K(z_i,z_j)\bigr].
\end{equation}
One may regard \eqref{Fredholm_expansion} as a formal series, but we will be
interested in situations when it converges numerically.
In particular, this happens when $K$ is trace class.
We refer to \cite{Bornemann_Fredholm2010} for a detailed discussion.

\begin{definition}
	\label{def:fluctuation_distributions}
	\begin{enumerate}[\quad\bf1.]
	\item 
	The GUE Tracy--Widom distribution 
	function
	\cite{tracy_widom1994level_airy}
	is defined as 
	\begin{equation*}
		F_2(r):=\det\left( 1-\mathsf{K}_{\mathrm{Ai}} \right)_{L^2(r,+\infty)},
	\end{equation*}
	where 
	$\mathsf{K}_{\mathrm{Ai}}(v,v')$ is the Airy kernel:
	\begin{equation*}
		\mathsf{K}_{\mathrm{Ai}}(v,v')
		=
		\frac{1}{(2\pi\mathbf{i})^2}
		\int_{e^{-2\pi\mathbf{i}/3}\infty}^{e^{2\pi\mathbf{i}/3}\infty}dw
		\int_{e^{-\pi\mathbf{i}/3}\infty}^{e^{\pi\mathbf{i}/3}\infty}dz\,
		\frac{1}{z-w}
		\exp\left\{ \frac{z^3}{3}-\frac{w^3}{3}-zv+wv' \right\},
	\end{equation*}
	where the integration contours do not intersect.

	\item The Baik--Ben Arous--P\'ech\'e (BBP)
	distribution function introduced in \cite{BBP2005phase}
	is defined for any $m\in\mathbb{Z}_{\ge1}$ and $\mathbf{b}=(b_1,\ldots,b_m )\in\mathbb{R}^m$
	as
	\begin{equation*}
		F_{\mathrm{BBP},m,\mathbf{b}}(r):=
		\det\left( 1-\mathsf{K}_{\mathrm{BBP},m,\mathbf{b}} \right)_{L^2(r,+\infty)},
	\end{equation*}
	where the kernel has the form
	\begin{equation*}
		\mathsf{K}_{\mathrm{BBP},m,\mathbf{b}}(v,v')
		=
		\frac{1}{(2\pi\mathbf{i})^2}
		\int_{e^{-2\pi\mathbf{i}/3}\infty}^{e^{2\pi\mathbf{i}/3}\infty}dw
		\int_{e^{-\pi\mathbf{i}/3}\infty}^{e^{\pi\mathbf{i}/3}\infty}dz\,
		\frac{1}{z-w}
		\exp\left\{ \frac{z^3}{3}-\frac{w^3}{3}-zv+wv' \right\}
		\prod_{i=1}^{m}\frac{z-b_i}{w-b_i}.
	\end{equation*}
	Here the integration contours do not intersect and pass to the right of 
	the points $b_1,\ldots,b_m $. When $m=0$, the BBP distribution 
	coincides with the GUE Tracy--Widom distribution.
	
	\item Let $G_m(r)$ be the distribution function of 
	the largest eigenvalue of a $m\times m$ GUE random matrix
	$H=[H_{ij}]_{i,j=1}^{m}$, $H^*=H$, 
	$\Re H_{ij}\sim \mathcal{N}\bigl(0,\frac{1+\mathbf{1}_{i=j}}{2}\bigr)$, $i\ge j$, 
	$\Im H_{ij}\sim \mathcal{N}\bigl(0,\frac{1}{2}\bigr)$, $i>j$.
	When $m=1$, this is the standard Gaussian distribution.
	\end{enumerate}
\end{definition}

\begin{definition}[Phases of the limit shape]
	\label{def:phases_of_the_limit_shape_TW_Gaussian}
	Depending on the cases coming from taking the minimum in 
	\eqref{limit_shape_definition_H_tau_x}, we say that a point $(\tau,x)$
	belongs to one of the phases
	according to the following table:
	\begin{equation*}
		\begin{tabular}{l|l}
			$(\tau,x)$ is in the
			Tracy--Widom phase& $\upomega^\circ_{\tau,x}<\mathcal{W}_x$
			\\\hline
			$(\tau,x)$ is a 
			transition point&$\upomega^\circ_{\tau,x}=\mathcal{W}_x$
			\\\hline
			$(\tau,x)$ is in the 
			Gaussian phase&$\upomega^\circ_{\tau,x}>\mathcal{W}_x$
		\end{tabular}
	\end{equation*}
	Note that the root $\upomega^\circ_{\tau,x}\in(0,\mathcal{W}_x^\circ)$ 
	afforded by \Cref{statement:TW_root_exists_and_is_unique}
	does not depend on which phase the point $(\tau,x)$ is in,
	and also does not depend on the 
	roadblocks or the value $\Speed(0)$.
	On the other hand, the definition of 
	$\mathcal{W}_x$ in \eqref{range_of_Speed_notation}
	includes the values of $\Speed$ at $0$ and at all 
	roadblocks $b\in\RoadblockSet$, $b<x$.

	If $(\tau,x)$ is a transition point or is in the Gaussian phase,
	denote
	\begin{equation}
		\label{m_x_multiplicity_definition_intro}
		m_x
		:=
		\#
		\Bigl\{ 
			y\in\left\{ 0 \right\}\cup\left\{ b\in\RoadblockSet\colon 0<b<x \right\} 
			\colon \Speed(y)=\mathcal{W}_x
		\Bigr\}.
	\end{equation}
\end{definition}

\begin{remark}
	\label{rmk:stays_in_Gaussian_phase_forever}
	\Cref{statement:TW_root_exists_and_is_unique} implies that if a point $(\tau,x)$ is in
	the Gaussian phase, then it ``stays there forever'': for any $\tau'>\tau$ the
	point $(\tau',x)$ is in the Gaussian phase as well.
\end{remark}

\begin{theorem}[Asymptotic fluctuations]
	\label{thm:main_theorem_on_fluctuations}
	\begin{enumerate}[\quad\bf1.]
		\item If $(\tau,x)$ is in the Tracy--Widom phase, then
			\begin{equation*}
				\lim_{\RateLambda\to+\infty}
				\mathbb{P}
				\left( 
					\frac{\HeightFunction_{\XProcess(\tau\RateLambda)}(x)-
					\RateLambda \LimitShape(\tau,x)}
					{
						\RateLambda^{\frac{1}{3}}
						\upomega^\circ_{\tau,x}d^{\mathrm{TW}}_{\tau,x}
					}
					\ge -r
				\right)=F_2(r),\qquad r\in\mathbb{R},
			\end{equation*}
			where $d^{\mathrm{TW}}_{\tau,x}$ depends on the parameters of the model as in
			\eqref{dispersion_TW}.
		\item If $(\tau,x)$ is a transition point, then
			\begin{equation*}
				\lim_{\RateLambda\to+\infty}
				\mathbb{P}
				\left( 
					\frac{\HeightFunction_{\XProcess(\tau\RateLambda)}(x)-
					\RateLambda \LimitShape(\tau,x)}
					{
						\RateLambda^{\frac{1}{3}}
						\upomega^\circ_{\tau,x}d^{\mathrm{TW}}_{\tau,x}
					}
					\ge -r
				\right)=F_{\mathrm{BBP},m_x,\mathbf{b}}(r),
				\qquad r\in\mathbb{R},
			\end{equation*}
			where $m_x$ is given in \eqref{m_x_multiplicity_definition_intro},
			$\mathbf{b}=(0,0,\ldots ,0)$,
			and $d^{\mathrm{TW}}_{\tau,x}$ is given by \eqref{dispersion_TW}.
		\item If $(\tau,x)$ is in the Gaussian phase, then
			\begin{equation*}
				\lim_{\RateLambda\to+\infty}
				\mathbb{P}
				\left( 
					\frac{\HeightFunction_{\XProcess(\tau\RateLambda)}(x)-
					\RateLambda \LimitShape(\tau,x)}
					{
						\RateLambda^{\frac{1}{2}}
						\mathcal{W}_x d^{\mathrm{G}}_{\tau,x}
					}
					\ge -r
				\right)
				=
				G_{m_x}(r),\qquad r\in\mathbb{R},
			\end{equation*}
			where $m_x$ is given in \eqref{m_x_multiplicity_definition_intro}
			and $d^{\mathrm{G}}_{\tau,x}$ is given by \eqref{dispersion_G}.
	\end{enumerate}
\end{theorem}
\begin{remark}
	\label{rmk:more_BBP}
	By changing $x$, $\tau$, and $\Speed$ on scales $\RateLambda^{-\frac{1}{3}}$ and
	$\RateLambda^{-\frac{2}{3}}$ one can obtain different (in particular, nonzero)
	parameters $\mathbf{b}$
	in the BBP distribution in the second part of \Cref{thm:main_theorem_on_fluctuations},
	but for simplicity we will not discuss this.
\end{remark}

\subsection{Traffic jams}
\label{sub:phase_transitions_short_discussion}

\Cref{thm:main_theorem_on_fluctuations} shows that
points where $\upomega^\circ_{\tau,x}=\mathcal{W}_x$
correspond to phase transitions in the fluctuation exponents and
the fluctuation behavior.
Note that the limit shape $\LimitShape(\tau,x)$ is continuous (in $x$) at these
points. 
On the other hand, the presence of spatial inhomogeneity
(coming from changes in the speed function $\Speed$
as well as from roadblocks)
makes it possible for $\LimitShape(\tau,x)$ to become discontinuous.
We will call such discontinuity points the \emph{traffic jams}
as they correspond to macroscopic buildup of
particles. 
An example of a traffic jam is given in \Cref{fig:intro_limit_shape}.

Let us discuss two mechanisms for creating traffic jams.
Fix $\tau>0$ and $\sigma\in(0,x_e(\tau))$ such that
there are no roadblocks at $\sigma$ and, moreover, 
$\Speed(x)$ is continuous at $x=\sigma$.
Then $\LimitShape(\tau,x)$ is also continuous at $x=\sigma$.
A traffic jam at $\sigma$ can be created by either:
\begin{enumerate}[\quad$\bullet$]
	\item Inserting a new roadblock at $\sigma$, i.e., 
		modifying $\RoadblockSet^{\vee}:=\RoadblockSet\cup\left\{ \sigma \right\}$,
		taking
		\begin{equation*}
			\Speed^{\vee}(x)
			:=
			\Speed(x)\mathbf{1}_{x\ne\sigma}
			+
			\alpha\,
			\mathbf{1}_{x=\sigma},
			\qquad 
			0<\alpha<\upomega^\circ_{\tau,x},
		\end{equation*}
		and arbitrary $\RoadblockProb(\sigma)\in(0,1)$.
		Then the limit shape $\LimitShape^{\vee}$ of the modified model
		will have a traffic jam at $\sigma$ with Gaussian phase to the right of it.
	\item Inserting a slowdown in the speed function, i.e., 
		changing the values of $\Speed(x)$ on a whole interval $(\sigma,\sigma_1)$ to the right of $\sigma$:
		\begin{equation*}
			\Speed^{\vee}(x)
			:=
			\Speed(x)\mathbf{1}_{x\notin (\sigma,\sigma_1)}
			+\kappa\,
			\mathbf{1}_{x\in(\sigma,\sigma_1)},
			\qquad 
			0<\kappa<\upomega^\circ_{\tau,x}.
		\end{equation*}
		Then the modified limit shape $\LimitShape^{\vee}$ will have a traffic jam at $\sigma$
		and the Tracy--Widom phase to the right of it.
\end{enumerate}
Clearly, $\LimitShape^{\vee}(\tau,x)=\LimitShape(\tau,x)$ for all $x<\sigma$.
See \Cref{fig:LS_jam_examples} for examples.
Observe that if a roadblock does not lead to a traffic jam 
(i..e, if $\alpha\ge \upomega^\circ_{\tau,x}$), then it does not
change the limit shape at all.\footnote{%
	Though for $\alpha=\upomega^\circ_{\tau,x}$
	this roadblock changes the fluctuation distribution
	at $x=\sigma$ from the GUE Tracy--Widom to a BBP one.%
} 
On the other hand, if a slowdown (or a speedup) 
does not make the limit shape discontinuous
(i.e., if $\kappa\ge\upomega^\circ_{\tau,x}$)
then its derivatives at $x=\sigma$ may become discontinuous.

\begin{figure}[htbp]
	\begin{tabular}{ll}
		(a)&(b)\\
		\hspace{20pt}\includegraphics[width=.37\textwidth]{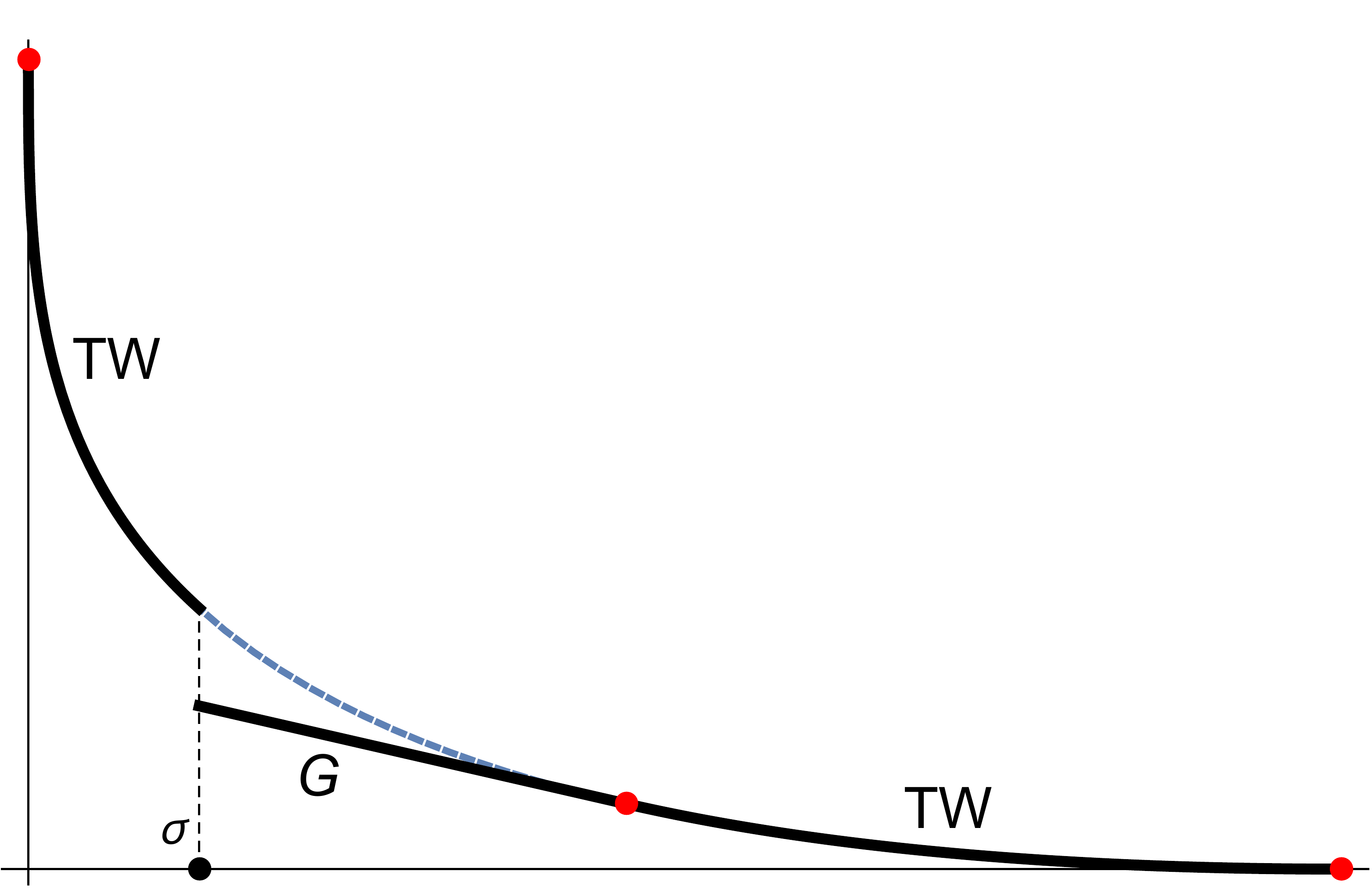}
		&
		\hspace{20pt}\includegraphics[width=.37\textwidth]{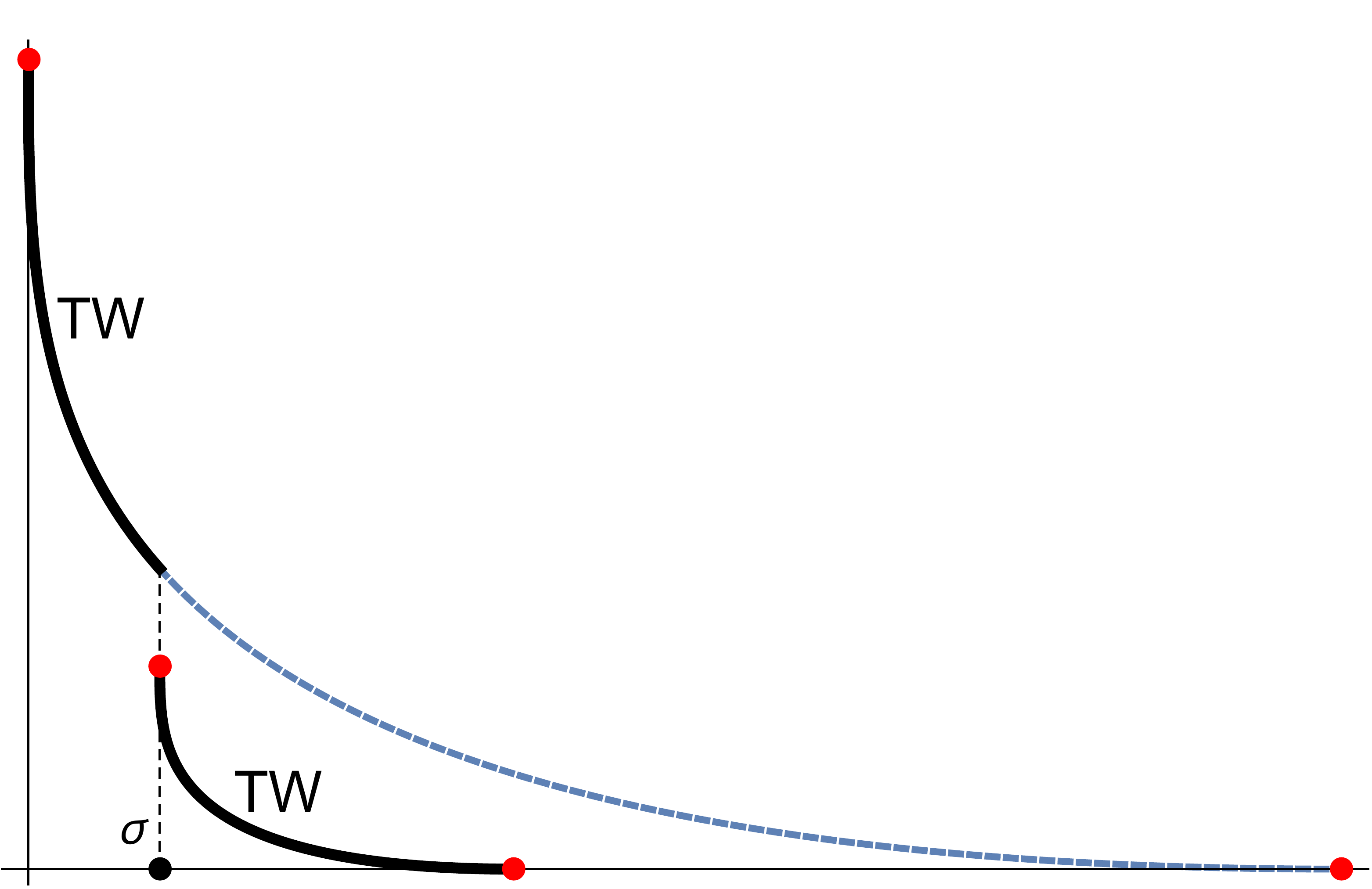}
		\\
		(c)&(d)\\
		\hspace{20pt}\includegraphics[width=.37\textwidth]{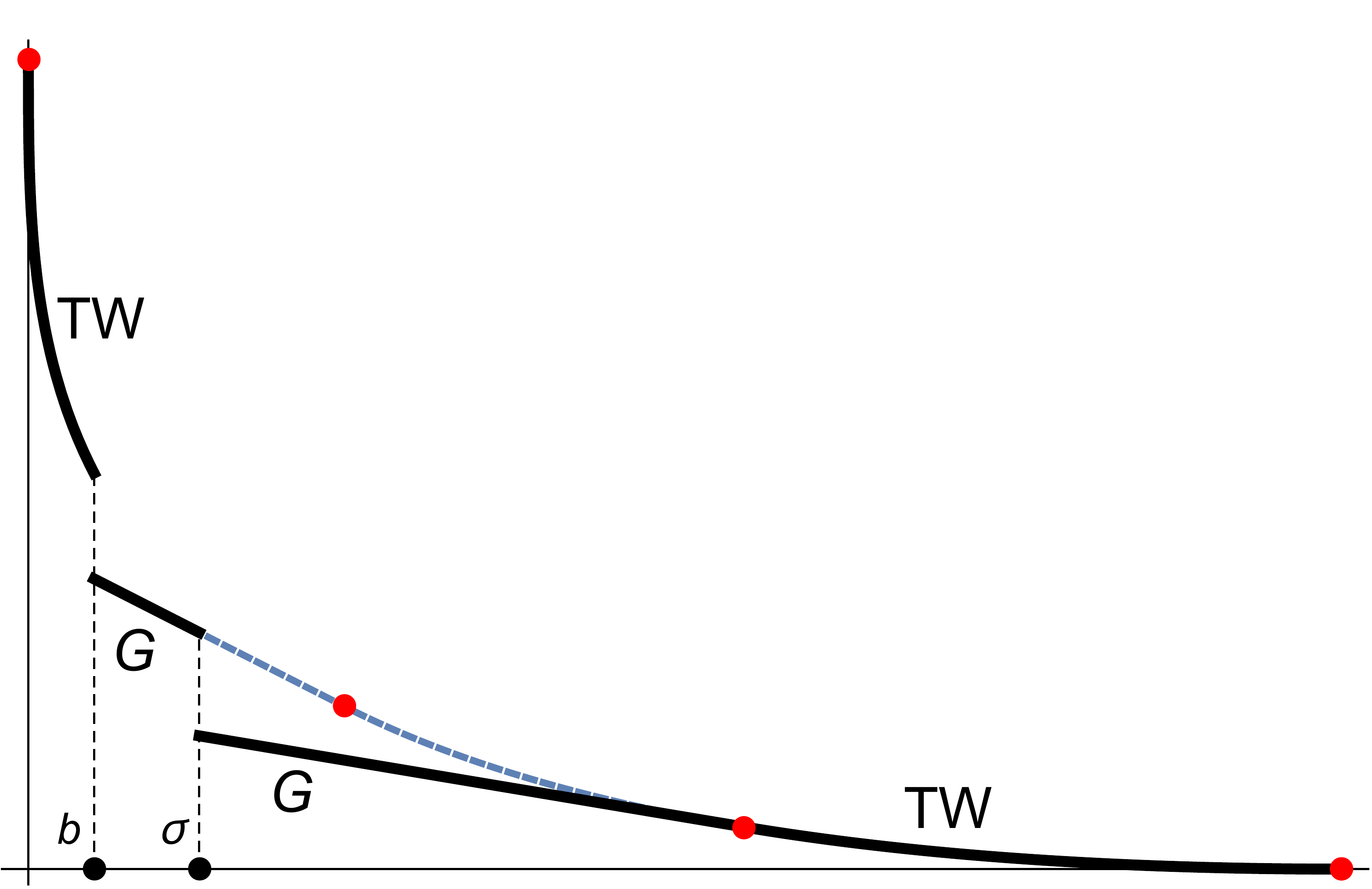}
		&
		\hspace{20pt}\includegraphics[width=.37\textwidth]{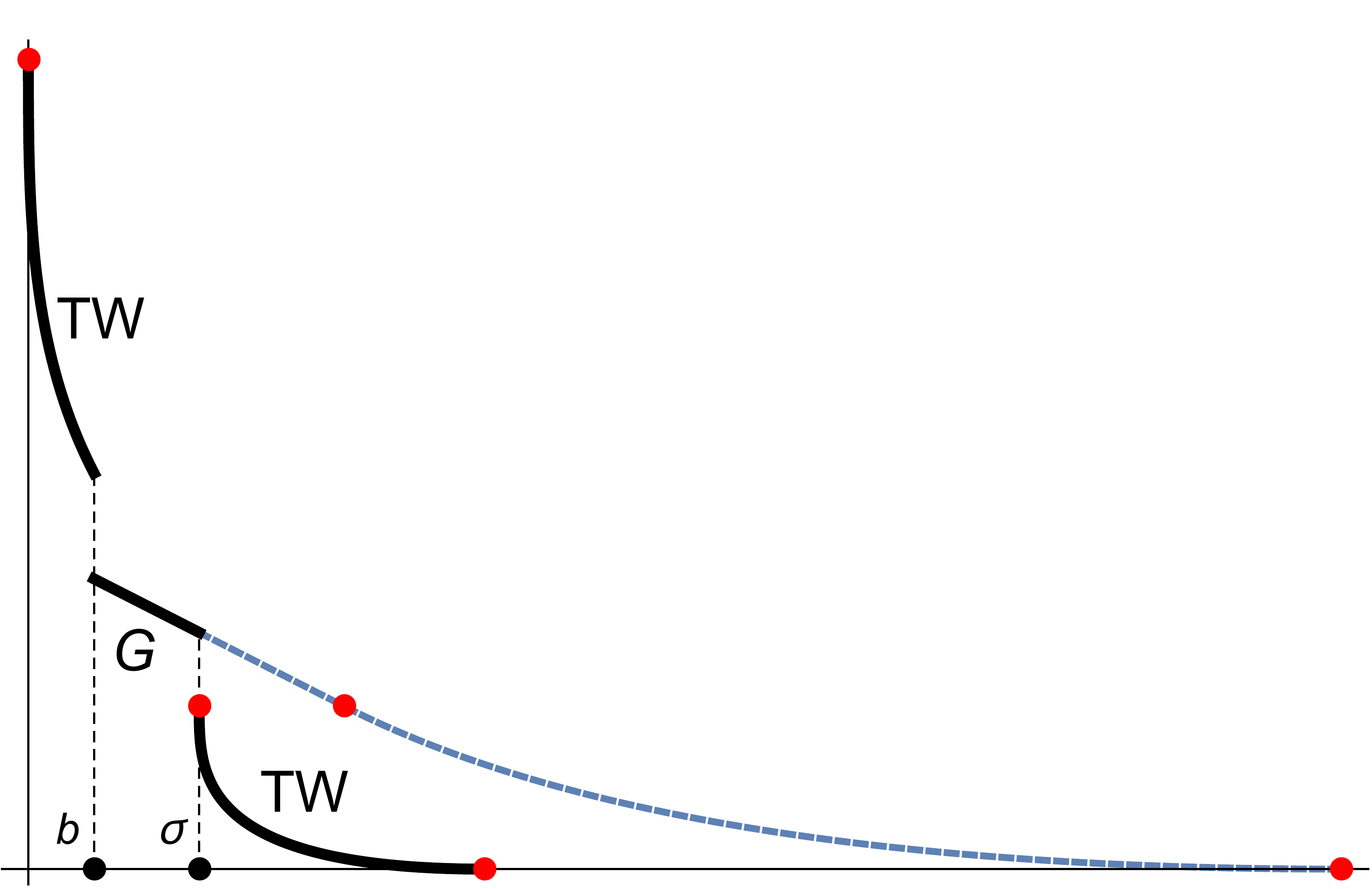}
	\end{tabular}
	\caption{
		Examples of traffic jams obtained by inserting a roadblock at $\sigma$
		(a,c) or a slowdown on $(\sigma,+\infty)$ (b,d). 
		There is an additional roadblock at $b$ in pictures (c,d) generating a
		Gaussian phase before the jam at $\sigma$. 
		The speed function $\Speed(x)$ in all examples is piecewise constant.
		Dashed and solid curves are limit shapes before and after creating a
		traffic jam at $\sigma$, respectively. 
		Tracy--Widom and Gaussian phases related to solid curves are also
		indicated.
		Solid red dots stand for points where a 
		curved shape is tangent to a straight line.%
	}
	\label{fig:LS_jam_examples}
\end{figure}

While 
computer experiments suggest that the fluctuations 
on both sides of a traffic jam (of any of the two above types)
are uncorrelated,
the microscopic behavior of particles 
is expected to be very different depending on the type of the jam.
Namely, if a jam at $\sigma$ is caused by a roadblock then there should
be one very large stack of particles located precisely at $\sigma$.
On the other hand, if a jam is caused by a slowdown then the buildup
of particles should happen to the right of $\sigma$ (but very close to it),
and locations of large stacks of particles there will be random.
A more detailed study of the behavior of the model
close to a traffic jam will be performed in a future publication.

\subsection{Pre-limit Fredholm determinant}
\label{sub:Fredholm_intro}

The starting point of our asymptotic analysis is the following Fredholm
determinantal formula for the $q$-Laplace transform of the height function
$\HeightFunction_{\XProcess(t)}$ of the exponential jump model (before the
$\RateLambda\to+\infty$ limit):

\begin{theorem}
	\label{thm:exponential_Fredholm}
	Fix $t>0$ and $x\in\mathbb{R}_{>0}$.
	We have for any $\zeta\in\mathbb{C}\setminus\mathbb{R}_{\ge0}$:\footnote{%
		Throughout the paper we use the $q$-Pochhammer symbol notation
		$(z;q)_{m}=\prod_{i=0}^{m-1}(1-zq^{i})$, $m\in\mathbb{Z}_{\ge0}$. 
		Since $0<q<1$ it makes sense for $m=+\infty$, too.%
	}
	\begin{equation}
		\label{Fredholm_for_exponential_model_section2_main_formula}
		\mathop{\mathbb{E}}
		\frac{1}{
		\bigl(
			\zeta q^{\HeightFunction_{\XProcess(t)}(x)};q
		\bigr)_{\infty}}
		=
		\det
		\bigl(
			1+K_{\zeta}
		\bigr)_{L^2(C_{a,\varphi})},
	\end{equation}
	where the contour $C_{a,\varphi}$ is given in
	\Cref{def:c_a_phi_contour_definition} below
	with
	$0<a<\mathcal{W}_x$ and
	$\varphi\in(0,\pi/2)$. 
	The kernel $K_\zeta$ in \eqref{Fredholm_for_exponential_model_section2_main_formula} is
	\begin{equation}
		\label{Fredholm_for_exponential_model_section2_main_formula_kernel}
		K_{\zeta}(w,w'):=\frac{1}{2\pi\mathbf{i}}\int_{D_w}
		\Gamma(-\mathsf{u})\Gamma(1+\mathsf{u})(-\zeta)^{\mathsf{u}}
		\frac{g(w)}{g(q^{\mathsf{u}}w)}\frac{d\mathsf{u}}{q^{\mathsf{u}}w-w'},
		\qquad 
		w,w'\in C_{a,\varphi},
	\end{equation}
	with the contour $D_w$ as in \Cref{def:d_w_contour_definition} below,
	and
	\begin{equation}
	\label{g_function_in_Fredholm_for_exponential_model_section2}
		g(w):=
		\frac{1}{(w/\Speed(0);q)_{\infty}}
		\prod_{b\in \RoadblockSet,\,b<x}
		\frac{(w \RoadblockProb(b)/\Speed(b);q)_\infty}
		{(w/\Speed(b);q)_\infty}
		\exp
		\biggl\{
			-tw
			+
			\RateLambda
			\int_0^x\phi_0
			\biggl(
				\frac{w}{\Speed(y)}
			\biggr)dy
		\biggr\}.
	\end{equation}
\end{theorem}
We prove \Cref{thm:exponential_Fredholm} in \Cref{sec:from_vertex_to_expon}.
The asymptotic analysis of the Fredholm determinantal formula 
\eqref{Fredholm_for_exponential_model_section2_main_formula} performed in
\Cref{sec:asymptotic_analysis} leads to our main result,
\Cref{thm:main_theorem_on_fluctuations}.

\section{From a vertex model to the exponential model}
\label{sec:from_vertex_to_expon}

In this section we explain how the inhomogeneous exponential jump model defined
in \Cref{sub:definition_of_the_model} arises as a degeneration of the
stochastic higher spin six vertex model studied in
\cite{Borodin2014vertex}, \cite{CorwinPetrov2015},
\cite{BorodinPetrov2016inhom}.
We also employ the $q$-moment formulas obtained in the latter paper to prove
our main pre-limit Fredholm determinantal formula
(\Cref{thm:exponential_Fredholm}).

\subsection{Stochastic higher spin six vertex model}
\label{sub:stochastic_higher_spin_definition}

For future convenience we describe the stochastic higher spin six vertex model
in the language of particle systems.
Moreover, to avoid unnecessary parameters and limit transitions we 
only focus on the time-homogeneous model with the step Bernoulli 
boundary condition (about the name see \Cref{rmk:on_boundary_conditions} below).
This model is a discrete time Markov chain 
$\{\DiscreteVertexXProcess(\mathfrak{t})\}_{\mathfrak{t}=0,1,\ldots }$
on the space of finite particle configurations 
$\mathrm{Conf}_{\bullet}(\mathbb{Z}_{\ge1})$
depending on the parameters
\begin{equation}
\label{stochastic_vertex_model_parameters}
	q\in(0,1);
	\qquad 
	u\in(0,+\infty);
	\qquad 
	\Speed(0)\in(0,+\infty); 
	\qquad 
	\upxi_i\in(0,+\infty), \quad \mathsf{s}_i\in(-1,0), \quad i\in\mathbb{Z}_{\ge1},
\end{equation}
such that $\upxi_i$ and $\mathsf{s}_i$ are uniformly bounded away from the 
endpoints of their corresponding intervals.
The parameter $\Speed(0)$ will eventually become the rate at which
new particles are added in the inhomogeneous exponential jump model,
so we can already use this notation here.

\begin{definition}[Stochastic higher spin six vertex model]
	\label{def:vertex_model}
	Under $\DiscreteVertexXProcess(\mathfrak{t})$
	during each step of the discrete time, the particle configuration
	$\eta=(\eta(1),\eta(2),\ldots )$ (where $\eta(i)\ge0$ is the number of
	particles at location $i\in\mathbb{Z}_{\ge1}$) is randomly updated to
	\begin{equation*}
		\eta'
		=
		(\eta'(1),\eta'(2),\eta'(3),\ldots )
		=
		\left( \eta(1)-h(1)+h(0),\eta(2)-h(2)+h(1),\eta(3)-h(3)+h(2),\ldots  \right),
	\end{equation*}
	where $h(0)\in\left\{ 0,1 \right\}$ is the number of new particles entering the system,
	and $h(i)\in\left\{ 0,1 \right\}$, $i\ge1$, is the 
	number of particles which moved from location $i$ to location $i+1$ during this time step 
	(so particles move only to the right).
	The update propagates from left to right and is governed by the following probabilities
	(where $i=1,2,\ldots $):
	\begin{align}
		\label{stochastic_weights_definition1}
		\mathbb{P}\left( h(0)=k \right)
		&=
		\frac{\Speed(0)u}{1+\Speed(0)u}\mathbf{1}_{k=1}
		+
		\frac{1}{1+\Speed(0)u}\mathbf{1}_{k=0};
		\\
		\label{stochastic_weights_definition2}
		\mathbb{P}\left( h(i)=k\mid h(i-1)=0 \right)
		&=
		\frac{-\upxi_i\mathsf{s}_iu(1-q^{\eta(i)})}{1-\upxi_i\mathsf{s}_iu}\mathbf{1}_{k=1}
		+
		\frac{1-\upxi_i\mathsf{s}_iu q^{\eta(i)}}{1-\upxi_i\mathsf{s}_iu}\mathbf{1}_{k=0}
		;
		\\
		\mathbb{P}\left( h(i)=k\mid h(i-1)=1 \right)
		&=
		\frac{\mathsf{s}_i^2q^{\eta(i)}-\upxi_i\mathsf{s}_iu}{1-\upxi_i\mathsf{s}_iu}\mathbf{1}_{k=1}
		+
		\frac{1-\mathsf{s}_i^2q^{\eta(i)}}{1-\upxi_i\mathsf{s}_iu}\mathbf{1}_{k=0}.
		\label{stochastic_weights_definition3}	
	\end{align}
	Because the configurations are finite
	and all weights in \eqref{stochastic_weights_definition2}--\eqref{stochastic_weights_definition3}
	are uniformly bounded away from $0$ and $1$, 
	almost surely there exists $M$ such that $h(j)=0$ for all $j>M$, 
	which means that the random update $\eta\mapsto\eta'$ eventually stops.
	See \Cref{fig:vertex_weights_stoch}.
\end{definition}

\begin{figure}[htbp]
	\def\epsx{.12}
	\begin{tabular}{l|c|c|c|c}
		probabilities 
		&
			 $\dfrac{1-\upxi_i \mathsf{s}_i u q^{\eta(i)}}{1- \upxi_i \mathsf{s}_iu}\rule{0pt}{22pt}$
     & $\dfrac{-\upxi_i \mathsf{s}_i u(1-q^{\eta(i)})}{1-\upxi_i \mathsf{s}_i u}$
     & $\dfrac{\mathsf{s}^{2}_i q^{\eta(i)}-\upxi_i \mathsf{s}_i u}{1-\upxi_i \mathsf{s}_iu}$
     & $\dfrac{1-\mathsf{s}^{2}_i q^{\eta(i)}}{1-\upxi_i \mathsf{s}_i u}$
			\phantom{\Bigg|}
		\\\hline
		\raisebox{20pt}{moves in $\DiscreteVertexXProcess(\mathfrak{t})$}
		&
		\begin{tikzpicture}
			[scale=1,very thick]
			\draw[->] (.3,0)--(1.7,0);
			\node [below] at (1,-.07) {$i$};
			\draw (1,-.1)--++(0,.2);
			\draw[fill] (1,.25) circle (2.5pt);
			\draw[fill] (1,.47) circle (2.5pt);
			\draw[fill] (1,.69) circle (2.5pt);
			\node[anchor=west, rotate=90] at (0,-.7) {\small$h(i-1)=0$};
			\node[anchor=west, rotate=90] at (2,-.7) {\ \small$h(i)=0$};
		\end{tikzpicture}
&
		\begin{tikzpicture}
			[scale=1,very thick]
			\draw[->] (.3,0)--(1.7,0);
			\node [below] at (1,-.07) {$i$};
			\draw (1,-.1)--++(0,.2);
			\draw[fill] (1,.25) circle (2.5pt);
			\draw[fill] (1,.47) circle (2.5pt);
			\draw[] (1,.69) circle (2.5pt) node (fl) {};
			\draw [->] (fl) to [out=70,in=160] (1.8,.9);
			\node[anchor=west, rotate=90] at (0,-.7) {\small$h(i-1)=0$};
			\node[anchor=west, rotate=90] at (2,-.7) {\ \small$h(i)=1$};
		\end{tikzpicture}
&
		\begin{tikzpicture}
			[scale=1,very thick]
			\draw[->] (.3,0)--(1.7,0);
			\node [below] at (1,-.07) {$i$};
			\draw (1,-.1)--++(0,.2);
			\draw[fill] (1,.25) circle (2.5pt);
			\draw[fill] (1,.47) circle (2.5pt);
			\draw[fill] (1,.69) circle (2.5pt) node (fl) {};
			\draw [->] (fl) to [out=70,in=160] (1.8,.9);
			\draw [->] (0.2,.9) to [out=20,in=110] (fl);
			\node[anchor=west, rotate=90] at (0,-.7) {\small$h(i-1)=1$};
			\node[anchor=west, rotate=90] at (2,-.7) {\ \small$h(i)=1$};
		\end{tikzpicture}
&
		\begin{tikzpicture}
			[scale=1,very thick]
			\draw[->] (.3,0)--(1.7,0);
			\node [below] at (1,-.07) {$i$};
			\draw (1,-.1)--++(0,.2);
			\draw[fill] (1,.25) circle (2.5pt);
			\draw[fill] (1,.47) circle (2.5pt);
			\draw[fill] (1,.69) circle (2.5pt);
			\draw (1,.91) circle (2.5pt) node (fl) {};
			\draw [->] (0.2,.98) to [out=40,in=130] (fl);
			\node[anchor=west, rotate=90] at (0,-.7) {\small$h(i-1)=1$};
			\node[anchor=west, rotate=90] at (2,-.7) {\ \small$h(i)=0$};
		\end{tikzpicture}
\\\hline
\raisebox{45pt}{\parbox{.15\textwidth}{vertex\\configurations}}
&
		{\scalebox{1}{\begin{tikzpicture}
		[scale=1, very thick]
		\node[anchor=north] (i1) at (0,-1) {$\eta(i)$};
		\node (j1) at (-1,0) {$0$};
		\node[anchor=south] (i2) at (0,1) {$\eta'(i)=\eta(i)$};
		\node (j2) at (1,0) {$0$};
		\draw[densely dotted] (j1) -- (j2);
		\foreach \shi in {(0,0), (\epsx,0), (-\epsx,0)}
		{\begin{scope}[shift=\shi]
			\node (shi1) at (0,-1) {};
			\node (shi2) at (0,1) {};
			\draw[->] (shi1) -- (shi2);
			\draw[->] (shi1) --++ (0,.5);
		\end{scope}}
		\end{tikzpicture}}}
&
		{\scalebox{1}{\begin{tikzpicture}
		[scale=1, very thick]
		\node[anchor=north] (i1) at (0,-1) {$\eta(i)$};
		\node (j1) at (-1,0) {$0$};
		\node[anchor=south] (i2) at (0,1) {$\eta'(i)=\eta(i)-1$};
		\node (j2) at (1,0) {$1$};
		\foreach \shi in {(0,0), (-\epsx,0)}
		{\begin{scope}[shift=\shi]
			\node (shi1) at (0,-1) {};
			\node (shi2) at (0,1) {};
			\draw[->] (shi1) -- (shi2);
			\draw[->] (shi1) --++ (0,.5);
		\end{scope}}
		\foreach \shi in {(\epsx,0)}
		{\begin{scope}[shift=\shi]
			\node (shi1) at (0,-1) {};
			\node (shi2) at (0,1) {};
			\draw[->] (shi1) -- (0,0) -- (j2);
			\draw[->] (shi1) --++ (0,.5);
		\end{scope}}
		\draw[densely dotted] (j1) -- (j2);
		\end{tikzpicture}}}
&
		{\scalebox{1}{\begin{tikzpicture}
		[scale=1, very thick]
		\node (i1shm1) at (-\epsx,-1) {};
		\node (i1sh1) at (\epsx,-1) {};
		\node (j1) at (-1,0) {$1$};
		\node[anchor=north] at (0,-1) {$\eta(i)$};
		\node (i1) at (0,-1) {};
		\node[anchor=south] at (0,1) {$\eta'(i)=\eta(i)$};
		\node (i2) at (0,1) {};
		\draw[densely dotted] (j1) -- (j2);
		\draw[densely dotted] (i1) -- (i2);
		\node (i2shm1) at (-\epsx,1) {};
		\node (i2sh1) at (\epsx,1) {};
		\node (j2) at (1,0) {$1$};
		\draw[->] (i1shm1) -- (-\epsx,-\epsx) -- (0,\epsx) -- (i2);
		\draw[->] (i1) -- (0,-\epsx) -- (\epsx,\epsx) -- (i2sh1);
		\draw[->] (i1) --++ (0,.5);
		\draw[->] (j1) -- (-\epsx*1.5,0)--++(.5*\epsx,.5*\epsx) -- (i2shm1);
		\draw[->] (i1sh1) -- (\epsx,-.5*\epsx)--++(.5*\epsx,.5*\epsx) -- (j2);
		\draw[->] (j1) --++ (.5,0);
		\draw[->] (i1sh1) --++ (0,.5);
		\draw[->] (i1shm1) --++ (0,.5);
		\end{tikzpicture}}}
&
		{\scalebox{1}{\begin{tikzpicture}
		[scale=1, very thick]
		\node[anchor=north] (i1) at (0,-1) {$\eta(i)$};
		\node (j1) at (-1,0) {$1$};
		\node[anchor=south] (i2) at (0,1) {$\eta'(i)=\eta(i)+1$};
		\node (i2shm2) at (-3/2*\epsx,1) {};
		\node (j2) at (1,0) {$0$};
		\draw[densely dotted] (j1) -- (j2);
		\draw[->] (j1) -- (-3/2*\epsx,0) -- (i2shm2);
		\foreach \shi in {(1/2*\epsx,0), (3/2*\epsx,0), (-1/2*\epsx,0)}
		{\begin{scope}[shift=\shi]
			\node (shi1) at (0,-1) {};
			\node (shi2) at (0,1) {};
			\draw[->] (shi1) -- (shi2);
			\draw[->] (shi1) --++ (0,.5);
		\end{scope}}
		\draw[->] (j1) --++ (.5,0);
		\end{tikzpicture}}}
	\end{tabular}
	\caption{%
		Probabilities
		\eqref{stochastic_weights_definition2}--\eqref{stochastic_weights_definition3}
		of individual moves in the particle system
		$\DiscreteVertexXProcess(\mathfrak{t})$ and their interpretation in terms
		of vertex weights: vertical arrows correspond to particles and horizontal
		arrows --- to their moves during one step of the discrete time.
	}
	\label{fig:vertex_weights_stoch}
\end{figure}

Probabilities
\eqref{stochastic_weights_definition2}--\eqref{stochastic_weights_definition3}
have a very special property which justifies their definition: they satisfy (a
version of) the Yang--Baxter equation.
We do not reproduce it here and refer to \cite{Mangazeev2014},
\cite{Borodin2014vertex}, \cite{CorwinPetrov2015},
\cite{BorodinPetrov2016_Hom_Lectures}, \cite{BorodinPetrov2016inhom} for
details in the context of $U_q(\widehat{\mathfrak{sl}_2})$ vertex models, and
to \cite{baxter2007exactly} for a general background. 
The Yang--Baxter equation is a key tool used in \cite{BorodinPetrov2016inhom}
to compute averages of certain observables of the higher spin six vertex model
in a contour integral form which we recall in
\Cref{sub:q_moment_formulas_for_vertex_and_half-cont_vertex} below.

Setting $\mathsf{s}_i^2\equiv 1/q$ turns
$\DiscreteVertexXProcess(\mathfrak{t})$ into the stochastic six vertex model in
which at most one particle per location $i\in\mathbb{Z}_{\ge1}$ is allowed. 
If a new particle from $i-1$ decides to move on top of a particle already
present at $i$, then the particle at $i$ gets displaced to the right (i.e., if
$\eta(i)=1$ and $h(i-1)=1$, then $h(i)=1$ with probability $1$).
This model was introduced in \cite{GwaSpohn1992}, and its asymptotic behavior
was investigated in \cite{BCG6V}, \cite{AmolBorodin2016Phase},
\cite{Amol2016Stationary}, \cite{borodin2016stochastic_MM}. 

\subsection{Remark. Boundary conditions}
	\label{rmk:on_boundary_conditions}
	
	The step Bernoulli boundary condition
	for $\DiscreteVertexXProcess(\mathfrak{t})$
	corresponds to particles entering the system at location $1$
	independently at each time step with probability
	$\Speed(0)u/(1+\Speed(0)u)$, see \eqref{stochastic_weights_definition1}.
	The term 
	follows \cite{AmolBorodin2016Phase} where 
	this type of boundary condition for the stochastic six vertex model
	was connected with the step Bernoulli
	(also sometimes called half stationary) initial configuration for the ASEP
	\cite{TW_ASEP4}.
	Formulas for the $q$-moments are also available for the step boundary
	condition corresponding to $\mathbb{P}\left( h(0)=1 \right)=1$ in
	\eqref{stochastic_weights_definition1} (i.e., a new particle enters at
	location $1$ at every time step), see \cite{BorodinPetrov2016inhom} or
	\cite{OrrPetrov2016}.

	The step boundary condition on $\mathbb{Z}_{\ge1}$
	can be degenerated to the step Bernoulli one on $\mathbb{Z}_{\ge2}$
	in two ways (related via the fusion procedure \cite{KulishReshSkl1981yang},
	\cite{CorwinPetrov2015}). 
	One way involves sending $\mathsf{s}_1\to0$ and $\upxi_1\to+\infty$ such that
	$\alpha=-\upxi_1\mathsf{s}_1>0$ is fixed. 
	In this limit we get the system with the step Bernoulli boundary condition on
	$\mathbb{Z}_{\ge2}$, in which the role of $\Speed(0)$ is played by $\alpha$.
	
	Another way \cite[Section 6.6]{BorodinPetrov2016inhom}
	deals with time-homogeneous parameters $u_1,u_2,\ldots $ (with $u_{\mathfrak{t+1}}$
	used during time step $\mathfrak{t}\to\mathfrak{t}+1$). 
	Specializing the first $K$ of them into a geometric progression, formally
	putting $\upxi_1=\mathsf{s}_1$, and sending $K\to+\infty$ places infinitely
	many particles at location $1$. 
	Probabilities
	\eqref{stochastic_weights_definition2}--\eqref{stochastic_weights_definition3}
	with $\eta(1)=+\infty$ reduce to \eqref{stochastic_weights_definition1}
	with the role of $\Speed(0)$ played by $-\mathsf{s}_1^2$ (which is assumed positive),
	and this also leads to a system with the step Bernoulli boundary condition
	on $\mathbb{Z}_{\ge2}$.
	
	This second way suggests that in the process
	$\DiscreteVertexXProcess(\mathfrak{t})$ on $\mathbb{Z}_{\ge1}$ described in
	\Cref{sub:stochastic_higher_spin_definition} one can for convenience put an
	infinite stack of particles at location $0$, similarly to 
	the exponential model (see \Cref{sub:definition_of_the_model}).

\subsection{Half-continuous vertex model}
\label{sub:continuous_vertex_coordinate}

Let us take the limit as $u\searrow0$ and $\mathfrak{t}\to+\infty$ so that
$u\mathfrak{t}\to t\in\mathbb{R}_{\ge0}$. Here $t$ is the rescaled continuous
time. The expansions as $u\searrow0$ of probabilities
\eqref{stochastic_weights_definition1}--\eqref{stochastic_weights_definition3}
are
\begin{align}
	\label{stochastic_weights_definition_expansions1}
	\mathbb{P}\left( h(0)=k \right)
	&=
	\left( \Speed(0)u+O(u^2) \right)
	\mathbf{1}_{k=1}
	+
	\left( 1+O(u) \right)
	\mathbf{1}_{k=0};
	\\
	\label{stochastic_weights_definition_expansions2}
	\mathbb{P}\left( h(i)=k\mid h(i-1)=0 \right)
	&=
	\left( 
		-\upxi_i\mathsf{s}_iu(1-q^{\eta(i)})+O(u^2)
	\right)
	\mathbf{1}_{k=1}
	+
	\left( 1+O(u) \right)
	\mathbf{1}_{k=0}
	;
	\\
	\mathbb{P}\left( h(i)=k\mid h(i-1)=1 \right)
	&=
	\left( \mathsf{s}_i^2q^{\eta(i)}+O(u) \right)
	\mathbf{1}_{k=1}
	+
	\left( 1-\mathsf{s}_i^2q^{\eta(i)}+O(u) \right)
	\mathbf{1}_{k=0}.
	\label{stochastic_weights_definition_expansions3}	
\end{align}
These expansions imply that the $u\searrow0$ limit leads to the following
system:
\begin{definition}[Half-continuous vertex model]
	\label{def:half_continuous_system}
	The \emph{half-continuous stochastic higher spin six vertex model} (or the
	\emph{half-continuous vertex model}, for short) is a continuous time Markov
	process $\left\{ \VertexXProcess(t) \right\}_{t\ge0}$ on
	$\mathrm{Conf}_{\bullet}(\mathbb{Z}_{\ge1})$ defined as follows.
	Initially all locations $i\in\mathbb{Z}_{\ge1}$ are empty, and there 
	is an infinite stack of particles at location $0$.
	Almost surely at most one particle can ``wake up'' and start moving in an
	instance of continuous time. 
	Particles wake up by either leaving the infinite stack at $0$ at rate
	$\Speed(0)$, or by leaving a stack of $\eta(i)$ particles at some location 
	$i\in\mathbb{Z}_{\ge1}$ at rate $\upxi_i(-\mathsf{s}_i)(1-q^{\eta(i)})$.
	The wake up events at different locations are independent
	and have exponential waiting times.

	Every particle which wakes up at some time moment $t\ge0$ 
	then instantaneously jumps to the 
	right according to the following probability:
	\begin{multline}
		\label{jumping_distribution_of_moving_particle_in_discrete_system}
		\mathbb{P}
		\bigl( 
			\textnormal{the moving particle ends exactly at location
			$j>i$} 
			\mid 
			\textnormal{it started at 
			$i\in\left\{ 0 \right\}\cup\mathbb{Z}_{\ge1}$}
		\bigr)
		\\
		=
		\mathsf{s}_{i+1}^2\ldots \mathsf{s}_{j-1}^{2}
		q^{\eta_{\VertexXProcess(t)}(i+1)+\ldots +
		\eta_{\VertexXProcess(t)}(j-1)}
		\left( 1-\mathsf{s}_j^2
		q^{\eta_{\VertexXProcess(t)}(j)} \right),
	\end{multline}
	where $j>i$ is arbitrary and the quantities $\eta_{\VertexXProcess(t)}$
	correspond to the configuration of particles before the moving
	particle started its jump. 
	In words, to move past any location $k>i$
	the moving particle flips a coin with probability of success
	$\mathsf{s}_k^2 q^{\eta_{\VertexXProcess(t)}(k)}$.
	If the coin comes up a success, the particle continues to move, otherwise it
	stops at location $k$.

	The process $\VertexXProcess(t)$ depends on $q$, $\Speed(0)$,
	and the parameters $\upxi_i$, $\mathsf{s}_i$, where $i\ge1$.
\end{definition}

The fact that $\VertexXProcess(t)$ is a continuous time limit of
$\DiscreteVertexXProcess(\mathfrak{t})$ follows by a standard application of a
Poisson-type limit theorem, much like how a discrete time random 
walk on $\mathbb{Z}$ with jumps~$\in\{0,1\}$ and with small
probability of a jump by $1$ can be
approximated in the continuous time limit by a Poisson jump process. 
Indeed, this is because up to any time $\mathfrak{t}=\lfloor tu^{-1}\rfloor$
the total number of particles in $\DiscreteVertexXProcess(\mathfrak{t})$ can be
approximated by a Poisson random variable with parameter $\Speed(0)t$, and
conditioned on having a given number of particles the discrete time finite
particle system is approximated by the corresponding continuous time finite
particle system.

\begin{remark}
	A similar half-continuous limit of the 
	stochastic six vertex model (the case $\mathsf{s}_i^2\equiv1/q$)
	was considered recently in \cite{BorodinBufetovWheeler2016},
	\cite{Ghosal2017KPZ},
	see also
	\cite{BCG6V}.
\end{remark}

\subsection{Limit to continuous space}
\label{sub:limit_to_continuous_space}

Consider the scaling limit of discrete to continuous space,
$\mathbb{Z}_{\ge1}\ni i\mapsto i\varepsilon
\in \mathbb{R}_{>0}$ as $\varepsilon\searrow0$. 
Let us show that under this limit the 
process $\VertexXProcess(t)$ becomes the inhomogeneous exponential jump model
$\XProcess(t)$ described in \Cref{sub:definition_of_the_model}. 
Let us first describe the scaling of parameters of $\VertexXProcess(t)$
assuming that the parameters 
$\RateLambda$, $\Speed(x)$, 
$\RoadblockSet$, and $\RoadblockProb(b)$, $b\in\RoadblockSet$,
of $\XProcess(t)$ are given.
Denote 
\begin{equation}
	\label{scale_discrete_to_continuous_space1}
	\RoadblockSet^{\varepsilon}:=
	\left\{ 
		\lfloor \varepsilon^{-1}b \rfloor \colon b\in \RoadblockSet
	\right\}
	\subset\mathbb{Z}_{\ge1},
\end{equation}
and put
\begin{equation}
	\label{scale_discrete_to_continuous_space2}
	\begin{array}{lll}
		\mathsf{s}_i^2=e^{-\varepsilon \RateLambda}
		,&\qquad 
		-\upxi_i\mathsf{s}_i=\Speed(i\varepsilon)
		,&\qquad 
		\textnormal{for}\quad 
		i\in\mathbb{Z}_{\ge1}\setminus \RoadblockSet^{\varepsilon};
		\\[3pt]
		\mathsf{s}_i^2=\RoadblockProb(b)
		,&\qquad 
		-\upxi_i\mathsf{s}_i=\Speed(b)
		,&\qquad 
		\textnormal{for}\quad 
		i=\lfloor \varepsilon^{-1}b \rfloor \in\RoadblockSet^{\varepsilon}.
	\end{array}
\end{equation}
Denote this $\varepsilon$-dependent process by $\VertexXProcess^{\varepsilon}(t)$.
Because the rates at which new particles are added to both
$\VertexXProcess^{\varepsilon}(t)$ and $\XProcess(t)$ are the same, we can
couple these processes so that they have the same Poisson number of particles
at each time $t\ge0$.
This reduces the question of convergence of
$\{\VertexXProcess^{\varepsilon}(t)\}_{t\ge0}$ to $\{\XProcess(t)\}_{t\ge0}$
to finite particle systems.
Because $\Speed(x)$ is piecewise continuous with left and right limits, 
the rates at
which particles decide to start moving in $\VertexXProcess^{\varepsilon}$ 
are close to those in $\XProcess$.
To conclude the convergence of $\{\VertexXProcess^{\varepsilon}(t)\}_{t\ge0}$
to $\{\XProcess(t)\}_{t\ge0}$ it remains to observe that the limit as
$\varepsilon\searrow0$ of traveling probabilities
\eqref{jumping_distribution_of_moving_particle_in_discrete_system} gives rise to
\eqref{jumpmodel_traveling_particle}, which is straightforward.

Let us summarize the development of 
\Cref{sub:stochastic_higher_spin_definition,rmk:on_boundary_conditions,sub:continuous_vertex_coordinate,sub:limit_to_continuous_space}:

\begin{proposition}
	\label{statement:limit_of_processes_to_exponential_model}
	Under the sequence of limit transitions
	\begin{equation*}
		\big\{ \DiscreteVertexXProcess(\mathfrak{t}) \big\}_{\mathfrak{t}=0,1,\ldots }
		\xrightarrow[
			\parbox{.19\textwidth}{%
				speed up time as $\mathfrak{t}=\lfloor tu^{-1} \rfloor $,
				$u\searrow0$%
			}
		]{}
		\big\{ \VertexXProcess(t) \big\}_{t\ge0}
		\xrightarrow[
			\parbox{.34\textwidth}{%
				rescale
				parameters $\upxi_i$, $\mathsf{s}_i$, $i\ge1$,
				as in 
				\eqref{scale_discrete_to_continuous_space1}--\eqref{scale_discrete_to_continuous_space2},
				and rescale the space
				$\mathbb{Z}_{\ge1}$ to $\mathbb{R}_{>0}$
			}
		]{}
		\big\{ \XProcess(t) \big\}_{t\ge0}
	\end{equation*}
	the stochastic higher spin six vertex model with the step Bernoulli boundary
	condition described in \Cref{sub:stochastic_higher_spin_definition} converges
	to the inhomogeneous
	exponential jump model defined in \Cref{sub:definition_of_the_model}.
\end{proposition}

\subsection{$q$-moments for the half-continuous model}
\label{sub:q_moment_formulas_for_vertex_and_half-cont_vertex}

Here and in
\Cref{sub:rewriting_q_moments,sub:Fredholm_half_continuous,sub:proof_of_exponential_Fredholm} 
below we prove \Cref{thm:exponential_Fredholm}.
Except passing to the continuous space in the last step, the proof 
is quite similar to the treatment of $q$-moments and
$q$-Laplace transforms performed in \cite{BorodinCorwinFerrari2012},
\cite{FerrariVeto2013}, \cite{barraquand2015phase}. 
However, to make the present paper self-contained we discuss all the necessary
steps, which are as follows:
\begin{enumerate}[\quad\bf1.]
	\item First, in this subsection
		we recall a nested contour integral formula for the $q$-moments of the height 
		function of the stochastic higher spin six vertex model,
		and take a (straightforward) limit to the $q$-moments
		in the half-continuous model.
	\item In \Cref{sub:rewriting_q_moments} we rewrite these $q$-moments
		for the half-continuous model in terms of contour integrals
		over certain infinite contours (which will be convenient for 
		asymptotic analysis in \Cref{sec:asymptotic_analysis}).
	\item In \Cref{sub:Fredholm_half_continuous}
		we turn the $q$-moment formulas for the half-continuous
		model into a Fredholm determinantal formula for this model.
		This requires some technical work to justify choices of integration
		contours which will be optimal for asymptotics.
	\item Finally, in \Cref{sub:proof_of_exponential_Fredholm} we pass to the
		limit to the continuous space, and obtain a Fredholm determinantal
		formula for the $q$-Laplace transform of the height function in the
		inhomogeneous exponential jump model.
\end{enumerate}

Let the height function $\HeightFunction_{\VertexXProcess(t)}(k)$,
$k\in\mathbb{Z}_{\ge0}$, of the process on the discrete space be defined
similarly to \eqref{height_function_definition}.  We have
$\HeightFunction_{\VertexXProcess(t)}(0)=+\infty$, and
$\HeightFunction_{\VertexXProcess(t)}(M)=0$ for $M$ large enough.
The next proposition is our first step towards 
\Cref{thm:exponential_Fredholm}.

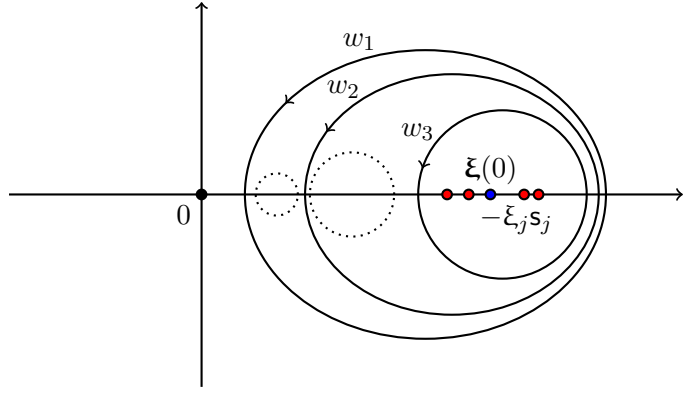
\begin{figure}[htpb]
	\centering
	\begin{tikzpicture}
		[scale=3.2,thick]
		\def\pt{0.02}
		\def\q{.5}
		\def\ss{.56}
		\draw[->, thick] (-.8,0) -- (2,0);
	  	\draw[->, thick] (0,-.8) -- (0,.8);
			\draw[fill=red] (1.34,0) circle (\pt) 
				node [below, xshift=-3pt] {$-\upxi_j\mathsf{s}_j$};
	  	\draw[fill=red] (1.11,0) circle (\pt);
			\draw[fill=red] (1.02,0) circle (\pt);
			\draw[fill=blue] (1.2,0) circle (1*\pt) 
			node [above,xshift=0] {$\Speed(0)$};
	  	\draw[fill=red] (1.4,0) circle (\pt);
	  	\draw[fill] (0,0) circle (\pt) node [below left] {$0$};
			\draw[thick,decoration={markings,mark=at position 0.45 with {\arrow{>}}},
			postaction={decorate}] (1.25,0) circle (.35) 
				node [below,xshift=-32,yshift=31] {$w_3$};
			\draw[dotted,thick] (1.25*\q,0) circle (.35*\q);
			\draw[dotted,thick] (1.25*\q*\q,0) circle (.35*\q*\q);
			\draw[thick,decoration={markings,mark=at position 0.42 with {\arrow{>}}},
			postaction={decorate}] (1.04,0) ellipse (.61 and .5) 
				node [below,xshift=-41,yshift=47] {$w_2$};
			\draw[thick,decoration={markings,mark=at position 0.4 with {\arrow{>}}},
			postaction={decorate}]
				(.93,0) ellipse (.75 and .6)
				node [below,xshift=-25,yshift=65] {$w_1$};
	\end{tikzpicture}
	\caption{
		A possible choice of nested integration
		contours in \eqref{nested_discrete_contours_q_moments} for 
		$\ell=3$. The contours $qw_3$ and $q^2w_3$ are shown dotted.
	}
	\label{fig:nested_discrete_contours_q_moments}
\end{figure}

\begin{proposition}
	\label{statement:first_moment_formula}
	Assume that the parameters of the 
	half-continuous model $\VertexXProcess(t)$ satisfy
	\begin{equation}
	\label{condition_inf_q_sup_on_discrete_parameters}
		\min\left( 
			\Speed(0),
			\inf_{j\ge1}\left\{ -\upxi_j\mathsf{s}_j \right\} 
		\right)
		>
		q\cdot 
		\min\left( 
			\Speed(0),
			\sup_{j\ge1}\left\{ -\upxi_j\mathsf{s}_j \right\} 
		\right).
	\end{equation}
	Then the $q$-moments of the height function of $\VertexXProcess(t)$
	at any $k\in\mathbb{Z}_{\ge1}$ are given by 
	\begin{multline}
		\label{nested_discrete_contours_q_moments}
		\mathop{\mathbb{E}}
		q^{\ell \HeightFunction_{\VertexXProcess(t)}(k)}
		=
		(-1)^{\ell}q^{\frac{\ell(\ell-1)}2}
		\oint \frac{dw_1   }{2\pi\mathbf{i}}\ldots 
		\oint \frac{dw_\ell}{2\pi\mathbf{i}}
		\prod_{1\le A<B\le \ell}\frac{w_A-w_B}{w_A-qw_B}
		\\
		\times
		\prod_{i=1}^{\ell}
		\left( 
			\frac{e^{(q-1)tw_i}}{w_i(1-w_i/\Speed(0))}\prod_{j=1}^{k-1}
			\frac{\upxi_j\mathsf{s}_j+\mathsf{s}_j^2w_i}
			{\upxi_j\mathsf{s}_j+w_i}
		\right),
	\end{multline}
	where $\ell=1,2,\ldots $, the integration contours 
	are positively oriented simple closed curves which
	encircle the poles $\Speed(0)$ and
	$-\upxi_j\mathsf{s}_j$, $j=1,2,\ldots $, but not $0$, 
	and the contour for $w_A$ contains the contour for $qw_B$ for all $A<B$
	(see \Cref{fig:nested_discrete_contours_q_moments}).
\end{proposition}
\begin{proof}
	We start with the formula of
	\cite[Corollary 10.3]{BorodinPetrov2016inhom} which
	gives $q$-moments of the height function of the 
	discrete time
	stochastic higher spin six vertex model
	with the step Bernoulli boundary condition on $\mathbb{Z}_{\ge2}$.
	However, since our process $\DiscreteVertexXProcess(\mathfrak{t})$
	lives on $\mathbb{Z}_{\ge1}$, we need 
	to shift the parameters by $1$.
	Thus, the formula yields for any $\ell\ge1$ and $k\ge0$:
	\begin{multline}
		\label{q_moments_old_step_Bernoulli_discrete}
		\mathop{\mathbb{E}}
		q^{\ell \HeightFunction_{\DiscreteVertexXProcess(\mathfrak{t})}(k)}
		=
		(-1)^{\ell}q^{\frac{\ell(\ell-1)}2}
		\oint \frac{dw_1   }{2\pi\mathbf{i}}\ldots 
		\oint \frac{dw_\ell}{2\pi\mathbf{i}}
		\prod_{1\le A<B\le \ell}\frac{w_A-w_B}{w_A-qw_B}
		\\
		\times
		\prod_{i=1}^{\ell}
		\left( 
			\frac{1}{w_i(1-w_i)}\prod_{j=0}^{k-1}
			\frac{\upxi_j-\mathsf{s}_jw_i}{\upxi_j-\mathsf{s}_j^{-1}w_i}
			\left( \frac{1-quw_i}{1-uw_i} \right)^{\mathfrak{t}}
		\right),
	\end{multline}
	where $\upxi_0\equiv \mathsf{s}_0$ 
	(and $-\mathsf{s}_0^2=\Speed(0)>0$), the integration
	contours are positively oriented and closed, 
	encircle $\left\{ \upxi_i\mathsf{s}_i \right\}_{i=0,1,\ldots }$,
	leave outside $0$, $1$ and $u^{-1}$, and the contour
	for $w_A$ contains the contour for $qw_B$ for all $A<B$.
	These contours exist thanks to 
	\eqref{condition_inf_q_sup_on_discrete_parameters}. 
	Formula \eqref{q_moments_old_step_Bernoulli_discrete}
	is obtained from a $q$-moment formula for the 
	step boundary condition \cite[Theorem 9.8]{BorodinPetrov2016inhom}
	via the second of the limit transitions mentioned in
	\Cref{rmk:on_boundary_conditions}.

	Let us now assume that $k\ge1$ since the case $k=0$ 
	will not be needed for our asymptotic analysis.
	When $k\ge1$ we get the following cancellation:
	\begin{equation*}
		\frac{1}{1-w}\frac{\upxi_0-\mathsf{s}_0w}{\upxi_0-\mathsf{s}_0^{-1}w}
		=
		\frac{\mathsf{s}_0^2}{\mathsf{s}_0^2-w}
		=
		\frac{1}{1+w/\Speed(0)}.
	\end{equation*}
	Let us change the variables as $w_i\to -w_i$.
	Then \eqref{q_moments_old_step_Bernoulli_discrete}
	turns into
	\begin{multline*}
		\mathop{\mathbb{E}}
		q^{\ell \HeightFunction_{\DiscreteVertexXProcess(\mathfrak{t})}(k)}
		=
		(-1)^{\ell}q^{\frac{\ell(\ell-1)}2}
		\oint \frac{dw_1   }{2\pi\mathbf{i}}\ldots 
		\oint \frac{dw_\ell}{2\pi\mathbf{i}}
		\prod_{1\le A<B\le \ell}\frac{w_A-w_B}{w_A-qw_B}
		\\
		\times
		\prod_{i=1}^{\ell}
		\left( 
			\frac{1}{w_i(1-w_i/\Speed(0))}\prod_{j=1}^{k-1}
			\frac{\upxi_j\mathsf{s}_j+\mathsf{s}_j^2w_i}
			{\upxi_j\mathsf{s}_j+w_i}
			\left( \frac{1+quw_i}{1+uw_i} \right)^{\mathfrak{t}}
		\right),
	\end{multline*}
	with the integration contours as in
	\Cref{fig:nested_discrete_contours_q_moments}.
	Putting $\mathfrak{t}=\lfloor tu^{-1} \rfloor $ and taking $u\searrow0$ (note
	that this is allowed by the integration contours) leads to the desired
	formula for the $q$-moments of the height function of 
	the half-continuous model $\VertexXProcess(t)$.
\end{proof}

\subsection{Rewriting $q$-moment formulas}
\label{sub:rewriting_q_moments}

Let us introduce new integration contours:
\begin{definition}
	\label{def:c_a_phi_contour_definition}
	Let $C_{a,\varphi}$,
	$a\in(0,+\infty)$, $\varphi\in(0,\pi/2)$, be the contour including two
	rays at the angles $\pm\varphi$ with the real line which meet at the point $a$,
	and such that the imaginary part decreases along this contour. Namely, 
	\begin{equation}
		\label{c_a_phi_contour_definition}
		C_{a,\varphi}:=\big\{ 
			a-\mathbf{i} y e^{\mathbf{i}\varphi\mathop{\mathrm{sgn}}(y)}\colon y\in\mathbb{R}
		\big\}.
	\end{equation}
	See \Cref{fig:infinite_contour} for an illustration.
\end{definition}

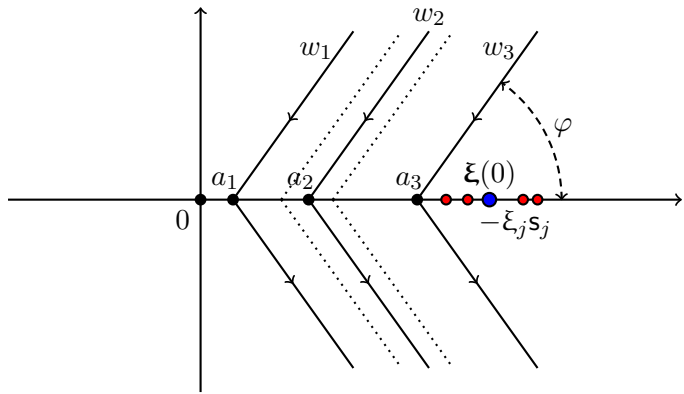
\begin{figure}[htbp]
	\centering
	\begin{tikzpicture}
		[scale=3.2,thick]
		\def\pt{0.02}
		\def\q{.61}
		\def\ss{.56}
		\draw[->] (-.8,0)--(2,0);
		\draw[->] (0,-.8)--(0,.8);
		\draw[fill=red] (1.34,0) circle (\pt) 
			node [below, xshift=-3pt] {$-\upxi_j\mathsf{s}_j$};
		\draw[fill=red] (1.11,0) circle (\pt);
		\draw[fill=red] (1.02,0) circle (\pt);
		\draw[fill=blue] (1.2,0) circle (1.4*\pt) 
		node [above,xshift=0] {$\Speed(0)$};
		\draw[fill=red] (1.4,0) circle (\pt);
		\draw[fill] (0,0) circle (\pt) node [below left] {$0$};
		\draw[decoration={markings,mark=at position 0.5 with {\arrow{<}}},
			postaction={decorate}] (.9,0)--++(.5,.7);
		\draw[decoration={markings,mark=at position 0.5 with {\arrow{>}}},
			postaction={decorate}] (.9,0)--++(.5,-.7);
		\draw[decoration={markings,mark=at position 0.5 with {\arrow{<}}},
			postaction={decorate}] (.9*\q-0.1,0)--++(.5,.7);
		\draw[decoration={markings,mark=at position 0.5 with {\arrow{>}}},
			postaction={decorate}] (.9*\q-0.1,0)--++(.5,-.7);
		\draw[decoration={markings,mark=at position 0.5 with {\arrow{>}}},
			postaction={decorate}] (.9*\q*\q-0.2,0)--++(.5,-.7);
		\draw[decoration={markings,mark=at position 0.5 with {\arrow{<}}},
			postaction={decorate}] (.9*\q*\q-0.2,0)--++(.5,.7);
		\draw[dotted] (.9*\q*\q,0)--++(.5,.7);
		\draw[dotted] (.9*\q*\q,0)--++(.5,-.7);
		\draw[dotted] (.9*\q,0)--++(.5,.7);
		\draw[dotted] (.9*\q,0)--++(.5,-.7);
		\draw[fill] (.9,0) circle (\pt) node [above, xshift=-3pt] {$a_3$};
		\draw[fill] (.9*\q-.1,0) circle (\pt) node [above, xshift=-3pt] {$a_2$};
		\draw[fill] (.9*\q*\q-.2,0) circle (\pt) node [above, xshift=-3pt] {$a_1$};
		\draw[densely dashed,
			decoration={
				markings,
				mark=at position 0.045 with {\arrow{<}},
				mark=at position 1 with {\arrow{>}}
			},
			postaction={decorate}
		] 
			(1.5,0) arc (0:54.5:.6);
		\node at (1.505,.3) {$\varphi$};
		\node at (.38+.1,.5+.12) {$w_1$};
		\node at (.95,.76) {$w_2$};
		\node at (1.14+.1,.5+.12) {$w_3$};
	\end{tikzpicture}
	\caption{%
		A possible choice of contours $C_{a_i,\varphi}$, $i=1,2,3$,
		in \Cref{statement:q_moments_nested_infinite_contours}
		for $\ell=3$. The contours
		$q C_{a_3,\varphi}$ and $q^2 C_{a_3,\varphi}$
		are shown dotted.
		The contour 
	}
	\label{fig:infinite_contour}
\end{figure}

\begin{proposition}
	\label{statement:q_moments_nested_infinite_contours}
	If the parameters of the half-continuous
	vertex model satisfy \eqref{condition_inf_q_sup_on_discrete_parameters},
	then the $q$-moments $\mathop{\mathbb{E}}q^{\ell
	\HeightFunction_{\VertexXProcess(t)}(k)}$, $k,\ell\ge1$, $t>0$, are given by the
	same formula as in the right-hand side of
	\eqref{nested_discrete_contours_q_moments},
	but with $w_j$ integrated over $C_{a_j,\varphi}$, where $\varphi\in(0,\pi/2)$
	and $a_1,\ldots,a_{\ell}$ are such that
	\begin{equation*}
		0<a_j<qa_{j+1}, \quad j=1,\ldots, \ell-1;
		\qquad 
		a_\ell<
		\min\left( 
			\Speed(0),
			\inf_{j\ge1}\left\{ -\upxi_j\mathsf{s}_j \right\} 
		\right).
	\end{equation*}
\end{proposition}
\begin{proof}
	This statement is established in the same way as \cite[Lemma
	4.10]{BorodinCorwinFerrari2012}, by expanding the closed contours for
	$w_1,\ldots, w_\ell$ in \eqref{nested_discrete_contours_q_moments} one by one
	to the right, and using the fact that the exponent $e^{(q-1)t\sum_{i}w_i}$ in
	the integrand makes the integrals over the far right parts of the contours
	negligible.
\end{proof}
\begin{proposition}
	\label{statement:q_moments_infinite_single_contour}
	The $q$-moments $\mathop{\mathbb{E}}q^{\ell
	\HeightFunction_{\VertexXProcess(t)}(k)}$, $k,\ell\ge1$, $t>0$, of the
	half-continuous vertex model can be rewritten in the following form with all
	integrals over one and the same contour $C_{a,\varphi}$ with
	$\varphi\in(0,\pi/2)$ and 
	$0<a<\min
	\bigl( 
		\Speed(0),
		\inf_{j\ge1}\left\{ -\upxi_j\mathsf{s}_j \right\} 
	\bigr)$:
	\begin{multline}
		\label{q_moments_infinite_single_contour}
		\mathop{\mathbb{E}}q^{\ell
		\HeightFunction_{\VertexXProcess(t)}(k)}
		=
		(q;q)_{\ell}
		\sum_{\substack{\mu\vdash \ell
		\\\mu=1^{m_1}2^{m_2}\ldots }}
		\frac{1}{m_1!m_2!\cdots }
		\frac{1}{(2\pi\mathbf{i})^{l(\mu)}}
		\int\ldots \int
		\mathop{\mathrm{det}}\limits_{i,j=1}^{l(\mu)}
		\left[ \frac{1}{w_iq^{\mu_i}-w_j} \right]
		\\\times
		\prod_{j=1}^{l(\mu)}f^{\mathrm{hc}}(w_j)f^{\mathrm{hc}}(qw_j)\ldots f^{\mathrm{hc}}(q^{\mu_j-1}w_j)dw_j,
	\end{multline}
	where the sum is over all partitions $\mu$ of $\ell$ 
	(i.e., $\ell=\sum_{i}\mu_i$) having 
	$m_1$ parts equal to $1$, $m_2$ parts equal to $2$, etc., 
	$l(\mu)$ denotes the number of nonzero parts in $\mu$, 
	and
	\begin{equation}
		\label{f_function_in_moments_definition}
		f^{\mathrm{hc}}(w):=\frac{e^{(q-1)tw}}{1-w/\Speed(0)}\prod_{j=1}^{k-1}
		\frac{\upxi_j\mathsf{s}_j+\mathsf{s}_j^2 w}{\upxi_j\mathsf{s}_j+w}.
	\end{equation}
\end{proposition}
\begin{proof}
	This is \cite[Proposition 3.2.1]{BorodinCorwin2011Macdonald}
	or 
	\cite[Proposition 7.4]{BorodinCorwinPetrovSasamoto2013}. 
\end{proof}

The $q$-moment formula in \Cref{statement:q_moments_infinite_single_contour}
implies that we can drop condition
\eqref{condition_inf_q_sup_on_discrete_parameters} on the parameters
$\Speed(0)$, $\upxi_j, \mathsf{s}_j$, $j\ge1$ (which was present in contour
	integral formulas with bounded contours coming from
\cite{BorodinPetrov2016inhom}).
From now on we only assume that $\Speed(0)>0$, $\upxi_j>0$ are uniformly
bounded away from $0$ and $+\infty$, and $\mathsf{s}_j\in(-1,0)$ are uniformly
bounded away from $-1$ and $0$, and with these assumptions the $q$-moment
formula \eqref{q_moments_infinite_single_contour} continues to hold.

\subsection{Fredholm determinantal formulas for the half-continuous model}
\label{sub:Fredholm_half_continuous}

Our first Fredholm determinantal formula follows by taking a generating function
of the $q$-moments given in \Cref{statement:q_moments_infinite_single_contour}:
\begin{proposition}
	\label{statement:first_Fredholm_formula}
	Fix $t>0$ and $k\in\mathbb{Z}_{\ge1}$. 
	For any
	$0<a<\min
	\bigl( 
		\Speed(0),
		\inf_{j\ge1}\left\{ -\upxi_j\mathsf{s}_j \right\} 
	\bigr)$,
	$\varphi\in(0,\pi/2)$, and $\zeta\in \mathbb{C}$ with sufficiently
	small $|\zeta|$ we have
	\begin{equation}
		\label{first_Fredholm_formula}
		\mathop{\mathbb{E}}
		\frac{1}{
		\bigl(
			\zeta q^{\HeightFunction_{\VertexXProcess(t)}(k)};q
		\bigr)_{\infty}}
		=
		\det
		\bigl(
			1+K_{\zeta}^{(1),\mathrm{hc}}
		\bigr)_{L^2(\mathbb{Z}_{>0}\times C_{a,\varphi})},
	\end{equation}
	where the kernel $K_{\zeta}^{(1),\mathrm{hc}}$
	is given by 
	\begin{equation}
	\label{first_Fredholm_formula_kernel}
	K_{\zeta}^{(1),\mathrm{hc}}(n_1,w_1;n_2,w_2)=
		\frac{\zeta^{n_1}f^{\mathrm{hc}}(w_1)f^{\mathrm{hc}}(qw_1)\cdots f^{\mathrm{hc}}(q^{n_1-1}w_1)}
	{q^{n_1}w_1-w_2},
	\end{equation}
	with $f^{\mathrm{hc}}(w)$ defined in \eqref{f_function_in_moments_definition}.
\end{proposition}
The Fredholm determinant of a kernel on $\mathbb{Z}_{>0}\times C_{a,\varphi}$
is defined similarly to \eqref{Fredholm_expansion}, but along with integrating
determinants of $K_{\zeta}^{(1),\mathrm{hc}}$ of sizes $M=1,2,\ldots $
in the continuous variables $w_1,\ldots, w_M\in C_{a,\varphi}$ we also sum them over the
discrete variables $n_1,\ldots, n_M\in\mathbb{Z}_{>0}$.
See also \eqref{first_Fredholm_formula_written_out} below.
\begin{proof}
	We first use the $q$-binomial theorem 
	\cite[(1.3.2)]{GasperRahman} to write
	\begin{equation*}
		\mathop{\mathbb{E}}\frac{1}{(\zeta q^{\HeightFunction_{\VertexXProcess(t)}(k)};q)_{\infty}}
		=
		\sum_{\ell=0}^{\infty}
		\frac{\zeta^{\ell}\mathop{\mathbb{E}}
		q^{\ell\HeightFunction_{\VertexXProcess(t)}(k)}}{(q;q)_{\ell}}.
	\end{equation*}
	Because $q^{\ell\HeightFunction_{\VertexXProcess(t)}(k)}\le1$, 
	the series converges for small enough
	$|\zeta|$, which justifies the interchange of the summation and the
	expectation.
	Using the formula of \Cref{statement:q_moments_infinite_single_contour} for 
	$\mathop{\mathbb{E}}q^{\ell\HeightFunction_{\VertexXProcess(t)}(k)}$
	we can reorganize the above summation 
	(see \cite[Proposition 3.2.8]{BorodinCorwin2011Macdonald} for details) 
	and write 
	\begin{multline}
	\label{Fredholm_first_expansion_need_to_justify_by_estimates}
		\mathop{\mathbb{E}}\frac{1}{(\zeta q^{\HeightFunction_{\VertexXProcess(t)}(k)};q)_{\infty}}
		=
		\sum_{M=0}^{\infty}
		\frac{1}{M!}
		\sum_{n_1,\ldots,n_M\in\mathbb{Z}_{>0}}
		\int_{C_{a,\varphi}}\ldots\int_{C_{a,\varphi}}
		\mathop{\mathrm{det}}\limits_{i,j=1}^{M}
		\left[ \frac{1}{q^{n_i}w_i-w_j} \right]
		\\
		\times
		\prod_{j=1}^{M}
		\zeta^{n_j}f^{\mathrm{hc}}(w_j)f^{\mathrm{hc}}(qw_j)
		\ldots f^{\mathrm{hc}}(q^{n_j-1}w_j)\frac{dw_j}{2\pi\mathbf{i}}.
	\end{multline}
	This coincides with the Fredholm expansion of 
	$
	\det
	\bigl(
		1+K_{\zeta}^{(1),\mathrm{hc}}
	\bigr)_{L^2(\mathbb{Z}_{>0}\times C_{a,\varphi})}
	$
	in the right-hand side of \eqref{first_Fredholm_formula}
	with the kernel given by \eqref{first_Fredholm_formula_kernel}.
	To finish the proof and show that identity
	\eqref{first_Fredholm_formula} holds numerically
	for small $|\zeta|$, we need to justify that the expansion in
	\eqref{Fredholm_first_expansion_need_to_justify_by_estimates}
	converges absolutely.
	First, observe that $f^{\mathrm{hc}}(q^nw)$ given by \eqref{f_function_in_moments_definition}
	is bounded on the contour $C_{a,\varphi}$ uniformly in $n\ge0$,
	and $f^{\mathrm{hc}}(w)$ decays exponentially 
	as $|w|$ grows. This implies that\footnote{%
		Here and below $B,B_1,B_2,\ldots $ and $c,c',c'',c_1,c_2,\ldots $ are positive
		constants which may depend on $q$, the parameters of the models, or the
		data in the formulation of the statements (such as the angle $\varphi$,
		etc.).%
	}
	\begin{equation}
	\label{Fredholm_first_expansion_need_to_justify_by_estimates_1}
		|f^{\mathrm{hc}}(w_j)f^{\mathrm{hc}}(qw_j)\ldots f^{\mathrm{hc}}(q^{n_j-1}w_j)|\le B_1^{n_j}e^{-c \Re(w_j)}.
	\end{equation}
	Next, observe that $\min_{w_1,w_2\in C_{a,\varphi}}|q^nw_1-w_2|=(1-q^n)a\sin\varphi
	\ge(1-q)a\sin \varphi$
	for any $n\in \mathbb{Z}_{>0}$, and so by the Hadamard's inequality
	we have
	\begin{equation*}
		\left|
			\mathop{\mathrm{det}}\limits_{i,j=1}^{M}
			\left[ \frac{1}{q^{n_i}w_i-w_j} \right]
		\right|
		\le
		B_2^{M}M^{M/2}.
	\end{equation*}
	Therefore, the sum of integrals of the absolute values in the right-hand side of 
	\eqref{Fredholm_first_expansion_need_to_justify_by_estimates}
	can be estimated as
	\begin{equation*}
		\sum_{M\ge0}\frac{B_2^MM^{M/2}}{M!}
		\left( \sum_{n\ge1}B_1^n\zeta^n B_3 \right)^{M},
	\end{equation*}
	where $B_3$ arises from integrating the exponent in 
	\eqref{Fredholm_first_expansion_need_to_justify_by_estimates_1}.
	The inside geometric series in $n$
	converges for sufficiently small $|\zeta|$,
	and the series in $M$ converges thanks to the factorial
	in the denominator. This completes the proof.
\end{proof}

Following \cite{BorodinCorwinFerrari2012}, 
let us define another contour which will play a role in 
the next Fredholm determinantal formula.
\begin{definition}
	\label{def:d_w_contour_definition}
	For $R,d>0$ let
	\begin{equation}
		\label{d_R_d_contour_definition}
		D_{R,d}:=
		\bigl(R-\mathbf{i}\infty, R-\mathbf{i}d\bigr]
		\cup
		\bigl(R-\mathbf{i}d,\tfrac{1}{2}-\mathbf{i}d\bigr]
		\cup
		\bigl(\tfrac{1}{2}-\mathbf{i}d,\tfrac{1}{2}+\mathbf{i}d\bigr]
		\cup
		\bigl(\tfrac{1}{2}+\mathbf{i}d,R+\mathbf{i}d\bigr]
		\cup
		\bigl(R+\mathbf{i}d,R+\mathbf{i}\infty\bigr),
	\end{equation}
	oriented so that the imaginary part does not decrease along $D_{R,d}$.
	Now, for every $w\in C_{a,\varphi}$, 
	where $a>0$, $\varphi\in(0,\pi/2)$, let us choose $R(w),d(w)$ such that:
	\begin{enumerate}[\quad$\bullet$]
		\item For any $\mathsf{u}\in D_{R(w),d(w)}$ we have
			$\arg(w(q^\mathsf{u}-1))\in(\pi/2+b,3\pi/2-b)$, where $b=\pi/4-\varphi/2$.
		\item For any $\mathsf{u}\in D_{R(w),d(w)}$ the point $q^{\mathsf{u}}w$ 
			stays to the left of the contour $C_{a,\varphi}$.
	\end{enumerate}
	Denote the resulting contour $D_{R(w),d(w)}$ simply by $D_w$.
	By \cite[Remark 4.9]{BorodinCorwinFerrari2012}, the contour $D_w$
	exists, and for large $|w|$ it suffices to take $d(w)\sim B_d|w|^{-1}$ and 
	$R(w)\sim B_R\log|w|$ for some constants $B_d,B_R>0$.
	See \Cref{fig:infinite_contour_D_definition} for an illustration.
\end{definition}

\begin{figure}[htbp]
	\centering
	\begin{tikzpicture}
		[scale=3.2, thick]
		\def\pt{0.02}
		\def\q{.61}
		\def\ss{.56}
		\draw[->] (-.1,0)--(1.7,0);
		\draw[->] (0,-.8)--(0,.8);
		\draw[
			very thick,
			decoration={
				markings,
				mark=at position 0.1 with {\arrow{>}},
				mark=at position 0.37 with {\arrow{>}},
				mark=at position 0.52 with {\arrow{>}},
				mark=at position 0.64 with {\arrow{>}},
				mark=at position 0.9 with {\arrow{>}},
			},
			postaction={decorate}
		] 
		(1,-.7)--(1,-.1)--(.2,-.1)--(.2,.1)--(1,.1)--(1,.7);
		\node at (1.17,.6) {$D_{R,d}$};
		\draw[<->] (1.2,-.1)--(1.2,.1) node [above, xshift=5pt] {$2d$};
		\draw[dotted] (1,.1)--++(.3,0);
		\draw[dotted] (1,-.1)--++(.3,0);
		\draw[<->] (0,-.6)--++(1,0) node [midway, yshift=-7pt] {$R$};
		\draw[<->] (0,-.3)--++(.2,0) node [midway, yshift=-10pt] {$\tfrac{1}{2}$};
		\draw[dotted] (.2,-.1)--(.2,-.4);
	\end{tikzpicture}
	\hspace{20pt}
		\begin{tikzpicture}
		[scale=3.2, thick]
		\def\pt{0.02}
		\def\q{.61}
		\def\ss{.56}
		\draw[->] (-.4,0)--(1.2,0);
		\draw[->] (0,-.8)--(0,.8);
		\draw[very thick] (0,0) circle (.2);
		\draw[very thick] (.8,0) arc (0:20:.8);
		\draw[very thick] (.8,0) arc (0:-20:.8);
		\draw[very thick] (0.187939, 0.068404)--(0.751754, 0.273616);
		\draw[very thick] (0.187939, -0.068404)--(0.751754, -0.273616);
		\draw[dotted] (.2,0) --++(0,-.4);
		\draw[<->] (0,-.3)--++(.2,0) node [midway, yshift=-10pt] {$q^R$};
		\draw[fill] (1,0) circle (.6pt) node[below right] {1};
		\draw[fill] (.8,0) circle (.6pt) node[below left] {$q^{\frac{1}{2}}$};
		\draw[dotted] (1,0) -- (0.751754, 0.273616) --++ ( 1.5*0.751754-1.5, 1.5*0.273616);
		\draw[dotted] (1,0) -- (0.751754, -0.273616) --++ ( 1.5*0.751754-1.5, -1.5*0.273616);
		\node at (-.4,.3) {$q^{\mathsf{u}}$, $\mathsf{u}\in D_{R,d}$};
		\draw[<->] (.87,0) arc (180:130:.13) node[above, anchor=west, yshift=5pt, xshift=-2pt] 
			{angle $<b=\frac{\pi}{4}-\frac{\varphi}{2}$};
	\end{tikzpicture}
	\caption{%
		Left: the contour $D_{R,d}$. 
		Right: the points $q^\mathsf{u}$ for $\mathsf{u}\in D_{R,d}$.
		The angle indicated in the picture should be less than $b=\pi/4-\varphi/2$ 
		in order to satisfy the condition
		$\arg(w(q^{\mathsf{u}}-1))\in(\pi/2+b,3\pi/2-b)$.%
	}
	\label{fig:infinite_contour_D_definition}
\end{figure}
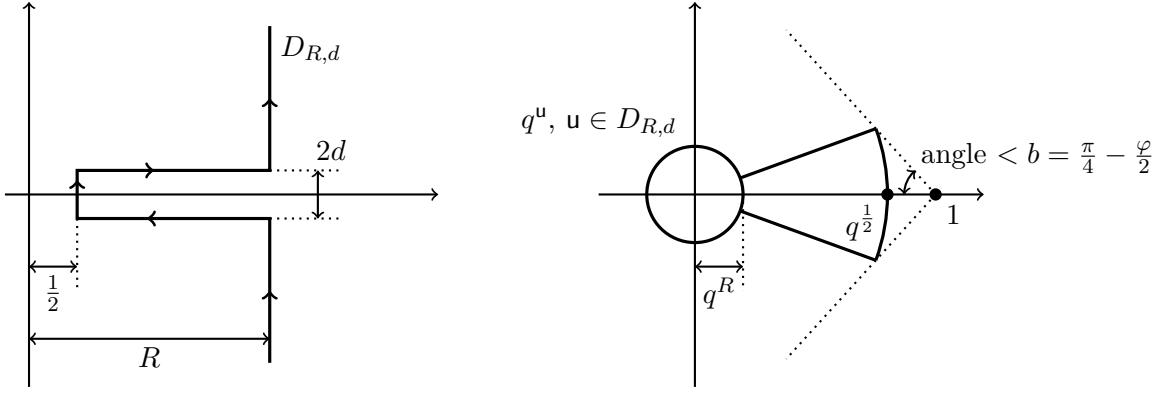

\begin{proposition}
	\label{statement:second_Fredholm_formula_}
	Fix $t>0$ and $k\in\mathbb{Z}_{\ge1}$.
	We have for any $\zeta\in\mathbb{C}\setminus\mathbb{R}_{\ge0}$:
	\begin{equation}
	\label{second_Fredholm_formula}
		\mathop{\mathbb{E}}
		\frac{1}{
		\bigl(
			\zeta q^{\HeightFunction_{\VertexXProcess(t)}(k)};q
		\bigr)_{\infty}}
		=
		\det
		\bigl(
			1+K_{\zeta}^{(2),\mathrm{hc}}
		\bigr)_{L^2(C_{a,\varphi})},
	\end{equation}
	where the contour $C_{a,\varphi}$ is given in
	\Cref{def:c_a_phi_contour_definition}
	with
	$0<a<\min
	\bigl( 
		\Speed(0),
		\inf_{j\ge1}\left\{ -\upxi_j\mathsf{s}_j \right\} 
	\bigr)$ and
	$\varphi\in(0,\pi/2)$. The kernel in \eqref{second_Fredholm_formula}
	is
	\begin{equation}
		\label{second_Fredholm_formula_kernel}
		K_{\zeta}^{(2),\mathrm{hc}}(w,w')=\frac{1}{2\pi\mathbf{i}}\int_{D_w}
		\Gamma(-\mathsf{u})\Gamma(1+\mathsf{u})(-\zeta)^{\mathsf{u}}
		\frac{g^{\mathrm{hc}}(w)}{g^{\mathrm{hc}}(q^{\mathsf{u}}w)}\frac{d\mathsf{u}}{q^{\mathsf{u}}w-w'},
	\end{equation}
	with the contour $D_w$ as in \Cref{def:d_w_contour_definition}
	and
	where $g^{\mathrm{hc}}(w)$ is expressed through $f^{\mathrm{hc}}(w)$ \eqref{f_function_in_moments_definition} via
	\begin{equation}
		\label{second_Fredholm_formula_g_function}
		g^{\mathrm{hc}}(w)
		:=
		f^{\mathrm{hc}}(w)f^{\mathrm{hc}}(qw)\ldots 
		= 
		\frac{e^{-tw}}{(w/\Speed(0);q)_{\infty}}
		\prod_{j=1}^{k-1}
		\frac{(-w\mathsf{s}_j\upxi_j^{-1};q)_{\infty}}{(-w\mathsf{s}_j^{-1}\upxi_j^{-1};q)_{\infty}}.
	\end{equation}
\end{proposition}
\begin{proof}
	\textbf{Step 1.}
	We first prove \eqref{second_Fredholm_formula} for small $|\zeta|$ by
	rewriting the previous formula of \Cref{statement:first_Fredholm_formula},
	and then will analytically continue in $\zeta$ using the properties of the
	contour $C_{a,\varphi}$. 
	To rewrite the sums coming from the $\mathbb{Z}_{>0}$
	part of the Fredholm determinant in 
	\eqref{first_Fredholm_formula}, we use the Mellin--Barnes summation formula
	(e.g., \cite[Lemma 7.1]{BorodinCorwinFerrari2012})
	which follows from the fact that
	$\mathop{\mathrm{Res}}_{\mathsf{u}=n}
	\Gamma(-\mathsf{u})\Gamma(1+\mathsf{u})=(-1)^{n+1}$:
	\begin{equation}
	\label{Mellin_barnes_summation}
		\sum_{n=1}^{\infty}F(q^n)\zeta^n
		=
		\frac{1}{2\pi\mathbf{i}}\int_{C_{1,2,\ldots }}
		\Gamma(-\mathsf{u})\Gamma(1+\mathsf{u})
		(-\zeta)^{\mathsf{u}}F(q^\mathsf{u})d\mathsf{u},
		\qquad 
		|\zeta|<1,
		\quad
		\zeta\notin\mathbb{R}_{\ge0},
	\end{equation}
	where the contour $C_{1,2,\ldots }$ encircles points $1,2,\ldots $ and winds
	around them in the negative direction, and encircles no other singularities
	of $F(q^{\mathsf{u}})$.
	For the above equality to hold, the series in the left-hand side must
	converge, and the integral in the right-hand side must be able to be
	approximated by integrals over a sequence of (negatively oriented) contours
	$C_\mathsf{k}$, $\mathsf{k}=1,2,\ldots $, such that: 
	\begin{enumerate}[\quad$\bullet$]
		\item Each contour $C_\mathsf{k}$ encircles $1,\ldots, \mathsf{k}$ and no other singularities
			of the integrand;
		\item The contours $C_\mathsf{k}$ partly coincide with $C_{1,2,\ldots }$;
		\item The integral over the symmetric difference of
			$C_\mathsf{k}$ and $C_{1,2,\ldots }$ goes to zero as $\mathsf{k}\to\infty$.
	\end{enumerate}
	
	The Fredholm determinant
	in the right-hand side of \eqref{first_Fredholm_formula} looks as
	\begin{multline}
		\label{first_Fredholm_formula_written_out}
		\det
		\bigl(
			1+K_{\zeta}^{(1),\mathrm{hc}}
		\bigr)_{L^2(\mathbb{Z}_{>0}\times C_{a,\varphi})}
		\\=
		1+\sum_{M=1}^{\infty}
		\frac{1}{M!}
		\sum_{\sigma\in S(M)}
		(-1)^{\sigma}
		\prod_{j=1}^{M}
		\int_{C_{a,\varphi}}
		\frac{dw_j}{2\pi\mathbf{i}}
		\sum_{n_j=1}^{\infty}
		\frac{\zeta^{n_j}}{q^{n_j}w_j-w_{\sigma(j)}}
		\frac{g^{\mathrm{hc}}(w_j)}{g^{\mathrm{hc}}(q^{n_j}w_j)}.
	\end{multline}
	We aim to apply \eqref{Mellin_barnes_summation} to each of the sums over $n_j$ above 
	(which converge for sufficiently small $|\zeta|$,
	as follows from the proof of \Cref{statement:first_Fredholm_formula}), 
	that is, with 
	\begin{equation}
	\label{function_F_in_Mellin_Barnes_application}
		F(q^\mathsf{u})=F_{w,w'}(q^\mathsf{u}):=\frac{g^{\mathrm{hc}}(w)}{(q^\mathsf{u}w-w')g^{\mathrm{hc}}(q^\mathsf{u}w)},
	\end{equation}
	where $w,w'\in C_{a,\varphi}$ are fixed.
	For that we define $C_\mathsf{k}$ to be a closed contour which coincides with $D_w$ of
	\Cref{def:d_w_contour_definition} inside the disc of radius $\mathsf{k}+\frac{1}{2}$
	centered at $0$ and which is closed by an arc of the corresponding circle,
	see \Cref{fig:contour_C_k_approximating_D_w}.
	The contour $D_w$ will then serve as $C_{1,2,\ldots }$.

	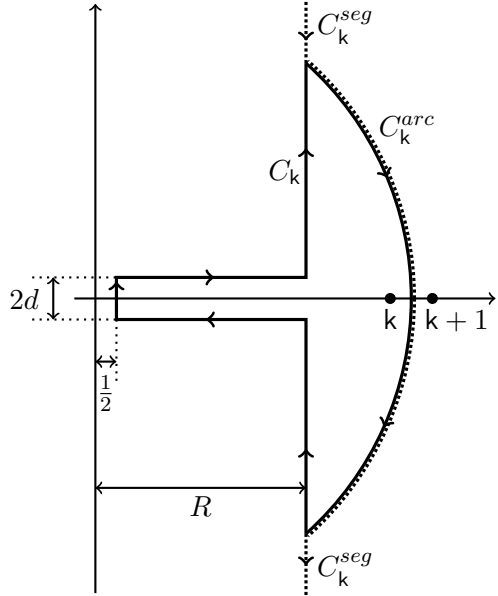
\begin{figure}[htbp]
		\centering
		\begin{tikzpicture}
			[scale=2.8, thick]
			\def\pt{0.02}
			\def\q{.61}
			\def\ss{.56}
			\draw[->] (-.1,0)--(1.9,0);
			\draw[->] (0,-1.4)--(0,1.4);
			\draw[
				very thick,
				decoration={
					markings,
					mark=at position 0.1 with {\arrow{>}},
					mark=at position 0.37 with {\arrow{>}},
					mark=at position 0.52 with {\arrow{>}},
					mark=at position 0.64 with {\arrow{>}},
					mark=at position 0.9 with {\arrow{>}},
				},
				postaction={decorate}
			] 
			(1,-1.118)--(1,-.1)--(.1,-.1)--(.1,.1)--(1,.1)--(1,1.118);
			\node at (0.9,.6) {$C_\mathsf{k}$};
			\draw[<->] (-.2,-.1)--(-.2,.1) node [midway, xshift=-11pt] {$2d$};
			\draw[dotted] (-.3,.1)--++(.5,0);
			\draw[dotted] (-.3,-.1)--++(.5,0);
			\draw[<->] (0,-.9)--++(1,0) node [midway, yshift=-7pt] {$R$};
			\draw[<->] (0,-.3)--++(.1,0) node [midway, yshift=-12pt] {$\tfrac{1}{2}$};
			\draw[dotted] (.1,-.1)--(.1,-.4);
			\draw[very thick,
			decoration={
					markings,
					mark=at position 0.5 with {\arrow{<}},
				},
				postaction={decorate}
			] (1.5,0) arc (0:48.2:1.5);
			\draw[very thick,
			decoration={
					markings,
					mark=at position 0.5 with {\arrow{>}},
				},
				postaction={decorate}
			] (1.5,0) arc (0:-48.2:1.5);
			\draw[fill] (1.4,0) circle (.6pt) node[below] {$\mathsf{k}$};
			\draw[fill] (1.6,0) circle (.6pt) node[below right, xshift=-7pt] {$\mathsf{k}+1$};
			\draw[very thick, densely dotted,
			decoration={
					markings,
					mark=at position 0.5 with {\arrow{<}},
				},
				postaction={decorate}
			] (1,1.118)--++(0,.3) node [below right] {$C_\mathsf{k}^{seg}$};
			\draw[very thick, densely dotted,
			decoration={
					markings,
					mark=at position 0.5 with {\arrow{>}},
				},
				postaction={decorate}
			] (1,-1.118)--++(0,-.3) node [above right] {$C_\mathsf{k}^{seg}$};
			\draw[very thick, densely dotted] (1.515,0) arc (0:-48.2:1.515);
			\draw[very thick, densely dotted] (1.515,0) arc (0:48.2:1.515);
			\node at (1.48,.8) {$C_\mathsf{k}^{arc}$};
		\end{tikzpicture}
		\caption{
			The contour $C_\mathsf{k}$ (solid) used in the proof of
			\Cref{statement:second_Fredholm_formula_}, and the parts
			$C_\mathsf{k}^{seg}$ and $C_{\mathsf{k}}^{arc}$ (dotted) of the symmetric difference between
			$C_\mathsf{k}$ and $D_w$.
		}
		\label{fig:contour_C_k_approximating_D_w}
	\end{figure}
	
	Because of our definitions of $C_{a,\varphi}$ and $D_w$, the contours $C_\mathsf{k}$
	and $D_w$ do not encircle any $\mathsf{u}$ singularities of the integrand
	$\Gamma(-\mathsf{u})\Gamma(1+\mathsf{u})(-\zeta)^{\mathsf{u}}
	F_{w,w'}(q^{\mathsf{u}})$ except for the poles of $\Gamma(-\mathsf{u})$. 
	Indeed, the part of the plane to the right of $D_w$ maps under
	$\mathsf{u}\mapsto q^{\mathsf{u}}$ to the union of the disc and the sector in
	\Cref{fig:infinite_contour_D_definition}, right.  Therefore, because
	$q^{\mathsf{u}}w$ lies to the left of $C_{a,\varphi}\ni w'$, the denominator
	$q^{\mathsf{u}}w-w'$ does not vanish for $\mathsf{u}$ to the right of $D_w$.
	Moreover, poles coming from $g^{\mathrm{hc}}(q^{\mathsf{u}}w)=0$ are
	$q^{\mathsf{u}}w=-\upxi_j\mathsf{s}_j^{-1}q^{-m}$, $m\in\mathbb{Z}_{\ge0}$,
	and they also do not occur for $\mathsf{u}$ to the right of $D_w$ because
	$q^{\mathsf{u}}w$ is to the left of $C_{a,\varphi}$ while the points
	$-\upxi_j\mathsf{s}_j^{-1}q^{-m}$ are all to the right of $C_{a,\varphi}$.
	
	Thus, to apply \eqref{Mellin_barnes_summation} 
	to \eqref{first_Fredholm_formula_written_out}
	it remains to show that the integrals over the symmetric difference of 
	$D_w$ and $C_\mathsf{k}$ go to zero as $\mathsf{k}\to+\infty$.
	The symmetric difference contains two straight lines the union of which is denoted by $C_\mathsf{k}^{seg}$
	and an arc denoted by $C_\mathsf{k}^{arc}$, see \Cref{fig:contour_C_k_approximating_D_w}.
	First, observe that for fixed $w,w'\in C_{a,\varphi}$ the 
	function $F_{w,w'}(q^{\mathsf{u}})$ is uniformly bounded in 
	$\mathsf{u}\in C_\mathsf{k}^{seg}\cup C_\mathsf{k}^{arc}$.

	Consider $C_\mathsf{k}^{seg}$. 
	We have $|(-\zeta)^{\mathsf{u}}|=r^xe^{-\sigma y}$ for
	$-\zeta=re^{\mathbf{i}\sigma}$, $r<1$, $\sigma\in(-\pi,\pi)$, and
	$\mathsf{u}=x+\mathbf{i}y$. 
	Moreover, we have 
	\begin{equation*}
		\left|
		\Gamma(-\mathsf{u})\Gamma(1+\mathsf{u})
		\right|
		=
		\left|
		\frac{\pi}{\sin(-\pi \mathsf{u})}
		\right|
		\le
		\frac{B}{e^{\pi |y|}}, 
		\qquad 
		\mathsf{u}=x+\mathbf{i}y, 
		\qquad 
		\mathop{\mathrm{dist}}(\mathsf{u},\mathbb{Z})\ge\frac{1}{2}.
	\end{equation*}
	Therefore, on $C_\mathsf{k}^{seg}$ one has
	\begin{equation*}
		\left|
		(-\zeta)^{\mathsf{u}}
		\Gamma(-\mathsf{u})\Gamma(1+\mathsf{u})
		\right|
		\le 
		B
		r^{R}e^{-y\sigma}e^{-\pi|y|}
		,\qquad 
		\mathsf{u}=x+\mathbf{i}y.
	\end{equation*}
	Because $|\sigma|<\pi$, the integral of the above expression
	over the infinite vertical contour converges, and so the 
	integral over $C_\mathsf{k}^{seg}$ goes to zero as $\mathsf{k}\to+\infty$.

	Now consider $C_\mathsf{k}^{arc}$. On this contour
	\begin{equation*}
		\left|
		(-\zeta)^{\mathsf{u}}
		\Gamma(-\mathsf{u})\Gamma(1+\mathsf{u})
		\right|
		\le 
		B
		r^x e^{-y\sigma-\pi|y|}
		\le
		B_1
		e^{-c(x+|y|)}.
	\end{equation*}
	Therefore, for $\mathsf{k}\to+\infty$ the integrand decays exponentially in $\mathsf{k}$, 
	while the length of the contour grows only linearly in $\mathsf{k}$,
	and so the integrals over $C_\mathsf{k}^{arc}$ go to zero. 
	Thus, we can apply the Mellin--Barnes summation
	which gives the desired Fredholm determinant 
	\eqref{second_Fredholm_formula} for small $|\zeta|$.
	
\medskip

	\textbf{Step 2.} Now that we have established \eqref{second_Fredholm_formula}
	for sufficiently small $|\zeta|$, it remains to justify that this identity can be
	analytically continued to $\zeta\in\mathbb{C}\setminus\mathbb{R}_{\ge0}$.
	The left-hand side is analytic because it can be represented as a series
	$\sum_{n=0}^{\infty}
	\mathbb{P}(\HeightFunction_{\VertexXProcess(t)}(k)=n)/(\zeta q^n;q)_{\infty}$
	with probabilities in the numerator bounded by~$1$.
	To show that the Fredholm determinant
	$\det(1+K_{\zeta}^{(2),\mathrm{hc}})_{L^2(C_{a,\varphi})}$ in the right-hand
	side of \eqref{second_Fredholm_formula} is analytic, we will show that
	its Fredholm expansion (as a sum over $M\ge0$) is uniformly absolutely
	convergent in $\zeta$ belonging to any closed disc in
	$\mathbb{C}\setminus\mathbb{R}_{\ge0}$.

	We have
	\begin{align}
		\nonumber
		&\det
		\bigl(
			1+
			K_{\zeta}^{(2),\mathrm{hc}}
		\bigr)_{L^2(C_{a,\varphi})}
		\\&\hspace{30pt}=
		1+\sum_{M=1}^{\infty}\frac{1}{M!}
		\int_{C_{a,\varphi}}\frac{dw_1}{2\pi\mathbf{i}}\ldots 
		\int_{C_{a,\varphi}}\frac{dw_M}{2\pi\mathbf{i}}
		\mathop{\mathrm{det}}\limits_{i,j=1}^{M}
		\Bigl[
			K_{\zeta}^{(2),\mathrm{hc}}(w_i,w_j)
		\Bigr]
		\nonumber
		\\&\hspace{30pt}=
		1+
		\sum_{M=1}^{\infty}\frac{1}{M!}
		\int_{C_{a,\varphi}}\frac{dw_1}{2\pi\mathbf{i}}\ldots 
		\int_{C_{a,\varphi}}\frac{dw_M}{2\pi\mathbf{i}}
		\int_{D_{w_1}}\frac{d\mathsf{u}_1}{2\pi\mathbf{i}}\ldots 
		\int_{D_{w_M}}\frac{d\mathsf{u}_M}{2\pi\mathbf{i}}
		\mathop{\mathrm{det}}\limits_{i,j=1}^{M}
		\biggl[
			\frac{1}{q^{\mathsf{u}_i}w_i-w_j}
		\biggr]
		\nonumber
		\\&\hspace{260pt}
		\times
		\prod_{j=1}^{M}
		\Gamma(-\mathsf{u}_j)\Gamma(1+\mathsf{u}_j)(-\zeta)^{\mathsf{u}_j}
		\frac{g^{\mathrm{hc}}(w_j)}{g^{\mathrm{hc}}(q^{\mathsf{u}_j}w_j)},
		\label{second_Fredholm_formula_estimating}
	\end{align}
	where $g^{\mathrm{hc}}(w)$ is given by \eqref{second_Fredholm_formula_g_function}.
	First let us estimate the product of the $q$-Pochhammer symbols coming from
	the product of $g^{\mathrm{hc}}(w_j)/g^{\mathrm{hc}}(q^{\mathsf{u}_j}w_j)$:
	\begin{equation}
	\label{product_of_q_Pochhammers_to_bound}
		\prod_{j=1}^{M}
		\frac{(w_jq^{\mathsf{u}_j}/\Speed(0);q)_{\infty}}{(w_j/\Speed(0);q)_{\infty}}
		\prod_{i=1}^{k-1}
		\frac{(-w_j\mathsf{s}_i\upxi_i^{-1};q)_{\infty}}{(-w_jq^{\mathsf{u}_j}\mathsf{s}_i\upxi_i^{-1};q)_{\infty}}
		\frac{(-w_jq^{\mathsf{u}_j}\mathsf{s}_i^{-1}\upxi_i^{-1};q)_{\infty}}{(-w_j\mathsf{s}_i^{-1}\upxi_i^{-1};q)_{\infty}}.
	\end{equation}
	The $q$-Pochhammer symbols are estimated as follows 
	(where $z\in\mathbb{C}\setminus\left\{ 0 \right\}$ is arbitrary):
	\begin{equation*}
		\left|(z;q)_{\infty}\right|
		=
		\left|(z;q)_{r}(zq^r;q)_{\infty}\right|
		\le
		B_1
		\left|
		(z;q)_r
		\right|
		\le
		B_2
		|z|^{c_2\log|z|}
		,
	\end{equation*}
	where $r$ is such that $|zq^r|\le1$ and hence can be chosen to satisfy 
	$r\le c_1\log|z|$.
	Because $\Re\mathsf{u}_j>0$ for $\mathsf{u}_j\in D_{w_j}$, we have
	$|q^{\mathsf{u}_j}|<1$.
	Thus, we can write for all $w_j\in C_{a,\varphi}$:
	\begin{equation}
	\label{product_of_pochhammers_bounding___ref_in_expon_model}
		\bigl|\eqref{product_of_q_Pochhammers_to_bound}\bigr|
		\le
		B_3^{M}\prod_{j=1}^{M}
		|w_j|^{
		c_3
		\left(
			\log|w_j/\Speed(0)|
			+
			\sum_{i=1}^{k-1}\log|w_j\mathsf{s}_i\upxi_i^{-1}|
			+
			\sum_{i=1}^{k-1}\log|w_j\mathsf{s}_i^{-1}\upxi_i^{-1}|
		\right)}
		\le
		B_4^{M}\prod_{j=1}^{M}
		|w_j|^{c_3 \log|w_j|+c_4}.
	\end{equation}
	Indeed, 
	$|w_j|$ is bounded from below 
	along $C_{a,\varphi}$,
	and so the product of the inverses of the $q$-Pochhammer symbols
	in \eqref{product_of_q_Pochhammers_to_bound}
	can be bounded from above by a constant.

	The product of $g^{\mathrm{hc}}(w_j)/g^{\mathrm{hc}}(q^{\mathsf{u}_j}w_j)$ also contains the following
	exponential terms:
	\begin{equation}
	\label{exponential_part_of_Fredholm_to_bound}
		\exp
		\Bigl\{ 
			t\sum_{j=1}^{M}w_j(q^{\mathsf{u}_j}-1)
		\Bigr\}.
	\end{equation}
	From \Cref{def:d_w_contour_definition} it follows that 
	$\Re\left( w_j(q^{\mathsf{u}_j}-1) \right)\le -c_\varphi|w_j|$
	for all $w_j\in C_{a,\varphi}$ and $\mathsf{u}_j\in D_{w_j}$,
	where $c_\varphi>0$ is a constant depending on $\varphi\in(0,\pi/2)$.
	Therefore, 
	\begin{equation*}
		\bigl|\eqref{exponential_part_of_Fredholm_to_bound}\bigr|\le
		\exp\Bigl\{
			-c_\varphi t\sum_{j=1}^{M}|w_j|
		\Bigr\}.
	\end{equation*}

	Moreover, the expression $|q^{\mathsf{u}_i}w_i-w_j|$ is bounded from below for
	$w_i,w_j\in C_{a,\varphi}$ and $\mathsf{u}_i\in D_{w_i}$, and so Hadamard's 
	inequality allows to bound the determinant in the right-hand side of
	\eqref{second_Fredholm_formula_estimating}
	by $B_5^{M}M^{M/2}$.

	Therefore, one can bound the absolute value of the $M$-th term in the series
	in the right-hand side of \eqref{second_Fredholm_formula_estimating}
	by 
	\begin{equation}
	\label{integral_with_respect_to_arc_length_in_proof_of_second_Fredholm}
		\frac{1}{M!}B_4^MB_5^MM^{M/2}
		\left( 
			\int_{C_{a,\varphi}}|dw|\int_{D_w}|d\mathsf{u}|
			\,\bigl|
				\Gamma(-\mathsf{u})\Gamma(1+\mathsf{u})(-\zeta)^{\mathsf{u}}
			\bigr|\,
			|w|^{c_3\log|w|+c_4}e^{-c_\varphi t |w|}
		\right)^M,
	\end{equation}
	where $|dw|$ and $|d\mathsf{u}|$ stand for integration with respect to the arc length.

	Fix a closed disc in $\mathbb{C}\setminus\mathbb{R}_{\ge0}$ in which $\zeta$ lies, and
	let $\sup|\zeta|=\rho>0$ in that disc.
	The integral with respect to $|d\mathsf{u}|$ can be estimated as follows. 
	For the part 
	$
		(R-\mathbf{i}d,\tfrac{1}{2}-\mathbf{i}d]
		\cup
		(\tfrac{1}{2}-\mathbf{i}d,\tfrac{1}{2}+\mathbf{i}d]
		\cup
		(\tfrac{1}{2}+\mathbf{i}d,R+\mathbf{i}d]
	$
	of the contour $D_w$ 
	(recall that $R=R(w)\sim B_R\log|w|$ and $d=d(w)\sim B_d |w|^{-1}$)
	one can check that $|\sin \pi\mathsf{u}|\ge B_6 d$, and so 
	$|\Gamma(-\mathsf{u})\Gamma(1+\mathsf{u})|\le B_7 |w|$. 
	This part of the integration contour has length of order $R$, and 
	$|(-\zeta)^{\mathsf{u}}|$ is bounded by $\rho^{R}$. 
	Therefore, the integral with respect to $|d\mathsf{u}|$ 
	over this part of the contour is bounded by 
	$B_8|w|\rho^{c_5\log|w|}\log |w|$.
	On the remaining part 
	$(R-\mathbf{i}\infty, R-\mathbf{i}d]
	\cup
	(R+\mathbf{i}d,R+\mathbf{i}\infty)$
	of the contour $D_w$ the integrand decays exponentially in $\mathsf{u}$
	thanks to the presence of the gamma functions.
	Therefore, the integral over this part of the contour
	is estimated by $B_9\rho^{R}\sim B_9\rho^{c_5\log|w|}$.

	Plugging these estimates into \eqref{integral_with_respect_to_arc_length_in_proof_of_second_Fredholm}
	we see that 
	\begin{equation*}
		\eqref{integral_with_respect_to_arc_length_in_proof_of_second_Fredholm}
		\le
		\frac{B^MM^{M/2}}{M!}
		\left( 
			\int_{C_{a,\varphi}}|dw|
			|w|^{c\log|w|+c'}
			\rho^{c''\log|w|}
			\log|w|
			e^{-c_\varphi t |w|}
		\right)^M.
	\end{equation*}
	Due to the exponential term the integral in $w$ converges, and thus the 
	sum over $M$ in \eqref{second_Fredholm_formula_estimating}
	is absolutely convergent, uniformly in $\zeta$ belonging to 
	any fixed closed disc in $\mathbb{C}\setminus\mathbb{R}_{\ge0}$.
	This completes the proof.
\end{proof}

\subsection{Completing the proof of \Cref{thm:exponential_Fredholm}}
\label{sub:proof_of_exponential_Fredholm}

Passing to the continuous space limit in the Fredholm determinantal formula of
\Cref{statement:second_Fredholm_formula_} for the half-continuous vertex model
$\VertexXProcess(t)$ yields \Cref{thm:exponential_Fredholm}. 
Indeed, let the parameters $\upxi_i$ and $\mathsf{s}_i$ depend on $\varepsilon$ and on 
parameters \eqref{Expmodel_data_assumptions} 
of the inhomogeneous exponential jump model $\XProcess(t)$
as explained in \Cref{sub:limit_to_continuous_space}.
Let also $k$ in \eqref{second_Fredholm_formula} depend on $\varepsilon$ 
as $k=\lfloor \varepsilon^{-1}x \rfloor$ with $x\in\mathbb{R}_{>0}$ fixed.
The convergence of the $\varepsilon$-dependent half-continuous vertex model
$\VertexXProcess^{\varepsilon}(t)$
to $\XProcess(t)$
(\Cref{statement:limit_of_processes_to_exponential_model}) readily implies that
\begin{equation*}
	\lim_{\varepsilon\searrow0}
	\mathop{\mathbb{E}}
	\frac{1}{
	\bigl(
		\zeta q^{\HeightFunction_{\VertexXProcess^{\varepsilon}(t)}(\lfloor \varepsilon^{-1}x \rfloor )};q
	\bigr)_{\infty}}
	=
	\mathop{\mathbb{E}}
	\frac{1}{
	\bigl(
		\zeta q^{\HeightFunction_{\XProcess(t)}(x)};q
	\bigr)_{\infty}}
	.
\end{equation*}
Let us show the convergence of the corresponding Fredholm determinants.
The integration contours $C_{a,\varphi}$ and $D_w$ in
\Cref{statement:second_Fredholm_formula_} can be chosen independent of $\varepsilon$ as
long as we take $0<a<\mathcal{W}_x$ and $\varepsilon$ sufficiently small
(recall the notation $\mathcal{W}_x$ from \Cref{def:essential_ranges}).
Observe that
\begin{multline*}
	g^{\mathrm{hc};\varepsilon}(w)
	=
	\frac{e^{-tw}}{(w/\Speed(0);q)_{\infty}}
	\prod_{b\in \RoadblockSet,\,b<x}
	\frac{(w \RoadblockProb(b)/\Speed(b);q)_\infty}
	{(w/\Speed(b);q)_\infty}
	\\\times\exp
	\Biggl\{
		\sum_{\substack{j=1\\j\notin \RoadblockSet^{\varepsilon}}}^{\lfloor \varepsilon^{-1}x \rfloor-1}
		\sum_{i=0}^{\infty}
		\log
		\biggl(
			1+(1-e^{- \RateLambda\varepsilon})
			\frac{q^i w/\Speed(j\varepsilon)}{1-q^i w/\Speed(j\varepsilon)}
		\biggl)
	\Biggr\}
\end{multline*}
(here we take the standard logarithm with cut along $\mathbb{R}_{\le0}$).
We can expand for small $\varepsilon>0$:
\begin{equation*}
	\sum_{i=0}^{\infty}\log
	\biggl(
		1+(1-e^{- \RateLambda\varepsilon})
		\frac{q^i w/\Speed(j\varepsilon)}{1-q^i w/\Speed(j\varepsilon)}
	\biggl)
	=
	\RateLambda\varepsilon
	\sum_{i=0}^{\infty}
	\frac{q^i w/\Speed(j\varepsilon)}{1-q^i w/\Speed(j\varepsilon)}
	+O(\varepsilon^2)
	=
	\RateLambda\varepsilon \,
	\phi_0\biggl(
		\frac{w}{\Speed(j\varepsilon)}
	\biggr)
	+O(\varepsilon^2)
	,
\end{equation*}
where $O(\varepsilon^2)$ is uniform in $j$ and $\phi_0$ is defined in \Cref{sec:appendix_q}. 
Thus, the sum over $j$ in the exponent in $g^{\mathrm{hc};\varepsilon}(w)$ can be approximated
by the integral $\RateLambda\int_0^x\phi_0\bigl(w/\Speed(y)\bigr)dy$,
and so $\lim_{\varepsilon\searrow0}g^{\mathrm{hc};\varepsilon}(w)=g(w)$, 
where $g(w)$ is given in \eqref{g_function_in_Fredholm_for_exponential_model_section2}.
Moreover, for the integrand in the kernel 
$K^{(2),\mathrm{hc};\varepsilon}_{\zeta}$
\eqref{second_Fredholm_formula_kernel} 
corresponding to the $\varepsilon$-dependent 
half-continuous model we have
\begin{equation}
	\label{uniform_convergence_of_integrands}
	\lim_{\varepsilon\searrow0}
	\frac{\Gamma(-\mathsf{u})\Gamma(1+\mathsf{u})(-\zeta)^{\mathsf{u}}}{q^{\mathsf{u}}w-w'}
	\frac{g^{\mathrm{hc};\varepsilon}(w)}{g^{\mathrm{hc};\varepsilon}(q^{\mathsf{u}}w)}
	=
	\frac{\Gamma(-\mathsf{u})\Gamma(1+\mathsf{u})(-\zeta)^{\mathsf{u}}}{q^{\mathsf{u}}w-w'}
	\frac{g(w)}{g(q^{\mathsf{u}}w)},
\end{equation}
where the convergence is uniform in $w,w'\in C_{a,\varphi}$ and $\mathsf{u}\in
D_w$ because of the rapid decay of the pre-limit and the limiting functions for large
$|w|$ or $|\mathsf{u}|$. 
This decay follows from arguments similar to
the proof of step~2 of \Cref{statement:second_Fredholm_formula_}.
The only new estimate needed is
\begin{lemma}
	\label{statement:estimate_decay_phi0}
	For $w\in C_{a,\varphi}$ and any $h>a$ 
	we have the following estimate:
	\begin{equation*}
		\Re \phi_0(w/h)<c \log|w|+c',\qquad c,c'>0.
	\end{equation*}
\end{lemma}
\begin{proof}
	For simplicity let us assume that $h=1$ and $a<1$.
	Since $\phi_0(w)$ is continuous on $C_{a,\varphi}$, 
	it suffices to obtain the estimate for large $|w|$. 
	Using representation \eqref{little_phi_series} for $\phi_0$,
	we have for $w=|w|e^{\mathbf{i}\theta}$, $\theta\in(0,\pi/2)$ 
	(here $\theta\approx\varphi$ for large $|w|$; the case when $\Im w<0$ is symmetric):
	\begin{equation*}
		\Re \phi_0(w)
		=
		\sum_{i=0}^{\infty}
		\frac{|w|q^{i}(\cos \theta -|w|q^{i})}
		{
			(1-|w|q^{i})^2
			+
			2|w|q^i(1-\cos\theta)
		}
		\le
		\sum_{i=0}^{\infty}
		\frac{|w|q^{i}\cos \theta }
		{
			(1-|w|q^{i})^2
			+
			2|w|q^i(1-\cos\theta)
		}
		.
	\end{equation*}
	Let $m$ be the smallest integer
	such that $|w|q^m<\frac{1}{2}$ (thus, $m$ is of order $\log|w|$).
	When $i>m$, we can estimate each term above by
	$\frac{|w|q^i\cos\theta}{(1-|w|q^i)^2}$, and the sum of these terms over $i>m$ 
	is bounded from above by 
	$\cos\theta\,\phi_1(|w|q^m)\le \cos\theta\,\phi_1(\frac{1}{2})$,
	which is a constant. 
	Next, when $i\le m$, we can write 
	\begin{equation*}
		\frac{|w|q^{i}\cos \theta }
		{
			(1-|w|q^{i})^2
			+
			2|w|q^i(1-\cos\theta)
		}
		\le
		\frac{\cos\theta}{2(1-\cos\theta)}, 
	\end{equation*}
	which is a constant. 
	Thus, the sum over $i\le m$ is bounded by a constant times $\log|w|$.
\end{proof}
The uniform convergence in \eqref{uniform_convergence_of_integrands}
plus the absolute convergence of the series in $M$ for the Fredholm determinant
\eqref{second_Fredholm_formula} uniformly in $\zeta$ belonging
to closed discs in $\mathbb{C}\setminus\mathbb{R}_{\ge0}$
(established in the proof of \Cref{statement:second_Fredholm_formula_})
implies that
\begin{equation*}
	\lim_{\varepsilon\searrow0}
	\det
	\bigl(
		1+K_{\zeta}^{(2),\mathrm{hc};\varepsilon}
	\bigr)_{L^2(C_{a,\varphi})}
	=
	\det
	\bigl(
		1+K_{\zeta}
	\bigr)_{L^2(C_{a,\varphi})}
	,
\end{equation*}
where $K_\zeta$ is given in \eqref{Fredholm_for_exponential_model_section2_main_formula}.
This completes the proof of \Cref{thm:exponential_Fredholm}.

\begin{remark}
	\label{rmk:assumptions_on_continuity_of_Speed_function}
	It seems likely that the continuity conditions on the speed function
	$\Speed(x)$ in \eqref{Expmodel_data_assumptions} can be relaxed, and
	\Cref{thm:exponential_Fredholm} together with our asymptotic results of
	\Cref{sec:asymptotic_analysis} would still hold. 
	However, these conditions are relatively general, and are convenient for
	taking the limit of the half-continuous vertex model formulas because
	one avoids pathologies in approximating the exponential model
	by the $\varepsilon$-dependent half-continuous models in discrete space.
\end{remark}

\section{Asymptotic analysis}
\label{sec:asymptotic_analysis}

In this section we perform the asymptotic analysis of the inhomogeneous
exponential jump model $\XProcess(t)$ described in
\Cref{sub:definition_of_the_model} in the regime $\RateLambda\to+\infty$ and
$t=\tau\RateLambda$ (with $\tau>0$ fixed), and prove the main result of the
paper, \Cref{thm:main_theorem_on_fluctuations}.

\subsection{Setup of the asymptotic analysis}
\label{sub:setup_of_asymptotic_analysis}

The starting point of our asymptotic analysis is the Fredholm determinantal
formula 
\begin{equation}
	\label{Fredholm_for_exponential_model_section4_copy}
	\mathop{\mathbb{E}}
	\bigl(
		\zeta q^{\HeightFunction_{\XProcess(t)}(x)};q
	\bigr)_{\infty}^{-1}
	=
	\det
	\bigl(
		1+K_{\zeta}
	\bigr)_{L^2(C_{a,\varphi})}
\end{equation}
of \Cref{thm:exponential_Fredholm}.
Set
\begin{equation}
	\label{zeta_depends_on_lambda_for_asymptotics}
	\zeta=\zeta(\RateLambda):=-q^{- \RateLambda \LimitShape(\tau,x)+r \RateLambda^{\beta}}.
\end{equation}
The function $\LimitShape(\tau,x)$ will be chosen so that both sides of
\eqref{Fredholm_for_exponential_model_section4_copy} have nontrivial limits,
and will eventually coincide with the limit shape described in
\Cref{sub:intro_limit_shape}.
The term $r \RateLambda^{\beta}$ (with $r\in\mathbb{R}$) is a lower order
correction capturing the distribution of fluctuations. 
The exponent $\beta$ is
equal to $\frac{1}{3}$ or $\frac{1}{2}$ depending on the phase (Tracy--Widom or
Gaussian, respectively, cf. \Cref{def:phases_of_the_limit_shape_TW_Gaussian}).
With this choice of $\zeta$, the asymptotic behavior of the left-hand side of
\eqref{Fredholm_for_exponential_model_section4_copy} is as follows:
\begin{lemma}
	\label{statement:limit_q_exponential_to_probability}
	With $\zeta(\RateLambda)$ given by
	\eqref{zeta_depends_on_lambda_for_asymptotics} we have
	\begin{equation}
	\label{limit_q_exponential_to_probability}
		\lim_{\RateLambda\to+\infty}
		\mathop{\mathbb{E}}
		\frac{1}{
			\bigl(
				\zeta(\RateLambda) q^{\HeightFunction_{\XProcess(\tau\RateLambda)}(x)};q
			\bigr)_{\infty}
		}
		=
		\lim_{\RateLambda\to+\infty}
		\mathbb{P}
		\left( 
			\frac{\HeightFunction_{\XProcess(\tau\RateLambda)}(x)-
			\RateLambda \LimitShape(\tau,x)}{\RateLambda^{\beta}}
			\ge -r
		\right).
	\end{equation}
	Equality \eqref{limit_q_exponential_to_probability}
	is understood in the sense that
	if one of the limits exists, then the other one also exists
	and they are equal to each other.
\end{lemma}
\begin{proof}
	See {\cite[Lemma 4.1.39]{BorodinCorwin2011Macdonald} and \cite[Lemmas 5.1 and 5.2]{FerrariVeto2013}}.
\end{proof}

Setting $t=\tau\RateLambda$ and plugging $\zeta=\zeta(\RateLambda)$ into the
Fredholm determinant in the right-hand side of
\eqref{Fredholm_for_exponential_model_section4_copy} we have in the integrand
in $K_\zeta$
\eqref{Fredholm_for_exponential_model_section2_main_formula_kernel}:
\begin{equation}
	\begin{split}
	&(-\zeta)^{\mathsf{u}}\frac{g(w)}{g(q^{\mathsf{u}}w)}
	=
	\frac{(z/\Speed(0);q)_{\infty}}{(w/\Speed(0);q)_{\infty}}
	\prod_{b\in \RoadblockSet,\,b<x}
	\frac{(z/\Speed(b);q)_\infty}{(w/\Speed(b);q)_\infty}
	\frac{(w \RoadblockProb(b)/\Speed(b);q)_\infty}{(z \RoadblockProb(b)/\Speed(b);q)_\infty}
	\\&\hspace{7pt}\times
	\exp
	\biggl\{
		\tau\RateLambda(z-w)
		+
		(\RateLambda \LimitShape(\tau,x)-r \RateLambda^{\beta})\log(w/z)
		+
		\RateLambda
		\int_0^x
		\left( 
			\phi_0
			\biggl(
				\frac{w}{\Speed(y)}
			\biggr)
			-
			\phi_0
			\biggl(
				\frac{z}{\Speed(y)}
			\biggr)
		\right)
		dy
	\biggr\},
	\end{split}
	\label{plugging_into_Fredholm_for_asymptotics}
\end{equation}
where we denoted $z=q^{\mathsf{u}}w$, and hence $\mathsf{u}=\log (z/w)/\log q$.
In fact, the change of variables form $\mathsf{u}$ to $z$ is not one to one,
and in \Cref{sub:extra_residues} below we will take care of this issue.
For now, observe that the terms above which may grow exponentially in $\RateLambda$
have the form
$\exp\bigl(\RateLambda(G_{\tau,x}(w)-G_{\tau,x}(z))\bigr)$, 
where 
\begin{equation}
	\label{G_function_definition}
	G_{\tau,x}(w):=
	-\tau w+\LimitShape(\tau,x)\log w
	+
	\int_0^x
	\phi_0
	\biggl(
		\frac{w}{\Speed(y)}
	\biggr)
	dy
\end{equation}
(since $\beta=\frac{1}{3}$ or $\frac{1}{2}$, the terms containing $r$ grow slower).
Note that replacing $\log(w/z)$ by $\log w-\log z$ may introduce
additional imaginary terms, but they do not contribute to the exponential growth.

In 
\Cref{sub:limit_shape_formulas}
we investigate critical points of the function $G_{\tau,x}(w)$, 
and in \Cref{sub:steep_descent_contours}
discuss steep descent or ascent contours for this function.
Using these results, in
\Cref{sub:extra_residues} we will return to the 
analysis of the whole Fredholm determinant in the right-hand side
of \eqref{Fredholm_for_exponential_model_section4_copy}
by the steepest descent method.

\subsection{Critical points of $G_{\tau,x}$ and limit shape formulas}
\label{sub:limit_shape_formulas}

Here we explain how formulas for the limit shape $\LimitShape(\tau,x)$ 
given in \Cref{def:limit_shape_definition_H_tau_x}
arise from \eqref{plugging_into_Fredholm_for_asymptotics}--\eqref{G_function_definition}.
For shorter notation, denote
\begin{equation}
	\label{Big_Phi_n_definition}
	\Phi_n(w\mid x):=
	\int_0^x
	\phi_n
	\biggl(
		\frac{w}{\Speed(y)}
	\biggr)
	dy,\qquad n=0,1,2,\ldots .
\end{equation}
Observe that the derivatives of $G_{\tau,x}$ have the following
form (using \Cref{statement:Phi_n_behavior}):
\begin{align}
	\nonumber
	G'_{\tau,x}(w)&=
	\frac{1}{w}\LimitShape(\tau,x)
	-\tau
	+
	\frac{1}{w}
	\Phi_1(w\mid x);
	\\
	G''_{\tau,x}(w)&=
	-\frac{1}{w^2}\LimitShape(\tau,x)
	+
	\frac{1}{w^2}
	\left( 
		\Phi_2(w\mid x)-\Phi_1(w\mid x)
	\right);
	\label{G_derivatives_formulas}
	\\
	G'''_{\tau,x}(w)&=
	\frac{2}{w^3}\LimitShape(\tau,x)+
	\frac{1}{w^{3}}
	\left( 
		\Phi_3(w\mid x)
		-
		3\Phi_2(w\mid x)
		+
		2\Phi_1(w\mid x)
	\right).
	\nonumber
\end{align}

We first consider double critical points of $G_{\tau,x}$ which in the end will
correspond to the Tracy--Widom phase. 
Equations $G'_{\tau,x}(w)=G''_{\tau,x}(w)=0$
for double critical points can be equivalently written as
\begin{align}
	\label{Equation_for_double_critical_points_1}
	\tau w
	&=
	\Phi_2(w\mid x);
	\\
	\LimitShape(\tau,x)
	&=
	\tau w
	-
	\Phi_1(w\mid x),
	\label{Equation_for_double_critical_points_2}
\end{align}
that is, we can separately find $w$ from the first equation 
\eqref{Equation_for_double_critical_points_1}
and then plug it into \eqref{Equation_for_double_critical_points_2}
to get the value of $\LimitShape(\tau,x)$ leading to a double critical point.
Existence and uniqueness of a solution to \eqref{Equation_for_double_critical_points_1}
on $(0,\mathcal{W}^\circ_x)$ 
(with $\mathcal{W}^\circ_x$ given in \eqref{essential_range_of_Speed_notation})
is afforded by \Cref{statement:TW_root_exists_and_is_unique} which we now prove.

\begin{proof}[Proof of \Cref{statement:TW_root_exists_and_is_unique}]
	Equation \eqref{Equation_for_double_critical_points_1} (which is the same as
	\eqref{critical_points_second_equation}) can be rewritten as
	\begin{equation}\label{Equation_for_double_critical_points_1_modified}
		\tau
		=
		\frac{\partial}{\partial w}\,
		\Phi_1(w\mid x).
	\end{equation}
	We need to show that this equation has a unique solution
	$w=\upomega^\circ_{\tau,x}$ in $w\in(0,\mathcal{W}^\circ_{x})$.
	We will use properties of the functions $\phi_n(w)$
	summarized in \Cref{statement:Phi_n_behavior}.
	The functions $\Phi_n(w\mid x)$ are smooth on $(-\infty,\mathcal{W}^\circ_x)$.
	Therefore, \eqref{Equation_for_double_critical_points_1_modified} 
	is equivalent to finding a point in $(0,\mathcal{W}^\circ_{x})$
	at which the tangent line to the graph of the function $\Phi_1(w\mid x)$ has slope $\tau$.

	The function $\Phi_1(w\mid x)$ is positive, strictly increasing, and 
	strictly convex on $(0,\mathcal{W}^\circ_{x})$. 
	Indeed, the positivity and monotonicity follow from the facts that $\phi_1(w)$ and $\phi_1'(w)=w^{-1}\phi_2(w)$
	are positive on $(0,1)$. 
	To get convexity, observe that
	$\phi_1''(w)=w^{-2}(\phi_3(w)-\phi_2(w))$, 
	and
	\begin{equation}
	\label{phi_3_minus_phi_2_positive}
		\phi_3(w)-\phi_2(w)
		=
		\sum_{k=0}^{\infty}
		\bigg(\frac{q^kw (1+4 q^kw +q^{2 k}w^2 )}{(1-q^kw)^4}
		-\frac{q^{k}w(1+q^{k}w)}{(1-q^kw)^{3}}\bigg)
		=
		\sum_{k=0}^{\infty}
		\frac{2 q^{2k} w^2 (2+q^kw)}{(1-q^kw)^4}
		>0
	\end{equation}
	for $w\in(0,1)$. Thus, if a solution to \eqref{Equation_for_double_critical_points_1_modified} 
	exists, it is unique.

	At $\mathcal{W}^\circ_x$ the function $\Phi_1(w\mid x)$ and all its derivatives go to
	infinity. 
	On the other hand, the slope of the tangent line to the graph of
	$\Phi_1(w\mid x)$ at $w=0$ is
	\begin{equation*}
		\frac{\partial}{\partial w}\bigg\vert_{w=0}
		\int_0^x
		\phi_1
		\biggl(
			\frac{w}{\Speed(y)}
		\biggr)
		dy
		=
		\int_{0}^{x}\frac{dy}{(1-q)\Speed(y)}=\tau_e(x)<\tau
	\end{equation*}
	because $x<x_e$ (recall \Cref{def:edge_of_the_limit_shape}). 
	Thus, $\tau$ is greater than the slope at $0$, so the solution
	$\upomega^\circ_{\tau,x}$ exists.
	All other claims 
	in \Cref{statement:TW_root_exists_and_is_unique} are straightforward.
\end{proof}
Note that equation \eqref{Equation_for_double_critical_points_1} 
can have other roots outside the interval $w\in(0,\mathcal{W}_x^\circ)$.

If $\upomega^\circ_{\tau,x}$ is accessible by contour deformations (see
\Cref{sub:extra_residues} below for details), we say that the space-time point
$(\tau,x)$ is in the \emph{Tracy--Widom phase}.
In this case $\LimitShape(\tau,x)$ should be chosen in such a way 
that $\upomega^\circ_{\tau,x}$ is a double critical point of $G_{\tau,x}$, 
i.e., equation \eqref{Equation_for_double_critical_points_2} should also hold.
This leads to the limit shape 
$\LimitShape(\tau,x)=
\tau \upomega^\circ_{\tau,x}-\Phi_1(\upomega^\circ_{\tau,x}\mid x)$
in the Tracy--Widom phase.

On the other hand, the point $\upomega^\circ_{\tau,x}$ may be inaccessible 
by contour deformations due to the presence of denominators in 
\eqref{plugging_into_Fredholm_for_asymptotics}. 
Then we say that the space-time point $(\tau,x)$ is in the \emph{Gaussian phase}.
The smallest of the poles in these denominators, 
$w=\mathcal{W}_x$ with $\mathcal{W}_x$ given in \eqref{range_of_Speed_notation},
is the first of the obstacles preventing the contour deformations 
(the obstacle exists if and only if $\mathcal{W}_x < \upomega^\circ_{\tau,x}$).
Then we can choose $\LimitShape(\tau,x)$ so that $\mathcal{W}_x$ becomes a simple critical
point (equation \eqref{Equation_for_double_critical_points_2} is equivalent to $G'_{\tau,x}(w)=0$).
Therefore, $\LimitShape(\tau,x)=\tau\mathcal{W}_x-\Phi_1(\mathcal{W}_x\mid x)$
is the limit shape in the Gaussian phase.
Thus, the function $\LimitShape(\tau,x)$ given in \eqref{limit_shape_definition_H_tau_x}
serves in both Tracy--Widom and Gaussian phases.

We have now explained how \Cref{def:phases_of_the_limit_shape_TW_Gaussian,def:limit_shape_definition_H_tau_x}
arise from looking at the integrand in the kernel in the Fredholm determinantal formula
\eqref{Fredholm_for_exponential_model_section4_copy}. 
Let us establish the monotonicity properties of 
$\LimitShape(\tau,x)$ \eqref{limit_shape_definition_H_tau_x}
listed in
\Cref{statement:properties_of_limit_shape_after_definition}:

\begin{proof}[Proof of \Cref{statement:properties_of_limit_shape_after_definition}]
	The left continuity of $x\mapsto\LimitShape(\tau,x)$ follows from the left continuity in $x$
	of $\mathcal{W}_x$ and $\Phi_{1,2}(w\mid x)$, which also implies that $\upomega^\circ_{\tau,x}$ 
	is left continuous in $x$.

	The claim that 
	$\LimitShape(\tau,x_e(\tau))=0$ 
	follows from the fact that 
	$\upomega^\circ_{\tau,x_e(\tau)}=0$.
	Note that this implies that points 
	$(\tau,x)$ 
	in a neighborhood of the edge, i.e., for
	$x\in(x_e(\tau)-\delta,x_e(\tau))$,
	are always in the Tracy--Widom phase.
	
	For the monotonicity of $\LimitShape(\tau,x)$ in $x$, observe that 
	$x\mapsto\min(\upomega^\circ_{\tau,x},\mathcal{W}_x)$
	is decreasing: in the Gaussian phase is it piecewise constant and decreases,
	and in the Tracy--Widom phase it strictly decreases by
	\Cref{statement:TW_root_exists_and_is_unique}.
	Denote 
	\begin{equation}
		\label{w_cr_notation}
		\upomega_{\mathsf{cr}}=\upomega_{\mathsf{cr}}(x)
		:=\min(\upomega^\circ_{\tau,x},\mathcal{W}_x)>0
	\end{equation}
	for simpler notation,
	and observe that for any 
	$0<x<x'<x_e(\tau)$
	we have
	\begin{equation*}
		\LimitShape(\tau,x)
		-
		\LimitShape(\tau,x')
		=
		\tau(\underbrace{\upomega_{\mathsf{cr}}(x)- \upomega_{\mathsf{cr}}(x')}_{\ge0})
		+
		\Phi_1(\upomega_{\mathsf{cr}}(x')\mid x)
		-
		\Phi_1(\upomega_{\mathsf{cr}}(x)\mid x)
		+
		\int_x^{x'}
		\phi_1
		\biggl(
			\frac{\upomega_{\mathsf{cr}}(x')}{\Speed(y)}
		\biggr)dy.
	\end{equation*}
	The last integral is positive.
	Since for fixed $x$ the function 
	$w\mapsto \Phi_1(w\mid x)$
	is differentiable, 
	there exists $w_0$ between $\upomega_{\mathsf{cr}}(x)$ and $\upomega_{\mathsf{cr}}(x')$ such that 
	\begin{equation*}
		\Phi_1(\upomega_{\mathsf{cr}}(x')\mid x)-\Phi_1(\upomega_{\mathsf{cr}}(x)\mid x)
		=
		(\upomega_{\mathsf{cr}}(x')-\upomega_{\mathsf{cr}}(x))
		\frac{\partial}{\partial w}
		\Big\vert_{w=w_0}\Phi_1(w\mid x).
	\end{equation*}
	To prove monotonicity it suffices to show that
	\begin{equation*}
		\tau
		-
		\frac{\partial}{\partial w}
		\Big\vert_{w=w_0}\Phi_1(w\mid x)\ge0.
	\end{equation*}
	But from \Cref{statement:TW_root_exists_and_is_unique}
	it follows that $\tau$
	is equal to the same derivative of $\Phi_1$,
	but at $w=\upomega^\circ_{\tau,x}$. 
	Since 
	$w_0\le \upomega_{\mathsf{cr}}(x)\le \upomega^\circ_{\tau,x}$,
	the above inequality holds by the convexity of $\Phi_1$, which completes the
	proof.
\end{proof}

\subsection{Steep descent and steep ascent contours}
\label{sub:steep_descent_contours}

In this subsection we
show that certain contours are steep descent or ascent for the
function $\Re G_{\tau,x}(w)$
\eqref{G_function_definition}
in the sense that on these contours $\Re G_{\tau,x}(w)$ 
attains its only maximum or
minimum, respectively, at a critical point of $G_{\tau,x}$.
For shorter formulas in the rest of the section
we will continue to use the notation $\upomega_{\mathsf{cr}}$
\eqref{w_cr_notation}.
Recall that $\upomega_{\mathsf{cr}}$
is a double critical point of $G_{\tau,x}$
in the Tracy--Widom phase or a simple critical
point of $G_{\tau,x}$ in the Gaussian phase.
Let also $\Gamma_{\mathsf{cr}}$ be the clockwise oriented circle
centered at zero of radius $\upomega_{\mathsf{cr}}$.

\subsubsection{Tracy--Widom phase}

Recall the contour $C_{a,\varphi}$ of \Cref{def:c_a_phi_contour_definition}.
\begin{proposition}
	\label{statement:TW_contour_C_is_steep_descent}
	If the space-time point $(\tau,x)$ is in the Tracy--Widom phase, then the 
	contour $C_{\upomega^\circ_{\tau,x},\frac{\pi}{4}}$ is steep descent for the function $\Re G_{\tau,x}$
	in the sense that 
	\begin{equation*}
		\Re G_{\tau,x}(w)<\Re G_{\tau,x}(\upomega^\circ_{\tau,x}),
		\qquad 
		w\in C_{\upomega^\circ_{\tau,x},\frac{\pi}{4}}\setminus\left\{ \upomega^\circ_{\tau,x} \right\}.
	\end{equation*}
\end{proposition}
\begin{proof}
	We aim to write down 
	the $v$-derivative
	of the real part of
	$G_{\tau,x}(\upomega^\circ_{\tau,x}+ve^{\mathbf{i}\varphi})$,
	where $v>0$ and $0<\varphi<\pi/2$ 
	(%
		the case of the negative imaginary part
		is symmetric%
	),
	and show that it is negative for $\varphi=\frac{\pi}{4}$.
	We have for the first two terms in 
	$G_{\tau,x}$
	\eqref{G_function_definition}:
	\begin{align*}
		&\frac{\partial}{\partial v}
		\Re
		\Bigl( 
			-\tau(\upomega^\circ_{\tau,x}+ve^{\mathbf{i}\varphi})
			+
			\bigl(
				\underbrace{\tau\upomega^\circ_{\tau,x}-\Phi_1(\upomega^\circ_{\tau,x}\mid x)}
				_{\LimitShape(\tau,x)}
			\bigr)
			\log
			(\upomega^\circ_{\tau,x}+ve^{\mathbf{i}\varphi})
		\Bigr)
		\\&\hspace{20pt}=
		-\tau \cos\varphi+
		\bigl(
			\tau\upomega^\circ_{\tau,x}-\Phi_1(\upomega^\circ_{\tau,x}\mid x)
		\bigr)
		\frac{v+
		\upomega^\circ_{\tau,x}
		\cos \varphi}{v^2+2 v \upomega^\circ_{\tau,x} \cos \varphi+(\upomega^\circ_{\tau,x})^2}
		\\&\hspace{20pt}=
		-v\,\frac{\Phi_2(\upomega^\circ_{\tau,x}\mid x)}{\upomega^\circ_{\tau,x}}
		\frac{v \cos \varphi+\upomega^\circ_{\tau,x}\cos (2 \varphi )}
		{v^2+2 v \upomega^\circ_{\tau,x} \cos\varphi+(\upomega^\circ_{\tau,x})^2}
		-
		\Phi_1(\upomega^\circ_{\tau,x}\mid x)\,
		\frac{v+
		\upomega^\circ_{\tau,x}
		\cos \varphi}{v^2+2 v \upomega^\circ_{\tau,x} \cos \varphi+(\upomega^\circ_{\tau,x})^2}.
	\end{align*}
	The advantage of expressing everything through the integrals $\Phi_{1,2}$
	using \eqref{Equation_for_double_critical_points_1}--\eqref{Equation_for_double_critical_points_2}
	is that it then suffices to prove the desired negativity under the integral.
	For the term in \eqref{G_function_definition} containing $\phi_0$ we have
	(using \eqref{little_phi_series})
	\begin{multline}
		\label{proof_of_TW_steepest_decent_shorthand}
		\frac{\partial}{\partial v}
		\Re\left( 
			\sum_{j=0}^{\infty}
			\frac{(\upomega^\circ_{\tau,x}+ve^{\mathbf{i}\varphi})q^j}
			{\Speed(y)-(\upomega^\circ_{\tau,x}+ve^{\mathbf{i}\varphi})q^j}
		\right)
		\\=
		\sum_{j=0}^{\infty}
		\frac
		{
			q^j \Speed(y)
			\left( 
				q^{2j}
				\bigl(
					2v \upomega^\circ_{\tau,x}+(v^2+(\upomega^\circ_{\tau,x}))\cos\varphi
				\bigr)
				-2 q^j\Speed(y)(v+\upomega^\circ_{\tau,x}\cos\varphi)
				+
				\Speed(y)^2\cos\varphi
			\right)
		}
		{
			\left( 
				q^{2j}
				\bigl(
					v^2 + (\upomega^\circ_{\tau,x})^2 + 2 v \upomega^\circ_{\tau,x} \cos\varphi
				\bigr)
				-
				2q^j \Speed(y) (\upomega^\circ_{\tau,x} + v \cos\varphi)
				+
				\Speed(y)^2
			\right)^2
		}.
	\end{multline}
	To shorten notation, let $D_{\ref{statement:TW_contour_C_is_steep_descent}}
	:=
	v^2+2 v \upomega^\circ_{\tau,x} \cos \varphi+(\upomega^\circ_{\tau,x})^2$,
	and let $\tilde D_{\ref{statement:TW_contour_C_is_steep_descent}}$ 
	stand for the denominator in \eqref{proof_of_TW_steepest_decent_shorthand}.
	Clearly, both $D_{\ref{statement:TW_contour_C_is_steep_descent}}$ 
	and $\tilde D_{\ref{statement:TW_contour_C_is_steep_descent}}$ are positive.
	Using \Cref{statement:Phi_n_behavior} and adding two previous
	expressions, we see that
	\begin{equation}
		\label{polynomial_P_previous_formula}
		\frac{\partial}{\partial v}
		\Re G_{\tau,x}(\upomega^\circ_{\tau,x}+ve^{\mathbf{i}\varphi})
		=
		\int_{0}^{x}
		\Bigg(
			\sum_{j=0}^{\infty}
			\underbrace{\frac{q^{2j}v^2 \Speed(y)}
			{D_{\ref{statement:TW_contour_C_is_steep_descent}}\tilde D_{\ref{statement:TW_contour_C_is_steep_descent}}
			\bigl(q^j \upomega^\circ_{\tau,x}-\Speed(y)\bigr)^3}}_{
			\text{$<0$ because $\upomega^\circ_{\tau,x}<\Speed(y)$}}
			P_{\ref{statement:TW_contour_C_is_steep_descent}}(q^j,\Speed(y),v,\upomega^\circ_{\tau,x},\varphi)
		\Biggr)
		dy,
	\end{equation}
	where $P_{\ref{statement:TW_contour_C_is_steep_descent}}$ is an explicit polynomial in $q^j$, $\Speed(y)$, $v$, and $\upomega^\circ_{\tau,x}$
	also containing $\cos(k\varphi)$, $k=1,2,3,4$.
	Recall that 
	$q^j\upomega^\circ_{\tau,x}-\Speed(y)<0$
	for all $j$. 
	To incorporate this condition into the analysis of $P_{\ref{statement:TW_contour_C_is_steep_descent}}$,
	let us change variables as $\Speed(y)=\Omega+q^j\upomega^\circ_{\tau,x}$,
	where $\Omega>0$. 
	Then the polynomial $P_{\ref{statement:TW_contour_C_is_steep_descent}}$ takes the form
	\begin{align}
		\nonumber
		&P_{\ref{statement:TW_contour_C_is_steep_descent}}(Q,\Omega+Q\upomega,v,\upomega,\varphi)
		=
		- 
		(2 \Omega+3 Q \upomega)\Omega Q^2 v^3
		\\&\hspace{90pt}+
		Q v^2 \left(4 \Omega^3+4 \Omega^2 Q \upomega+\Omega Q^2 v^2-3 \Omega Q^2 \upomega^2+2
		Q^3 v^2 \upomega\right)
		\cos\varphi
		\nonumber\\&\hspace{90pt}
		-v \left(
			2 \Omega^4-\Omega^3 Q \upomega+2 \Omega^2 Q^2 v^2-6 \Omega^2 Q^2 \upomega^2+3
			\Omega Q^3 v^2 \upomega-2 Q^4 v^2 \upomega^2
		\right)
		\cos(2\varphi)
		\nonumber\\&\hspace{90pt}
		-\Omega \left(2 \Omega^3 \upomega-\Omega^2 Q v^2+3 \Omega^2 Q \upomega^2+4 Q^3 v^2
		\upomega^2\right)
		\cos(3\varphi)
		\nonumber\\&\hspace{90pt}+
		\Omega^2 Q v \upomega (\Omega+2 Q \upomega)
		\cos(4\varphi). 
		\label{polynomial_P}
	\end{align}
	Let us now substitute $\varphi=\pi/4$ and check that $P_{\ref{statement:TW_contour_C_is_steep_descent}}$ is always positive.
	First, one can readily verify that 
	\begin{equation*}
		\frac{\partial^2}{\partial v^2}
		P_{\ref{statement:TW_contour_C_is_steep_descent}}(Q,\Omega+Q\upomega,v,\upomega,\tfrac{\pi}{4})>0
		\qquad \textnormal{for all $v\in\mathbb{R}$}
	\end{equation*}
	(this derivative is a quadratic polynomial in $v$). 
	Therefore, $P_{\ref{statement:TW_contour_C_is_steep_descent}}$ 
	is strictly convex in $v$ and thus has a only minimum at $v=v_{\min}$.
	If 
	$v_{\min}\le 0$,
	then we are done, because
	\begin{equation*}
		P_{\ref{statement:TW_contour_C_is_steep_descent}}(Q,\Omega+Q\upomega,0,\upomega,\tfrac{\pi}{4})
		=
		\frac{\Omega^3 \upomega (2 \Omega+3 Q \upomega)}{\sqrt{2}}>0.
	\end{equation*}
	If $v_{\min}>0$, observe that
	\begin{equation*}
		P_{\ref{statement:TW_contour_C_is_steep_descent}}(Q,\Omega+Q\upomega,v_{\min},0,\tfrac{\pi}{4})=
		\frac{1}{2} \Omega Q v_{\min}^2 
		\bigl(3 \sqrt{2} \Omega^2-4 \Omega Q
		v_{\min}+\sqrt{2} Q^2 v_{\min}^2\bigr)>0,
	\end{equation*}
	and 
	\begin{multline*}
		\frac{\partial}{\partial \upomega}P_{\ref{statement:TW_contour_C_is_steep_descent}}
		(Q,\Omega+Q\upomega,v_{\min},\upomega,\tfrac{\pi}{4})
		=
		\Omega Q \upomega 
		\bigl(
			3 \sqrt{2} \Omega^2-4 \Omega Q v_{\min}+\sqrt{2} Q^2
			v_{\min}^2
		\bigr)
		\\+
		\sqrt{2} \Omega^4-\Omega^3 Q v_{\min}+2 \sqrt{2} \Omega^2 Q^2 v_{\min}^2-3 \Omega
		Q^3 v_{\min}^3+\sqrt{2} Q^4 v_{\min}^4.
	\end{multline*}
	The polynomials 
	$3 \sqrt{2} a^2-4 a+\sqrt{2}$ 
	and 
	$\sqrt{2} a^4-a^3+2 \sqrt{2} a^2-3 a+\sqrt{2}$ 
	have no real roots and hence are always positive, 
	so the above $\upomega$-derivative
	is also always positive.
	This implies that 
	$P_{\ref{statement:TW_contour_C_is_steep_descent}}(Q,\Omega+Q\upomega,v_{\min},\upomega,\tfrac{\pi}{4})>0$,
	which completes the proof of the proposition.
\end{proof}

\begin{proposition}
	\label{statement:TW_circle_steep_ascent}
	If the space-time point $(\tau,x)$ is in the Tracy--Widom phase then
	the circle $\Gamma_{\mathsf{cr}}$ centered at the origin of radius $\upomega^\circ_{\tau,x}$ 
	is steep ascent for $\Re G_{\tau,x}$ in the sense that 
	\begin{equation*}
		\Re G_{\tau,x}(w)>\Re G_{\tau,x}(\upomega^\circ_{\tau,x}),
		\qquad 
		w\in 
		\Gamma_{\mathsf{cr}}
		\setminus\left\{ \upomega^\circ_{\tau,x} \right\}.
	\end{equation*}
\end{proposition}
\begin{proof}
	Similarly to the proof of \Cref{statement:TW_contour_C_is_steep_descent},
	let us show that the 
	$\theta$-derivative of 
	$\Re G_{\tau,x}(\upomega^\circ_{\tau,x}e^{\mathbf{i}\theta})$
	is positive for $\theta\in(0,\pi)$
	(the case of the negative imaginary part is symmetric).
	We have
	\begin{equation*}
		\frac{\partial}{\partial\theta}\Re G_{\tau,x}
		(\upomega^\circ_{\tau,x}e^{\mathbf{i}\theta})
		=
		\int_{0}^{x}
			\sum_{j=0}^{\infty}
			\frac{P_{\ref{statement:TW_circle_steep_ascent}}(q^j,\Speed(y),\upomega^\circ_{\tau,x},\theta)
			}{D_{\ref{statement:TW_circle_steep_ascent}}(q^j,\Speed(y),\upomega^\circ_{\tau,x},\theta)}
		\,
		dy,
	\end{equation*}
	where (using $\xi>Q\upomega$)
	\begin{align*}
		D_{\ref{statement:TW_circle_steep_ascent}}(Q,\xi,\upomega,\theta)
		&=
		(Q \upomega-\xi )^3 \left(\xi ^2+Q^2 \upomega^2-2 \xi  Q \upomega \cos \theta \right)^2<0;\\
		P_{\ref{statement:TW_circle_steep_ascent}}(Q,\xi,\upomega,\theta)
		&=
		-16 \xi ^2 Q^2 \upomega^2 
		(\xi +Q \upomega) 
		\left(
			\xi ^2+Q^2 \upomega^2
			-\xi  Q\upomega(1+\cos\theta)
		\right)
		\sin ^3(\theta/2)
		\cos (\theta/2)<0,
	\end{align*}
	which completes the proof.
\end{proof}

\begin{remark}
	\label{rmk:problem_to_define_G_for_negative_reals}
	In connection with \Cref{statement:TW_circle_steep_ascent} 
	note that 
	the real part 
	$\Re G_{\tau,x}(w)$
	is well-defined for negative real $w$
	regardless of the branch of the logarithm.
\end{remark}

\subsubsection{Gaussian phase}

Let us now turn to the Gaussian phase
(%
	we also include into the consideration 
	the transition case when 
	$\upomega^\circ_{\tau,x}=\mathcal{W}_x$%
).
Then the limit shape is 
\begin{equation*}
	\LimitShape(\tau,x)
	=
	\tau\mathcal{W}_x
	-
	\Phi_1(\mathcal{W}_x\mid x), 
\end{equation*}
where 
$0<\mathcal{W}_x \le \upomega^\circ_{\tau,x}$
can be arbitrary.
Recall that by
\Cref{statement:TW_root_exists_and_is_unique}
the root 
$\upomega^\circ_{\tau,x}$
of \eqref{Equation_for_double_critical_points_1}
is defined regardless of which phase $(\tau,x)$ is in.
Let us use this root and rewrite 
$G_{\tau,x}$ 
as
\begin{equation}
	\label{G_of_w_in_Gaussian_phase_using_Phi}
	G_{\tau,x}(w)
	=
	\frac{\Phi_2(\upomega^\circ_{\tau,x}\mid x)}{\upomega^\circ_{\tau,x}}
	\bigl( \mathcal{W}_x\log w-w \bigr)
	-\Phi_1(\mathcal{W}_x\mid x)\log w
	+\Phi_0(w\mid x).
\end{equation}
\begin{proposition}
	\label{statement:Gaussian_contour_C_is_steep_descent}
	If the space-time point $(\tau,x)$ is in the Gaussian phase, then the 
	contour $C_{\mathcal{W}_x,\frac{\pi}{4}}$ is steep descent for 
	$\Re G_{\tau,x}$
	in the sense that on this contour the function
	$\Re G_{\tau,x}(w)$
	attains its only maximum at $w=\mathcal{W}_x$.
\end{proposition}
\begin{proof}
	The function $w\mapsto w^{-1}\Phi_2(w\mid x)$ 
	is strictly increasing in $w\in(0,\mathcal{W}^\circ_x)$, 
	so
	\begin{equation}\label{PHI2_increasing_for_proof_of_descent}
		\tau=\frac{\Phi_2(\upomega^\circ_{\tau,x}\mid x)}{\upomega^\circ_{\tau,x}}
		\ge
		\frac{\Phi_2(\mathcal{W}_x\mid x)}{\mathcal{W}_x}.
	\end{equation}
	On the other hand, 
	$\Re(\mathcal{W}_x\log w-w)$
	on our contour 
	attains its only maximum at 
	$w=\mathcal{W}_x$.
	Indeed, we have for $v>0$:
	\begin{equation*}
		\frac{\partial}{\partial v}\Re\left( 
			\mathcal{W}_x
			\log (\mathcal{W}_x+v e^{\mathbf{i}\varphi})
			-
			(\mathcal{W}_x+v e^{\mathbf{i}\varphi})
		\right)
		=
		-
		\frac{v (v \cos\varphi +\mathcal{W}_x \cos (2 \varphi ))}
		{v^2+2 v \mathcal{W}_x \cos \varphi +\mathcal{W}_x^2}<0
	\end{equation*}
	for $\varphi=\pi/4$.
	Therefore, 
	\begin{equation*}
		\biggl(
			\frac{\Phi_2(\upomega^\circ_{\tau,x}\mid x)}{\upomega^\circ_{\tau,x}}
			-
			\frac{\Phi_2(\mathcal{W}_x\mid x)}{\mathcal{W}_x}
		\biggr)
		\Bigl(
			\Re(\mathcal{W}_x\log w-w)
			-
			\Re(\mathcal{W}_x\log \mathcal{W}_x-\mathcal{W}_x)
		\Bigr)\le0,
	\end{equation*}
	so
	it suffices to prove that the contour 
	$C_{\mathcal{W}_x,\frac{\pi}{4}}$
	is steep descent for a modification 
	$\tilde G_{\tau,x}$
	of 
	$G_{\tau,x}$ \eqref{G_of_w_in_Gaussian_phase_using_Phi}
	obtained by replacing 
	$\Phi_2(\upomega^\circ_{\tau,x}\mid x)/\upomega^\circ_{\tau,x}$
	by 
	$\Phi_2(\mathcal{W}_x\mid x)/\mathcal{W}_x$
	in the first summand.
	Denote
	$\tilde\tau:=\Phi_2(\mathcal{W}_x\mid x)/\mathcal{W}_x$
	and note that $\mathcal{W}_x=\upomega^\circ_{\tilde\tau,x}$,
	so the modified function
	$\tilde G_{\tau,x}$ 
	is simply the same as 
	$G_{\tilde\tau,x}$ corresponding to the Tracy--Widom phase. 
	Therefore, the desired statement now follows from
	\Cref{statement:TW_contour_C_is_steep_descent},
	and so we are done.
\end{proof}

\begin{proposition}
	\label{statement:Gaussian_steep_ascent_contour}
	If the point $(\tau,x)$ is in the Gaussian phase then 
	the circle $\Gamma_{\mathsf{cr}}$ centered at the origin of radius $\mathcal{W}_x$ 
	is steep ascent for 
	$\Re G_{\tau,x}$
	in the sense that on this contour the function
	$\Re G_{\tau,x}(w)$
	attains its only minimum at $w=\mathcal{W}_x$.
\end{proposition}
\begin{proof}
	The function $\Re G_{\tau,x}$ is well-defined on all of
	$\Gamma_{\mathsf{cr}}$, cf. \Cref{rmk:problem_to_define_G_for_negative_reals}.
	We have using \eqref{G_of_w_in_Gaussian_phase_using_Phi}:
	\begin{equation*}
		\frac{\partial}{\partial\theta}G_{\tau,x}(\mathcal{W}_x e^{\mathbf{i}\theta})
		=
		\int_0^x
		\Bigg(
			\sum_{j=0}^{\infty}
			\frac{-q^{2j}\Speed(y)\mathcal{W}_x\sin\theta}
			{D_{\ref{statement:Gaussian_steep_ascent_contour}}
			(q^j,\Speed(y),\upomega^\circ_{\tau,x},\mathcal{W}_x,\theta)}\,
			P_{\ref{statement:Gaussian_steep_ascent_contour}}\bigl(
				q^j,\Speed(y),\upomega^\circ_{\tau,x},\mathcal{W}_x,\theta
			\bigr)
		\Biggr)
		dy,
	\end{equation*}
	where 
	$D_{\ref{statement:Gaussian_steep_ascent_contour}}(\cdots)<0$, 
	and 
	$P_{\ref{statement:Gaussian_steep_ascent_contour}}$ 
	has the form
	\begin{multline*}
		P_{\ref{statement:Gaussian_steep_ascent_contour}}(Q,\xi,\upomega,\mathcal{W},\theta)
		=
		-
		Q^4 \upomega^3 \mathcal{W}^2
		+
		Q^4 \upomega \mathcal{W}^4
		+
		3 \xi  Q^3 \upomega^2 \mathcal{W}^2
		+
		\xi  Q^3\mathcal{W}^4
		+
		\xi ^2 Q^2 \upomega^3
		\\+
		\xi ^2 Q^2 \upomega \mathcal{W}^2
		-
		3 \xi ^3 Q \upomega^2
		+
		5 \xi^3 Q \mathcal{W}^2
		+
		4 \xi ^4 \upomega
		\\
		-
		4 \xi  \mathcal{W} (\xi +Q \upomega) \left(\xi ^2+Q^2 \mathcal{W}^2\right)
		\cos\theta
		+
		2 \xi ^2 Q \mathcal{W}^2 (\xi +Q \upomega)\cos(2\theta).
	\end{multline*}
	It suffices to show that
	$P_{\ref{statement:Gaussian_steep_ascent_contour}}$ is positive for all $\theta$
	(due to the factor $\sin\theta$ in front of $P_{\ref{statement:Gaussian_steep_ascent_contour}}$
	which changes sign in the lower half plane).
	Viewing $P_{\ref{statement:Gaussian_steep_ascent_contour}}$ as a quadratic polynomial
	$\mathsf{a}U^2 + \mathsf{b}U + \mathsf{c}$
	in 
	$U=\cos\theta$,
	one can check that $\mathsf{a}>0$, 
	that the polynomial is positive for
	$U=1$ and $U=-1$ 
	(here one should use 
	$\xi>Q\upomega>Q\mathcal{W}$), 
	and that 
	the half-sum of its roots is
	$-\frac{\mathsf{b}}{2\mathsf{a}}
	=
	\frac{Q \mathcal{W}}{2 \xi }+\frac{\xi }{2 Q \mathcal{W}}
	\ge 1$.
	This implies that the polynomial is positive for
	$U\in[-1,1]$, and completes the proof.
\end{proof}

\subsection{Contour deformations and extra residues}
\label{sub:extra_residues}

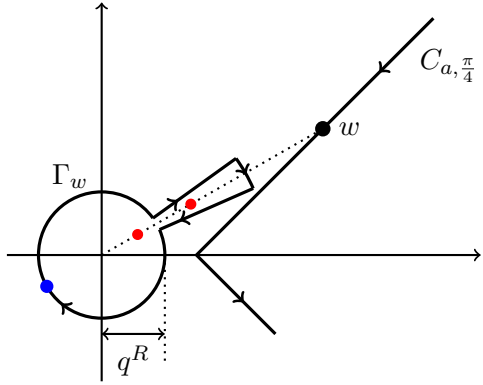
\begin{figure}[htbp]
	\centering
	\begin{tikzpicture}
		[scale=4.2, thick]
		\def\pt{0.02}
		\def\q{.61}
		\def\ss{.56}
		\draw[->] (-.3,0)--(1.2,0);
		\draw[->] (0,-.4)--(0,.8);
		\draw[
			very thick,
			decoration={markings,mark=at position 0.8 with {\arrow{<}}},
			postaction={decorate}
		] (.3,0)--++(.75,.75);
		\draw[
			very thick,
			decoration={markings,mark=at position 0.6 with {\arrow{>}}},
			postaction={decorate}
		] (.3,0)--++(.25,-.25);
		\node at (1.1,.6) {$C_{a,\frac{\pi}{4}}$};
		\draw[fill] (.3+.4,.4) circle (.6pt) node (w) [right,xshift=2pt] {$w$};
		\draw[dotted] (0,0) -- (.7,.4);
		\def\angle{6};
		\def\qqq{.65};
		\draw[
			very thick,
			decoration={markings,mark=at position 0.1 with {\arrow{>}}},
			postaction={decorate}
		] (.7*\qqq,.4*\qqq) arc (29.7449:29.7449-\angle:0.524047);
		\draw[very thick] (.7*\qqq,.4*\qqq) arc (29.7449:29.7449+\angle:0.524047);
		\draw[very thick] (.2,0) arc (0:29.7449-\angle:.2);
		\draw[
			very thick,
			decoration={
				markings,
				mark=at position 0.4 with {\arrow{>}}
			},
			postaction={decorate}
		] (.2,0) arc (0:-360+29.7449+\angle:.2);
		\draw[
			very thick,
			decoration={markings,mark=at position 0.3 with {\arrow{>}}},
			postaction={decorate}
		] (0.162325, 0.116835) -- (0.42533, 0.306136);
		\draw[
			very thick,
			decoration={markings,mark=at position 0.3 with {\arrow{<}}},
			postaction={decorate}
		] (0.183069, 0.080533) -- (0.479685, 0.211015);
		\node at (-.1,.25) {$\Gamma_w$};
		\draw[dotted] (.2,0) --++(0,-.35);
		\draw[<->] (0,-.25)--++(.2,0) node [midway, yshift=-10pt] {$q^R$};
		\def\qqqz{.635};
		\draw[fill, color=red] (.7*\qqqz*\qqqz,.4*\qqqz*\qqqz) circle (0.4pt);
		\draw[fill, color=red]
		(.7*\qqqz*\qqqz*\qqqz*\qqqz,.4*\qqqz*\qqqz*\qqqz*\qqqz) circle (0.4pt);
		\draw[fill,color=blue] (-0.173649, -0.0992278) circle(.5pt);
	\end{tikzpicture}
	\caption{%
		The contour 
		$C_{a,\frac{\pi}{4}}\ni w$ 
		and the 
		contour
		$\Gamma_w$ 
		which is the image of
		$D_{w}^{(0)}\ni\mathsf{u}$
		under the change of variables
		$z=q^{\mathsf{u}}w$.
		The red dots inside $\Gamma_w$ are $qw$ and $q^2w$. 
		The blue dot on $\Gamma_w$ corresponds to the branch cut of $\log(z/w)$.
	}
	\label{fig:Gamma_0_contour}
\end{figure}

Our next goal is to understand the asymptotic behavior of the whole
kernel 
$K_{\zeta}(w,w')$ \eqref{Fredholm_for_exponential_model_section2_main_formula_kernel}
with $\zeta$ given by \eqref{zeta_depends_on_lambda_for_asymptotics}.
To this end, let us split the integration contour 
$D_w$ (described in \Cref{def:d_w_contour_definition})
into parts
\begin{equation*}
	D^{(k)}_w
	:=
	\left\{ 
		\mathsf{u}\in D_w\colon 
		\frac{\pi}{\log q^{-1}}(2k-1)
		\le
		\Im \mathsf{u}
		<
		\frac{\pi}{\log q^{-1}}(2k+1)
	\right\}
	,\qquad 
	k\in\mathbb{Z}.
\end{equation*}
By taking $d=d(w)$ in $D_w$ smaller if needed
we can make sure that 
$\Re\mathsf{u}=R(w)$
for 
$k\ne 0$
and all 
$\mathsf{u}\in D_w^{(k)}$.
(While in 
\Cref{def:d_w_contour_definition}
the parameters
$R(w)$ and $d(w)$ in $D_w$
depend on $w$, 
the upper bound on $d(w)$ required for
the latter condition $\Re \mathsf{u}=R(w)$
can be taken independent of $w$.)
Later we will see that only the 
contribution from the part $D_w^{(0)}$
matters in the $\RateLambda\to+\infty$ limit.

On each $D_w^{(k)}$ the map 
$\mathsf{u}\mapsto q^{\mathsf{u}}$ 
is one to one, so we can change 
the variables as $z=q^{\mathsf{u}}w$, or
\begin{equation*}
	\mathsf{u}
	=
	\frac{\log(z/w)}{\log q}+\frac{2\pi\mathbf{i}k}{\log q},
\end{equation*}
where $\log(z/w)$ is the standard branch of the 
logarithm with argument $\in(-\pi,\pi)$, 
so thus defined $\mathsf{u}$ belongs to $D_w^{(k)}$.
The variable $z$ is integrated over a circle of radius $q^{R(w)}|w|$
for $k\ne 0$ or, for $k=0$, over a more complicated contour 
which we denote by $\Gamma_w$, see \Cref{fig:Gamma_0_contour}.

With this splitting of the $\mathsf{u}$ contour 
and under the above change of variables
the kernel 
$K_\zeta(w,w')$, $w,w'\in C_{a,\frac{\pi}{4}}$ 
(with $0<a<\mathcal{W}_x$),
becomes
\begin{multline}
	\label{Fredholm_kernel_periodic_expansion_section44}
	K_{\zeta(\RateLambda)}(w,w')
	=
	-
	\frac{1}{2\mathbf{i}}
	\int_{\Gamma_w}
	\frac{e^{(- \RateLambda \LimitShape(\tau,x)+r \RateLambda^{\beta})\log(z/w)}}
	{\sin\bigl(\pi\log(z/w)/\log q\bigr)}
	\frac{g(w)}{g(z)}\frac{dz}{(z-w')z\log q}
	\\-
	\sum_{k\in\mathbb{Z}\setminus\left\{ 0 \right\}}
	\frac{1}{2\mathbf{i}}
	\int_{|z|=q^{R(w)}|w|}
	\frac{e^{(- \RateLambda \LimitShape(\tau,x)+r \RateLambda^{\beta})
	(\log(z/w)+2\pi\mathbf{i}k)}}
	{\sin \left(\pi
		\left( 
			\frac{\log(z/w)}{\log q}+
			\frac{2\pi\mathbf{i}k}{\log q}
		\right)\right)
	}
	\frac{g(w)}{g(z)}\frac{dz}{(z-w')z\log q},
\end{multline}
where 
$(-\zeta)^{\mathsf{u}}g(w)/g(z)
=
e^{(- \RateLambda \LimitShape(\tau,x)+r \RateLambda^{\beta})\log(z/w)}g(w)/g(z)$
is given by \eqref{plugging_into_Fredholm_for_asymptotics}.

Take $a=\upomega_{\mathsf{cr}}$ \eqref{w_cr_notation} 
in the contour $C_{a,\frac{\pi}{4}}\ni w,w'$. 
Then by
\Cref{statement:TW_contour_C_is_steep_descent,statement:Gaussian_contour_C_is_steep_descent}
this contour is a steep descent one.
In the Gaussian phase or at the transition point
(when $\upomega_{\mathsf{cr}}=\mathcal{W}_x$)
we need to modify the contour
in a small neighborhood of $\mathcal{W}_x$ to avoid the pole at this point 
(%
	i.e., the contour will pass slightly to the left of $\mathcal{W}_x$, see
	\Cref{sub:Gaussian_contribution,sub:BBP_contribution} below for details on
	local structure of contours%
).
Let us keep the same notation 
$C_{\upomega_{\mathsf{cr}},\frac{\pi}{4}}$ 
for this modified contour.
This choice of the steep descent contour
does not change 
$\det(1+K_\zeta)_{L^2(C_{\upomega_{\mathsf{cr}},\frac{\pi}{4}})}$.

Next, we aim to deform the contour 
$\Gamma_w$ for $z$
to the steep ascent contour 
$\Gamma_{\mathsf{cr}}$
(%
	again, we need to locally modify the latter contour in a small neighborhood
	of the critical point to avoid the intersection with
	$C_{\upomega_{\mathsf{cr}},\frac{\pi}{4}}$, see
	\Cref{sub:TW_contribution,sub:Gaussian_contribution,sub:BBP_contribution}
	below%
).
When $k\ne0$ in \eqref{Fredholm_kernel_periodic_expansion_section44},
this deformation does not encounter any poles. 
However, for $k=0$ such a deformation may pass through
a pole coming from the sine in the denominator. 
These poles have the form 
$z=q^{n}w$, $n\in\mathbb{Z}$, 
but for fixed $w$ one can 
encounter only a finite (logarithmic in $|w|$) 
number of poles corresponding to 
$n=1,2,\ldots N_w$ 
(for some $N_w\in\mathbb{Z}_{\ge0}$; if $N_w=0$ then it means that the deformation
encounters no poles).
Taking the residue at $z=q^nw$ makes the terms under the integral which may
grow exponentially in $\RateLambda$ look as 
$\exp\bigl(
	\RateLambda(G_{\tau,x}(w)-G_{\tau,x}(q^nw))
\bigr)$.
Comparing 
$\Re G_{\tau,x}(w)$
and 
$\Re G_{\tau,x}(q^nw)$
by moving along the contours similarly to
\cite{barraquand2015phase}
we will show that these extra residues are asymptotically negligible:

\begin{proposition}
	\label{prop:extra_negligible}
	For all
	$w\in C_{\upomega_{\mathsf{cr}},\frac{\pi}{4}}$ 
	and 
	$n=1,\ldots,N_w $ we have
	\begin{equation}\label{proof_Fredholm_does_not_change_if_Gamma_0}
		\Re(G_{\tau,x}(w)-G_{\tau,x}(q^nw))<-\delta_1|w|-\delta_2,
	\end{equation}
	where $\delta_1,\delta_2>0$ do not depend on $w$ and $\RateLambda$.
\end{proposition}

\Cref{prop:extra_negligible} will imply that the deformations
of the contours explained above do not affect
the asymptotics of the Fredholm determinant
(a detailed statement is \Cref{statement:Fredholm_does_not_change_if_Gamma_0} below).
The proof of \Cref{prop:extra_negligible} is based
on \Cref{statement:linear_part,statement:circle_part}
which we establish first.
Before discussing these lemmas, observe that
it suffices to consider only the case $\Im w>0$. 
Indeed, the case $\Im w<0$ is symmetric, 
and $\Im w=0$ corresponds to $w=\upomega_{\mathsf{cr}}$, in which case 
all points of the form $q^n \upomega_{\mathsf{cr}}$, $n\in \mathbb{Z}_{\ge1}$,
are inside $\Gamma_{\mathsf{cr}}$ and thus are not encountered by our contour
deformation.

\begin{figure}[htbp]
	\centering
	\begin{tikzpicture}
		[scale=4.2, thick]
		\def\pt{0.02}
		\def\q{.61}
		\def\ss{.56}
		\draw[->] (-.3,0)--(1.2,0);
		\draw[->] (0,-.4)--(0,.8);
		\draw[
			decoration={markings,mark=at position 0.2 with {\arrow{<}}},
			postaction={decorate}
		] (.5,0)--++(.95,.95);
		\draw[
			decoration={markings,mark=at position 0.6 with {\arrow{>}}},
			postaction={decorate}
		] (.5,0)--++(.25,-.25);
		\node at (.86,.17) {$C_{\upomega_{\mathsf{cr}},\frac{\pi}{4}}$};
		\draw[dashed] (.5,-.2) --++ (0,1);
		\draw[dashed] (.5,0) arc (0:90:.5);
		\draw[ultra thick] (.5,.2) arc (21.8014:32.5604:0.538516);
		\draw[dotted] (0,0) -- (.5,.2);
		\draw[ultra thick] (.5,.2) --++ (.5*2/3,.2*2/3);
		\draw[ultra thick] (.5*5/3,.2*5/3) --++(.55,.55);
		\draw[fill] (.5*5/3+.55,.2*5/3+.55) circle(.4pt) node [right,xshift=1pt] {$w$};
		\draw[dotted] (0,0)--(.5*5/3+.55,.2*5/3+.55);
		\draw[fill,opacity=.1] (.5,0) arc (0:32.5604:.5) -- (.5*5/3+.55,.2*5/3+.55) -- cycle;
		\node at (1.12,.5) {$V_1$};
		\node at (.7,.354) {$V_2$};
		\node at (.44,.37) {$V_3$};
		\node at (.43,-.07) {$\upomega_{\mathsf{cr}}$};
		\draw[fill] (.5,0) circle (.6pt);
	\end{tikzpicture}
	\caption{%
		For any $w\in C_{\upomega_{\mathsf{cr}},\frac{\pi}{4}}$
		and any $\tilde w$ belonging to the (closed) shaded region 
		we have $\Re G_{\tau,x}(w)<\Re G_{\tau,x}(\tilde w)$.
		Indeed, any such $\tilde w$ can be reached from $w$ by moving 
		along a combination of contours $V_1$ (part of $C_{\upomega_{\mathsf{cr}},\frac{\pi}{4}}$), 
		$V_2$ (part of a line passing through the origin 
		in the half plane $\Re \tilde w>\Re \upomega_{\mathsf{cr}}$), and $V_3$
		(arc of a circle centered at the origin
		in the half plane $\Re \tilde w<\Re \upomega_{\mathsf{cr}}$).
		The increase of $\Re G_{\tau,x}$ in the direction from $w$ to $\tilde w$ 
		along these contours follows from
		\Cref{statement:TW_contour_C_is_steep_descent,statement:Gaussian_contour_C_is_steep_descent}
		(for~$V_1$),
		\Cref{statement:linear_part}
		(for~$V_2$),
		and
		\Cref{statement:circle_part}
		(for~$V_3$).
		The dashed arc centered at the origin is a part of the 
		steep ascent contour $\Gamma_{\mathsf{cr}}$.%
	}
	\label{fig:contours_extra_residues}
\end{figure}
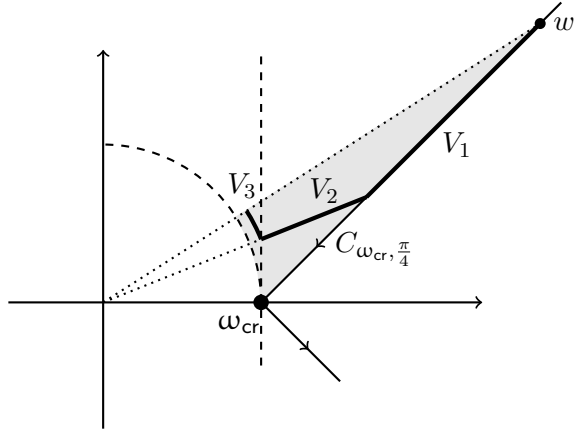

\begin{lemma}
	\label{statement:circle_part}
	On any contour of the form
	$w=(\upomega_{\mathsf{cr}}+\mathbf{i}v)e^{\mathbf{i}\theta}$,
	where $v\in(0,\upomega_{\mathsf{cr}})$ 
	is fixed and $\theta$ increases from
	$0$ to $\frac{\pi}{4}$
	the function $w\mapsto \Re G_{\tau,x}(w)$ is increasing in $\theta$.
\end{lemma}
See the contour $V_3$ in \Cref{fig:contours_extra_residues};
note that $V_3$ is a part of the larger contour corresponding to $\theta\in(0,\frac{\pi}{4})$
considered in \Cref{statement:circle_part}.
\begin{proof}[Proof of \Cref{statement:circle_part}]
	Conditions $v<\upomega_{\mathsf{cr}}$ and $0<\theta<\frac{\pi}{4}$ reflect the geometry
	of the contour.

	In the Tracy--Widom phase (where $\upomega_{\mathsf{cr}}=\upomega^\circ_{\tau,x}$)
	we have 
	\begin{equation*}
		\frac{\partial}{\partial\theta}\Re G_{\tau,x}
		\bigl((
			\upomega^\circ_{\tau,x}+\mathbf{i}v)e^{\mathbf{i}\theta}
		\bigr)
		=
		\int_0^x
		\Bigg(
			\sum_{j=0}^{\infty}
			\frac{-q^{2j} \Speed(y) (v \cos\theta + \upomega^\circ_{\tau,x}\sin\theta) }
			{D_{\ref{statement:circle_part}}(q^j,\Speed(y),\upomega^\circ_{\tau,x},v,\theta)}
			\,
			P_{\ref{statement:circle_part}}(q^j,\Speed(y),\upomega^\circ_{\tau,x},v,\theta)
		\Biggr)dy,
	\end{equation*}
	where 
	$D_{\ref{statement:circle_part}}(\cdots)<0$, 
	and we want to show that the polynomial
	$P_{\ref{statement:circle_part}}$ is positive.
	It has the form
	\begin{equation*}
		P_{\ref{statement:circle_part}}
		(Q,\xi,\upomega,v,\theta)
		=
		R_0
		+
		R_1\cos\theta
		+
		R_2\cos(2\theta)
		+
		R_3\sin \theta
		+
		R_4\sin (2\theta),
	\end{equation*}
	where $R_i=R_i(Q,\xi,\upomega,v)$ are the following polynomials:
	\begin{align*}
		R_0&=
		Q^4 v^2 \upomega \left(v^2+\upomega^2\right)+\xi  Q^3 \left(v^4+5 v^2 \upomega^2+4
		\upomega^4\right)
		\\&\hspace{130pt}
		+\xi ^2 Q^2 \upomega \left(v^2+2 \upomega^2\right)+\xi ^3 Q \left(5 v^2+2
		\upomega^2\right)+4 \xi ^4 \upomega,
		\\
		R_1&=
		-4 \xi  \upomega (\xi +Q \upomega) \left(\xi ^2+
			Q^2v^2+Q^2\upomega^2\right),
		\\
		R_2&=
		-2 \xi ^2 Q \left(v^2-\upomega^2\right) (\xi +Q \upomega),
		\\
		R_3&=
		4
		\xi  v (\xi +Q \upomega) \left(\xi ^2+Q^2 v^2+Q^2\upomega^2\right),
		\\
		R_4&=
		-4 \xi ^2 Q v
		\upomega (\xi +Q \upomega).
	\end{align*}
	One can readily check that these polynomials satisfy
	(for our range of parameters)
	\begin{equation*}
		R_0>0,\qquad 
		R_1<0,\qquad 
		R_2>0,\qquad 
		R_3>0,\qquad 
		R_4<0,
	\end{equation*}
	and
	\begin{equation*}
		R_3
		+
		2R_4
		=
		4 \xi  v (\xi +Q \upomega) \left(\xi ^2+Q^2 \left(v^2+\upomega^2\right)-2 \xi  Q \upomega\right)
		>0,
	\end{equation*}
	which implies that
	\begin{equation*}
		R_3\sin \theta+R_4\sin(2\theta)
		=
		\left( R_3+2R_4\cos\theta \right)\sin\theta>0.
	\end{equation*}
	Moreover, after substituting $\xi=\Omega+Q\upomega$ 
	with $\Omega>0$ (since $Q\upomega<\xi$), we have
	\begin{multline*}
		7R_0+5R_1
		=
		8 \Omega^4 \upomega+35 \Omega^3 Q v^2+26 \Omega^3 Q \upomega^2+92 \Omega^2 Q^2 v^2
		\upomega+24 \Omega^2 Q^2 \upomega^3
		+7 \Omega Q^3 v^4
		\\
		+94 \Omega Q^3 v^2 \upomega^2+10
		\Omega Q^3 \upomega^4+14 Q^4 v^4 \upomega+44 Q^4 v^2 \upomega^3+4 Q^4 \upomega^5
		>0,
	\end{multline*}
	so
	\begin{equation*}
		7(R_0+R_1\cos\theta)\ge
		7R_0+\frac{7}{\sqrt2}R_1>
		\left( -5+\frac{7}{\sqrt 2} \right)R_1>0,
	\end{equation*}
	which establishes the claim in the Tracy--Widom phase.
	
	The proof in the Gaussian phase
	(when $\upomega_{\mathsf{cr}}=\mathcal{W}_x\le \upomega^\circ_{\tau,x}$)
	is similar to how \Cref{statement:Gaussian_contour_C_is_steep_descent}
	was reduced to \Cref{statement:TW_contour_C_is_steep_descent}.
	Namely, the function $w\mapsto \Re(\mathcal{W}_x\log w-w)$
	for $w=(\mathcal{W}_x+\mathbf{i}v)e^{\mathbf{i}\theta}$
	is increasing in $\theta$, and so to show that
	the function $G_{\tau,x}$ \eqref{G_of_w_in_Gaussian_phase_using_Phi}
	is increasing on the desired contour
	is suffices to replace $\upomega^\circ_{\tau,x}$ by $\mathcal{W}_x$ in the first summand in 
	\eqref{G_of_w_in_Gaussian_phase_using_Phi} due to 
	\eqref{PHI2_increasing_for_proof_of_descent}.
	The statement for the modified function $\tilde G_{\tau,x}$
	(i.e., with $\upomega^\circ_{\tau,x}$ replaced by $\mathcal{W}_x$)
	is the same as for the Tracy--Widom phase with a different 
	time $\tilde\tau=\Phi_2(\mathcal{W}_x\mid x)/\mathcal{W}_x$.
	This completes the proof.
\end{proof}

\begin{lemma}
	\label{statement:linear_part}
	On any contour of the form $w=s e^{\mathbf{i}\theta}$, where $\theta\in(0,\frac{\pi}{4})$
	is fixed and $s$ increases from 
	$\frac{\upomega_{\mathsf{cr}}}{\cos\theta}$ 
	to
	$\frac{\upomega_{\mathsf{cr}}}{\cos\theta-\sin\theta}$,
	the function $w\mapsto\Re G_{\tau,x}(w)$
	is decreasing in $s$.
\end{lemma}
The contour in \Cref{statement:linear_part} 
is exactly the contour $V_2$ in \Cref{fig:contours_extra_residues},
where $\theta$ corresponds to the angle between $V_2$ and the real line.
\begin{proof}[Proof of \Cref{statement:linear_part}]
	Again, conditions $0<\theta<\frac{\pi}{4}$ and 
	$\frac{\upomega_{\mathsf{cr}}}{\cos\theta}
	<s<
	\frac{\upomega_{\mathsf{cr}}}{\cos\theta-\sin\theta}$
	reflect the geometry of the contour.

	Consider the Tracy--Widom phase, so $\upomega_{\mathsf{cr}}=\upomega^\circ_{\tau,x}$. 
	We have
	\begin{equation*}
		\frac{\partial}{\partial s}\Re G_{\tau,x}(se^{\mathbf{i}\theta})
		=
		\int_0^x
		\Bigg(
			\sum_{j=0}^{\infty}
			\frac{
				q^{2j}\Speed(y)
			}
			{D_{\ref{statement:linear_part}}(q^j,\Speed(y),\upomega^\circ_{\tau,x},s,\theta)}
			\,
			P_{\ref{statement:linear_part}}(q^j,\Speed(y),\upomega^\circ_{\tau,x},s,\theta)
		\Biggr)
		dy,
	\end{equation*}
	where
	$D_{\ref{statement:linear_part}}(\cdots)<0$
	and 
	we would like to show that the polynomial
	$P_{\ref{statement:linear_part}}(Q,\xi,\upomega,s,\theta)$ is positive
	for our range of parameters.
	To see this, change the variables as
	\begin{equation*}
		s=\frac{\upomega}{\cos\theta(1-V\tan\theta)}, \qquad 0<V<1,
	\end{equation*}
	and let $\theta=\arctan U$ with $0<U<1$.
	With these substitutions we have
	\begin{equation}
	\label{big_3var_polynomial}
		\begin{split}
			&\frac{(1-U V)^5}{U^2 \upomega^2 \xi^4}\,
			P_{\ref{statement:linear_part}}
			(Q,\xi,\upomega,s(U,V),\theta(U))
			\\&\hspace{40pt}=
			Q^4\upomega^4\xi^{-4} \left(U^2+1\right)  (2 U V+V^2-1 )
			+
			 Q^3 \upomega^3\xi^{-3} (1+U^2) 
			(2 U V^3-V^2+1 )
			\\&\hspace{60pt}+
			 Q^2 \upomega^2\xi^{-2} \left(-2 U^3 V^3+U^2 (V^2+11 ) V^2
			+2 U (V^2-6 ) V-3 V^2+3\right)
			\\&\hspace{60pt}+
			Q \upomega\xi^{-1} (1-U V)^2 (6 U V+5 V^2-5 )
			+
			2 (1-V^2) (1-U V)^3.
		\end{split}
	\end{equation}
	Denote $T:=Q\upomega/\xi$, and observe that $0<T<1$.
	The right-hand side of \eqref{big_3var_polynomial}
	becomes a polynomial in $T,U,V\in(0,1)$,
	denote it by $\mathscr{P}(T,U,V)$.
	One can see that $\mathscr{P}$ is cubic in $U$.
	Its discriminant in $U$ is
	\begin{multline*}
		-4 (1-T)^6 T^6 
		(1+V^2 )^3
		\Big(T^4 (1+V^2)+2 T^3 (1+9 V^2)
			\\+
			T^2
			(11 V^2+59 ) V^2+24 T (3+V^2 ) V^2-V^6+10 V^4+27 V^2 
		\Big)
		<0,
	\end{multline*}
	hence $\mathscr{P}$ has one real root in $U$.
	We have
	\begin{equation*}
		\mathscr{P}(T,0,V)=
		(1-T)^3 (T+2) (1-V^2 )>0,
	\end{equation*}
	Thus, it suffices to show that 
	$\mathscr{P}(T,1,V)>0$. 
	We have
	\begin{multline}
		\mathscr{P}(T,1,1-V)
		=
		4 T^3 (T+1)
		-
		8 T^2 (T^2+T+1 ) V
		+
		2 T (T+1)^2(T+3) V^2
		\\-
		4 T (T^2+T+4 ) V^3
		+
		(T+1) (T+4) V^4
		-
		2 V^5.
		\label{2var_positive_polynomial}
	\end{multline}
	Let us minimize \eqref{2var_positive_polynomial}
	in $(T,V)\in[0,1]^2$ and show that the minimum is positive.
	Solving 
	$\frac{\partial}{\partial T}\mathscr{P}(T,1,1-V)
	=
	\frac{\partial}{\partial V}\mathscr{P}(T,1,1-V)=0$
	numerically, we see that the only critical points 
	of \eqref{2var_positive_polynomial} on $[0,1]^2$ are
	$(T,V)=(0,0)$ and
	$(T,V)\approx(0.646, 0.829)$, and the values
	of
	$\mathscr{P}(T,1,1-V)$
	at both points are nonnegative.
	The critical point inside $(0,1)^2$ is a saddle,
	so the polynomial attains its minimum on the boundary.
	Further looking at the univariate polynomials on the boundary
	one can readily check 
	that the minimum of \eqref{2var_positive_polynomial}
	on $[0,1]^2$
	is 
	$\mathscr{P}(0,1,1)=0$, 
	which shows that 
	$P_{\ref{statement:linear_part}}$ is positive, as desired.
	
	Now assume that the space-time point $(\tau,x)$ is in the Gaussian phase. 
	Observe that
	the function $w\mapsto \Re(\mathcal{W}_x\log w-w)$
	decreases along the contour 
	$\{s e^{\mathbf{i}\theta}\colon
	s>\frac{\mathcal{W}_x}{\cos\theta}\}$ 
	(with $0<\theta<\frac{\pi}{4}$ fixed).
	Thus, it suffices to replace $\upomega^\circ_{\tau,x}$ 
	by $\mathcal{W}_x$ in the first summand in 
	\eqref{G_of_w_in_Gaussian_phase_using_Phi} due to 
	\eqref{PHI2_increasing_for_proof_of_descent},
	and prove the statement for the resulting 
	modified function $\tilde G_{\tau,x}$.
	Taking 
	$\tilde\tau=\Phi_2(\mathcal{W}_x\mid x)/\mathcal{W}_x$
	we have $\tilde G_{\tau,x}=G_{\tilde\tau,x}$,
	and for the latter function the desired statement follows
	from the Tracy--Widom case just established.
	This completes the proof.
\end{proof}

\begin{proof}[Proof of \Cref{prop:extra_negligible}]
	We need to show that 
	$\Re(G_{\tau,x}(w)-G_{\tau,x}(q^nw))<-\delta_1|w|-\delta_2$
	for all
	$w\in C_{\upomega_{\mathsf{cr}},\frac{\pi}{4}}$ 
	and 
	$n=1,\ldots,N_w $.
	First, observe that the distance from $w$ to every $q^nw$, $n=1,\ldots,N_w$, 
	along the contours
	$V_1,V_2$, and $V_3$ as in \Cref{fig:contours_extra_residues}
	is bounded from below uniformly in $|w|$
	because $N_w=0$ for $w$ close to $\upomega_{\mathsf{cr}}$.
	For $|w|$ bounded from above by a constant independent of $\RateLambda$
	the 
	bounds on the derivative of $\Re G_{\tau,x}$
	along the contours $V_i$ in 
	\Cref{statement:circle_part,statement:linear_part}
	can be made uniform in $|w|$, which leads to
	\eqref{proof_Fredholm_does_not_change_if_Gamma_0} without $\delta_1|w|$.
	However, if $|w|$ is bounded from above then this implies the full desired estimate
	\eqref{proof_Fredholm_does_not_change_if_Gamma_0}.
	
	For large $|w|$ the path as in \Cref{fig:contours_extra_residues} 
	from $w$ to any of the points 
	$q^nw$, $n=1,\ldots, N_w$ 
	has length of order $|w|$.
	One can readily check that 
	both derivatives 
	\begin{equation*}
		\frac{\partial}{\partial s}\Re G_{\tau,x}(\upomega_{\mathsf{cr}}+s e^{\mathbf{i}\frac{\pi}{4}})
		\qquad \textnormal{and}\qquad 
		\frac{\partial}{\partial s}\Re G_{\tau,x}(s e^{\mathbf{i}\theta})
	\end{equation*}
	along contours $V_1$ and $V_2$, respectively, have strictly negative limits as $s\to\infty$.
	Together with \Cref{statement:linear_part,statement:circle_part}
	this implies \eqref{proof_Fredholm_does_not_change_if_Gamma_0} for large $|w|$,
	and thus completes the proof.
\end{proof}

We need one more definition to formulate 
the main result of this subsection:

\begin{definition}\label{def:K_steep}
	Let $K_{\zeta}^{\mathrm{steep}}(w,w')$ stand for the kernel 
	as in \eqref{Fredholm_kernel_periodic_expansion_section44}
	but with the $z$ integration contours 
	replaced by $\Gamma_{\mathsf{cr}}$ for all $k\in\mathbb{Z}$
	(%
		recall that the latter is the clockwise
		oriented circle centered at $0$ with radius $\upomega_{\mathsf{cr}}$
		modified in a neighborhood of $\upomega_{\mathsf{cr}}$ to avoid 
		poles, see \Cref{sub:TW_contribution,sub:Gaussian_contribution,sub:BBP_contribution}
		below for details%
	).
\end{definition}

\begin{proposition}
	\label{statement:Fredholm_does_not_change_if_Gamma_0}
	With the above notation and with $\zeta=\zeta(\RateLambda)$
	given by \eqref{zeta_depends_on_lambda_for_asymptotics} we have
	\begin{equation*}
		\lim_{\RateLambda\to+\infty}
		\det\bigl(1+K_{\zeta(\RateLambda)}\bigr)_{L^2(C_{\upomega_{\mathsf{cr}},\frac{\pi}{4}})}
		=
		\lim_{\RateLambda\to+\infty}
		\det\bigl(1+K^{\mathrm{steep}}_{\zeta(\RateLambda)}\bigr)_{L^2(C_{\upomega_{\mathsf{cr}},\frac{\pi}{4}})}
	\end{equation*}
	The above equality
	is understood in the sense that
	if one of the limits exists, then the other one also exists
	and they are equal to each other.
\end{proposition}
\begin{proof}
	We have
	\begin{equation*}
		K_{\zeta(\RateLambda)}(w,w')
		=
		K_{\zeta(\RateLambda)}^{\mathrm{steep}}(w,w')
		+
		R(w,w'),
	\end{equation*}
	where $R(w,w')$ is the sum of residues corresponding to the first integral
	in \eqref{Fredholm_kernel_periodic_expansion_section44}
	at $z=q^nw$, $n=1,\ldots,N_w$ (recall that $N_w$ is of order $\log |w|$).
	As will follow from the analysis in a small neighborhood
	of $\upomega_{\mathsf{cr}}$ in 
	\Cref{sub:TW_contribution,sub:Gaussian_contribution,sub:BBP_contribution}
	below, 
	the terms corresponding to 
	$\frac{1}{q^{\mathsf{u}}w-w'}$ under the integral in the kernel
	can be bounded from above by a power of $\RateLambda$, 
	uniformly in $q^{\mathsf{u}}w$ and $w'$.
	Indeed, this is because the steep descent and ascent integration contours
	$C_{\upomega_{\mathsf{cr}},\frac{\pi}{4}}$ and $\Gamma_{\mathsf{cr}}$, respectively,
	are going to be separated by a distance of order $\RateLambda^{-\delta}$ for
	some $\delta>0$.
	Thus, to get the desired claim it suffices to show that 
	$\Bigl|\int_{C_{\upomega_{\mathsf{cr}},\frac{\pi}{4}}}R(w,w')dw\Bigr|$ 
	decays exponentially in $\RateLambda$.
	Having estimate \eqref{proof_Fredholm_does_not_change_if_Gamma_0}
	of \Cref{prop:extra_negligible},
	we can write
	\begin{equation*}
		\biggl|\int_{C_{\upomega_{\mathsf{cr}},\frac{\pi}{4}}}R(w,w')dw\biggr|
		\le
		\RateLambda^{c_1}
		\int_{w\in C_{\upomega_{\mathsf{cr}},\frac{\pi}{4}},
		\; |\Im w| >c_2}
		e^{- \RateLambda \delta_1|w|}\log|w|
		dw
	\end{equation*}
	for some $c_1,c_2>0$ independent of $\RateLambda$.
	Here the condition $|\Im w|>c_2$ arises from the fact that for 
	$w$ close to $\upomega_{\mathsf{cr}}$ 
	(up to distance
	which depends on $q,\Speed,\tau,x$ but not on $\RateLambda$) 
	we have $N_w=0$
	and so no residues are picked during the contour deformation.
	Thus, the integral of $R$ decays exponentially in $\RateLambda$, 
	which yields the statement.
\end{proof}

\Cref{statement:Fredholm_does_not_change_if_Gamma_0} and results of
\Cref{sub:steep_descent_contours} on contours $C_{\upomega_{\mathsf{cr}},\frac{\pi}{4}}$
and $\Gamma_{\mathsf{cr}}$ being steep descent and steep ascent, respectively, imply
that the asymptotic behavior of the Fredholm determinant
$\det(1+K_{\zeta(\RateLambda)})_{L^2(C_{\upomega_{\mathsf{cr}},\frac{\pi}{4}})}$ as
$\RateLambda\to+\infty$ is governed by the contribution coming from a small
neighborhood of the critical point $\upomega_{\mathsf{cr}}$.
Therefore, to finish the proof of our main results
(\Cref{thm:main_theorem_on_fluctuations}) it remains to compute the
contributions from a small neighborhood of $\upomega_{\mathsf{cr}}$ in each of the
phases.  
This is performed in
\Cref{sub:TW_contribution,sub:Gaussian_contribution,sub:BBP_contribution}
below.

\subsection{Contribution in the Tracy--Widom phase}
\label{sub:TW_contribution}

In the Tracy--Widom phase 
($\upomega_{\mathsf{cr}}=\upomega^\circ_{\tau,x}<\mathcal{W}_x$) 
we take the power
in \eqref{zeta_depends_on_lambda_for_asymptotics}
to be $\beta=\frac{1}{3}$. 
Then $w=\upomega^\circ_{\tau,x}$ 
is a double critical point of the function
$G_{\tau,x}(w)$. 
\begin{lemma}
	\label{statement:third_derivative_G_critical_point}
	We have 
	$G_{\tau,x}'''(\upomega^\circ_{\tau,x})>0$.
\end{lemma}
\begin{proof}
	From \eqref{G_derivatives_formulas} and 
	\eqref{Equation_for_double_critical_points_1}--\eqref{Equation_for_double_critical_points_2}
	we have
	\begin{equation*}
		(\upomega^\circ_{\tau,x})^3
		G_{\tau,x}'''(\upomega^\circ_{\tau,x})
		=
		\Phi_3(\upomega^\circ_{\tau,x}\mid x)
		-
		\Phi_2(\upomega^\circ_{\tau,x}\mid x),
	\end{equation*}
	which is positive by \eqref{phi_3_minus_phi_2_positive}.
\end{proof}
This implies that locally in a neighborhood of 
$\upomega^\circ_{\tau,x}$
the regions where 
$\Re G_{\tau,x}(w)-\Re G_{\tau,x}(\upomega^\circ_{\tau,x})$
has constant sign look as in 
\Cref{fig:locally_TW}. 
Deform the $z$ and $w$ contours
in the kernel $K_{\zeta(\RateLambda)}(w,w')$ \eqref{Fredholm_kernel_periodic_expansion_section44}
to $\Gamma_{\mathsf{cr}}$ and $C_{\upomega^\circ_{\tau,x},\frac{\pi}{4}}$, respectively,
to get the kernel $K_{\zeta(\RateLambda)}^{\mathrm{steep}}(w,w')$
of \Cref{def:K_steep}.
Here $\Gamma_{\mathsf{cr}}$ is a circle centered at zero with radius $\upomega^\circ_{\tau,x}$
modified in a small neighborhood of $\upomega^\circ_{\tau,x}$
to look as the left contour in \Cref{fig:locally_TW}.
Define 
(recall notation \eqref{Big_Phi_n_definition})
\begin{equation}
	\label{dispersion_TW}
	d^{\mathrm{TW}}_{\tau,x}:=
	\sqrt[3]{
		\frac{G_{\tau,x}'''(\upomega^\circ_{\tau,x})}{2}
	}
	=
	\frac{2^{-\frac{1}{3}}}{\upomega^\circ_{\tau,x}}
	\left( 
		\Phi_3(\upomega^\circ_{\tau,x}\mid x)
		-
		\Phi_2(\upomega^\circ_{\tau,x}\mid x)
	\right)^{\frac{1}{3}}>0
\end{equation}
and make a change of variables 
\begin{equation}\label{TW_change_of_variables}
	w=\upomega^\circ_{\tau,x}+\RateLambda^{-\frac{1}{3}}(d^{\mathrm{TW}}_{\tau,x})^{-1}\tilde w,
	\qquad 
	w'=\upomega^\circ_{\tau,x}+\RateLambda^{-\frac{1}{3}}(d^{\mathrm{TW}}_{\tau,x})^{-1}\tilde w',
	\qquad 
	z=\upomega^\circ_{\tau,x}+\RateLambda^{-\frac{1}{3}}(d^{\mathrm{TW}}_{\tau,x})^{-1}\tilde z,
\end{equation}
where 
$\tilde z\in \gamma_{z}^{\mathrm{Ai}}$
and 
$\tilde w,\tilde w'\in \gamma_{w}^{\mathrm{Ai}}$ 
(these contours are made out of straight half lines 
depicted in \Cref{fig:locally_TW}).
Let us work in the neighborhood of $\upomega^\circ_{\tau,x}$
of size $\RateLambda^{-\frac{1}{6}}$,
so that $|\tilde w|$, $|\tilde w'|$, and $|\tilde z|$ are $O(\RateLambda^{\frac{1}{6}})$.
In this neighborhood the Taylor expansion of $G_{\tau,x}$ has the form
\begin{equation*}
	G_{\tau,x}(w)
	=
	G_{\tau,x}(\upomega^\circ_{\tau,x})
	+\frac{\tilde w^3}{3 \RateLambda}+o(\RateLambda^{-1}).
\end{equation*}

\begin{figure}[htbp]
	\begin{tikzpicture}
		[scale=1]
		\def\sq{0.866025}
		\filldraw[draw=none,right color=black!60!white, left color=white] (-3*\sq,-1.5) -- (0,0) -- (-3*\sq,1.5);
		\filldraw[draw=none,top color=black!60!white, bottom color=white] (0,-3) -- (0,0) -- (3*\sq,-1.5);
		\filldraw[draw=none,bottom color=black!60!white, top color=white] (0,3) -- (0,0) -- (3*\sq,1.5);
		\draw[very thick, ->] (-3,0)--(3.5,0);
		\draw[fill] (0,0) circle (4pt) node (w1) {};
	  	\node (w1l) at (-1.1,3.2) 
	  	[draw=black, line width=.5, rectangle, 
	  	rounded corners, text centered] 
			{$\upomega^\circ_{\tau,x}$};
        \draw[->, line width=.9,densely dashed] (w1l)--(w1);
        \draw (3*\sq,-1.5)--(-3*\sq,1.5);
        \draw (-3*\sq,-1.5)--(3*\sq,1.5);
        \draw (0,-3)--(0,3);
				\node at (1.95,3*\sq) {\Large$\gamma^{\mathrm{Ai}}_{w}$};
				\node at (-2.65,3*\sq) {\Large$\gamma^{\mathrm{Ai}}_{z}$};
				\draw[
					line width=2.2,
					decoration={
						markings,
						mark=at position 0.25 with {\arrow{>}},
						mark=at position 0.75 with {\arrow{>}}
					},
					postaction={decorate}
				] (-2.2,3)--++(1.73205,-3)--(-2.2,-3);
				\draw[
					line width=2.2,
					decoration={
						markings,
						mark=at position 0.25 with {\arrow{>}},
						mark=at position 0.75 with {\arrow{>}}
					},
					postaction={decorate}
				] (3,3)--++(-3,-3)--(3,-3);
	\end{tikzpicture}
	\caption{%
		Behavior in a neighborhood of the double critical point
		$\upomega^\circ_{\tau,x}$ of $G_{\tau,x}$. 
		Shaded are regions where 
		$\Re G_{\tau,x}(z)<\Re G_{\tau,x}(\upomega^\circ_{\tau,x})$. 
		The $z$ and $w$ contours in a neighborhood of $\upomega^\circ_{\tau,x}$ are
		also shown. 
		The contour $\gamma^{\mathrm{Ai}}_{z}$ is shifted in order to avoid the pole
		at $z=w'$ coming from the denominator in \eqref{Fredholm_kernel_periodic_expansion_section44}.%
	}
	\label{fig:locally_TW}
\end{figure}
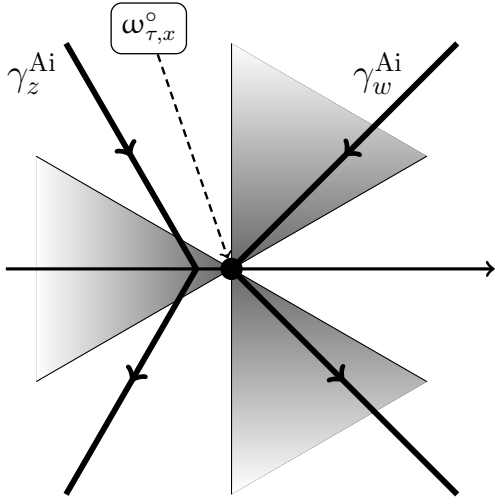

Thus, from 
\eqref{Fredholm_kernel_periodic_expansion_section44} 
and
\eqref{plugging_into_Fredholm_for_asymptotics} 
we obtain the following scaled kernel
which now contains $r\in\mathbb{R}$ as a 
parameter:\footnote{%
	The change of variables in a neighborhood of $\upomega^\circ_{\tau,x}$
	introduces
	a scaling 
	factor $\RateLambda^{-\frac{2}{3}}(d^{\mathrm{TW}}_{\tau,x})^{-2}$
	coming from $dz$ in the kernel itself and from the integrals over $w,w'$,
	cf. \eqref{Fredholm_expansion}.
	This scaling of the kernel is reflected in the notation $\sqrt{dw\,dw'}$
	which we will also use below in similar situations.%
}
\begin{align}
		&K_{\zeta(\RateLambda)}^{\mathrm{steep}}(w,w')\sqrt{dw\,dw'}
		\nonumber\\\nonumber&\hspace{10pt}=
		-
		\frac{1}{2\mathbf{i}}
		\sum_{k\in\mathbb{Z}}
		\int_{\Gamma_{\mathsf{cr}}}
		\frac{e^{(- \RateLambda \LimitShape(\tau,x)+r \RateLambda^{\frac{1}{3}})
		(\log(z/w)+2\pi\mathbf{i}k)}}
		{\sin \left(\pi
			\left( 
				\frac{\log(z/w)}{\log q}+
				\frac{2\pi\mathbf{i}k}{\log q}
			\right)\right)
		}
		\frac{g(w)}{g(z)}\frac{dz}{(z-w')z\log q}
		\sqrt{dw\,dw'}
		\\
		\label{TW_limit_of_periodic_expansion}
		&\hspace{10pt}=
		-
		\frac{\RateLambda^{-\frac{1}{3}}(d^{\mathrm{TW}}_{\tau,x})^{-1}}{2\mathbf{i}}
		\sum_{k\in\mathbb{Z}}
		\int_{\tilde\Gamma_{\mathsf{cr}}}
		\frac{
			e^{
				2\pi\mathbf{i}k
				(- \RateLambda \LimitShape(\tau,x)+r \RateLambda^{\frac{1}{3}})
			}
		}
		{\sin \left(\pi
			\left( 
				\frac{\tilde z-\tilde w}{\upomega^\circ_{\tau,x}d^{\mathrm{TW}}_{\tau,x}\log q}
				\RateLambda^{-\frac{1}{3}}(1+o(1))+
				\frac{2\pi\mathbf{i}k}{\log q}
			\right)\right)
		}\\\nonumber&\hspace{35pt}\times
		\exp\left\{ 
			\frac{\tilde w^3}{3}-\frac{\tilde z^3}{3}
			+\frac{r}{\upomega^\circ_{\tau,x}d^{\mathrm{TW}}_{\tau,x}}(\tilde z-\tilde w)
			+o(1)
		\right\}
		\frac{\bigl(1+o(1)\bigr)d\tilde z}
		{(\tilde z-\tilde w')(\upomega^\circ_{\tau,x}+o(1))\log q}
		\sqrt{d\tilde w\,d\tilde w'}
		\\\nonumber&\hspace{10pt}=:
		\tilde K^{\mathrm{Ai}}_{r}(\tilde w,\tilde w')\sqrt{d\tilde w\,d\tilde w'},
\end{align}
where $\tilde\Gamma_{\mathsf{cr}}$
is the contour $\Gamma_{\mathsf{cr}}$ 
under the above change of variables $z\to\tilde z$.
Here we used the fact that in the Tracy--Widom phase 
the non-exponential prefactor in \eqref{plugging_into_Fredholm_for_asymptotics}
is regular in $w$ and thus behaves as $1+o(1)$ as $\RateLambda\to+\infty$.
When $k\ne 0$, the sine in the denominator is regular:
\begin{equation}
	\label{sine_in_the_denominator_is_regular}
	\sin \left(\pi
		\left( 
			\frac{\tilde z-\tilde w}{\upomega^\circ_{\tau,x}d^{\mathrm{TW}}_{\tau,x}\log q}
			\RateLambda^{-\frac{1}{3}}(1+o(1))+
			\frac{2\pi\mathbf{i}k}{\log q}
		\right)\right)
	=
	\sin\left( \frac{2\pi^2\mathbf{i}k}{\log q} \right)\left( 1+o(1) \right)\ne 0,
\end{equation}
and so all summands in \eqref{TW_limit_of_periodic_expansion} 
with $k\ne 0$ vanish as $\RateLambda\to+\infty$ due to the prefactor $\RateLambda^{-\frac{1}{3}}$.
On the other hand, for $k=0$ the sine behaves as
\begin{equation*}
	\sin \left(\pi\,
		\frac{\tilde z-\tilde w}{\upomega^\circ_{\tau,x}d^{\mathrm{TW}}_{\tau,x}\log q}
		\RateLambda^{-\frac{1}{3}}(1+o(1))
	\right)
	=
	\pi\,
	\frac{\tilde z-\tilde w}{\upomega^\circ_{\tau,x}d^{\mathrm{TW}}_{\tau,x}\log q}
	\RateLambda^{-\frac{1}{3}}(1+o(1)).
\end{equation*}
Therefore, 
\begin{equation*}
	\lim_{\RateLambda\to+\infty}\tilde K^{\mathrm{Ai}}_{r}(\tilde w,\tilde w')=
	-\frac{1}{2\pi\mathbf{i}}\int_{\gamma_{z}^{\mathrm{Ai}}}
	\exp\left\{ 
		\frac{\tilde w^3}{3}-\frac{\tilde z^3}{3}
		+\frac{r}{\upomega^\circ_{\tau,x}d^{\mathrm{TW}}_{\tau,x}}(\tilde z-\tilde w)
	\right\}
	\frac{d\tilde z}{(\tilde z-\tilde w)(\tilde z-\tilde w')},
\end{equation*}
where $\gamma_{z}^{\mathrm{Ai}}$ is the left contour
in \Cref{fig:locally_TW}.
Denote the kernel in the right-hand side above by 
$K_r^{\mathrm{Ai}}(\tilde w,\tilde w')$, 
where $\tilde w,\tilde w'$ belong to $\gamma_{w}^{\mathrm{Ai}}$,
the right contour in \Cref{fig:locally_TW}.

Combining the above computation in a neighborhood of $\upomega^\circ_{\tau,x}$
with the results of \Cref{sub:steep_descent_contours,sub:extra_residues}, we conclude that
\begin{equation*}
	\lim_{\RateLambda\to\infty}
	\det 
	\left( 1+K_{\zeta(\RateLambda)} \right)_{L^2(C_{\upomega^\circ_{\tau,x},\frac{\pi}{4}})}
	=
	\det 
	\left( 1+K_r^{\mathrm{Ai}} \right)_{L^2(\gamma_w^{\mathrm{Ai}})}.
\end{equation*}
The Fredholm determinant in the right-hand side is readily identified with a
Fredholm determinant of the Airy kernel, producing the
GUE Tracy--Widom distribution function of \Cref{def:fluctuation_distributions}
(cf.
\cite{TW_ASEP2}, \cite[Lemma C.1]{BorodinCorwinFerrari2012}):
\begin{equation*}
	\det 
	\left( 1+K_r^{\mathrm{Ai}} \right)_{L^2(\gamma_w^{\mathrm{Ai}})}
	=
	F_2
	\left( 
		\frac{r}{\upomega^\circ_{\tau,x}d^{\mathrm{TW}}_{\tau,x}}
	\right),
\end{equation*}
which completes the proof of the first part of \Cref{thm:main_theorem_on_fluctuations}.

\subsection{Contribution in the Gaussian phase}
\label{sub:Gaussian_contribution}

In the Gaussian phase $(\upomega_{\mathsf{cr}}=\mathcal{W}_x<\upomega^\circ_{\tau,x})$
the function $G_{\tau,x}$ has a simple critical point at
$\mathcal{W}_x$, and 
we take the power $\beta=\frac{1}{2}$.
Define 
\begin{equation}
	\label{m_x_definition}
	m_{x}:=
	\#
	\Bigl\{ 
		y\in\left\{ 0 \right\}\cup\left\{ b\in\RoadblockSet\colon 0<b<x \right\} 
		\colon \Speed(y)=\mathcal{W}_x
	\Bigr\}
\end{equation}
(cf. \eqref{range_of_Speed_notation}),
this is the multiplicity of the pole at $w=\mathcal{W}_x$
in \eqref{plugging_into_Fredholm_for_asymptotics}.

\begin{lemma}
	We have $G_{\tau,x}''(\mathcal{W}_x)<0$.
	\label{statement:second_derivative_G_critical_point}
\end{lemma}
\begin{proof}
	We have from \eqref{Equation_for_double_critical_points_2}
	\begin{equation*}
		\mathcal{W}_x G_{\tau,x}''(\mathcal{W}_x)
		=
		-\tau+\frac{1}{\mathcal{W}_x}\Phi_2(\mathcal{W}_x\mid x),
	\end{equation*}
	which is nonpositive by \eqref{PHI2_increasing_for_proof_of_descent}
	and cannot be zero because $\mathcal{W}_x>\upomega^\circ_{\tau,x}$.
\end{proof}
Thus, locally in a neighborhood of $\mathcal{W}_x$ the regions where 
$\Re G_{\tau,x}(w)-\Re G_{\tau,x}(\mathcal{W}_x)$
has constant sign look as in \Cref{fig:locally_G}.
Deform the $z$ and $w$ contours in the kernel 
$K_{\zeta(\RateLambda)}(w,w')$ \eqref{Fredholm_kernel_periodic_expansion_section44}
to $\Gamma_{\mathsf{cr}}$ and $C_{\mathcal{W}_x,\frac{\pi}{4}}$,
respectively, modified in a small neighborhood of $\mathcal{W}_x$
to look as in \Cref{fig:locally_G}.
Define
(recall notation \eqref{Big_Phi_n_definition})
\begin{equation}
	\label{dispersion_G}
	d^{\mathrm{G}}_{\tau,x}:=\sqrt{-G_{\tau,x}''(\mathcal{W}_x)}=
	\mathcal{W}_x^{-1}
	\sqrt{\tau\mathcal{W}_x-\Phi_2(\mathcal{W}_x\mid x)}>0,
\end{equation}
and make a change of variables
\begin{equation*}
	w=\mathcal{W}_x+\RateLambda^{-\frac{1}{2}}(d^{\mathrm{G}}_{\tau,x})^{-1}\tilde w
	,\qquad 
	w'=\mathcal{W}_x+\RateLambda^{-\frac{1}{2}}(d^{\mathrm{G}}_{\tau,x})^{-1}\tilde w'
	,\qquad 
	z=\mathcal{W}_x+\RateLambda^{-\frac{1}{2}}(d^{\mathrm{G}}_{\tau,x})^{-1}\tilde z,
\end{equation*}
where $\tilde z\in\gamma_z^{\mathrm{G}}$
and $\tilde w,\tilde w'\in \gamma_w^{\mathrm{G}}$
(these contours are made out of straight half lines, see \Cref{fig:locally_G}).
Let us work in the neighborhood of $\mathcal{W}_x$ of size $\RateLambda^{-\frac{1}{3}}$,
so that $|\tilde w|$, $|\tilde w'|$, and $|\tilde z|$ are $O(\RateLambda^{\frac{1}{6}})$.
In this neighborhood the Taylor expansion of of $G_{\tau,x}$
looks as
\begin{equation*}
	G_{\tau,x}(w)
	=
	G_{\tau,x}(\mathcal{W}_x)
	-
	\frac{\tilde w^2}{2 \RateLambda}+o(\RateLambda^{-1}).
\end{equation*}

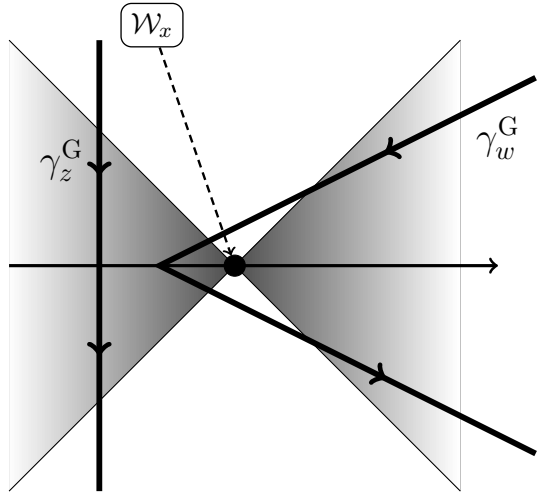
\begin{figure}[htbp]
	\begin{tikzpicture}
		[scale=1]
		\def\sq{0.866025}
		\filldraw[draw=none,left color=black!70!white, right color=white] (0,0) -- (3,3) -- (3,-3);
		\filldraw[draw=none,right color=black!70!white, left color=white] (0,0) -- (-3,3) -- (-3,-3);
		\draw (-3,-3)--(3,3);
		\draw (-3,3)--(3,-3);
		\draw[very thick, ->] (-3,0)--(3.5,0);
		\draw[fill] (0,0) circle (4pt) node (w1) {};
	  	\node (w1l) at (-1.1,3.2) 
	  	[draw=black, line width=.5, rectangle, 
	  	rounded corners, text centered] 
			{$\mathcal{W}_x$};
        \draw[->, line width=.9,densely dashed] (w1l)--(w1);
				\node at (3.5,1.7) {\Large$\gamma^{\mathrm{G}}_{w}$};
				\node at (-2.3,1.4) {\Large$\gamma^{\mathrm{G}}_{z}$};
				\draw[
					line width=2.2,
					decoration={
						markings,
						mark=at position 0.3 with {\arrow{>}},
						mark=at position 0.7 with {\arrow{>}}
					},
					postaction={decorate}
				] (-1.8,3)--++(0,-6);
				\draw[
					line width=2.2,
					decoration={
						markings,
						mark=at position 0.2 with {\arrow{>}},
						mark=at position 0.8 with {\arrow{>}}
					},
					postaction={decorate}
				] (4,2.5)--(-1,0)--(4,-2.5);
	\end{tikzpicture}
	\caption{%
		Behavior in a neighborhood of the simple critical point
		$\mathcal{W}_x$ of $G_{\tau,x}$.
		Shaded are regions where
		$\Re G_{\tau,x}(z)<\Re G_{\tau,x}(\mathcal{W}_x)$.
		The $z$ and $w$ contours in a neighborhood of $\mathcal{W}_x$ are
		also shown, modified in this neighborhood to avoid
		the pole in the kernel.%
	}
	\label{fig:locally_G}
\end{figure}

Thus, we obtain the following scaled kernel (which contains $r\in\mathbb{R}$ as a parameter):
\begin{align}
		\nonumber&
		K_{\zeta(\RateLambda)}^{\mathrm{steep}}(w,w')\sqrt{dw\,dw'}
		\\
		\nonumber&\hspace{20pt}=
		-
		\frac{1}{2\mathbf{i}}
		\sum_{k\in\mathbb{Z}}
		\int_{\Gamma_{\mathsf{cr}}}
		\frac{e^{(- \RateLambda \LimitShape(\tau,x)+r \RateLambda^{\frac{1}{2}})
		(\log(z/w)+2\pi\mathbf{i}k)}}
		{\sin \left(\pi
			\left( 
				\frac{\log(z/w)}{\log q}+
				\frac{2\pi\mathbf{i}k}{\log q}
			\right)\right)
		}
		\frac{g(w)}{g(z)}\frac{dz}{(z-w')z\log q}
		\sqrt{dw\,dw'}
		\\
		\nonumber&\hspace{20pt}=
		-
		\frac{\RateLambda^{-\frac{1}{2}}(d_{\tau,x}^{\mathrm{G}})^{-1}}{2\mathbf{i}}
		\sum_{k\in\mathbb{Z}}
		\int_{\tilde\Gamma_{\mathsf{cr}}}
		\frac{
			e^{
				2\pi\mathbf{i}k
				(- \RateLambda \LimitShape(\tau,x)+r \RateLambda^{\frac{1}{2}})
			}
		}
		{\sin \left(\pi
			\left( 
				\frac{\tilde z-\tilde w}{\mathcal{W}_x d_{\tau,x}^{\mathrm{G}}\log q}
				\RateLambda^{-\frac{1}{2}}(1+o(1))+
				\frac{2\pi\mathbf{i}k}{\log q}
			\right)\right)
		}
		\frac
		{\bigl( 1+\frac{\tilde z}{\mathcal{W}_xd_{\tau,x}^{\mathrm{G}}}\RateLambda^{-\frac{1}{2}} ;q \bigr)_{\infty}^{m_x}}
		{\bigl( 1+\frac{\tilde w}{\mathcal{W}_xd_{\tau,x}^{\mathrm{G}}}\RateLambda^{-\frac{1}{2}} ;q \bigr)_{\infty}^{m_x}}
		\\\nonumber
		&\hspace{40pt}\times
		\exp\left\{ 
			-\frac{\tilde w^2}{2}+\frac{\tilde z^2}{2}+
			\frac{r}{\mathcal{W}_xd_{\tau,x}^{\mathrm{G}}}
			(\tilde z-\tilde w)+o(1)
		\right\}
		\frac{\bigl(1+o(1)\bigr)d\tilde z}{(\tilde z-\tilde w')
		(\mathcal{W}_x+o(1))\log q}
		\sqrt{d\tilde w\,d\tilde w'}
		\\&\hspace{20pt}
		=:
		\tilde K^{\mathrm{G}}_{r,m_x}(\tilde w,\tilde w')\sqrt{d \tilde w\,d \tilde w'},
	\label{Gaussian_scaling_of_periodic_kernel}
\end{align}
where $\tilde \Gamma_{\mathsf{cr}}$
is the image of $\Gamma_{\mathsf{cr}}$
under the change of variables $z\to\tilde z$.
Here we used the fact that among 
the factors in \eqref{plugging_into_Fredholm_for_asymptotics}
corresponding to $\Speed(0)$ or the roadblocks, 
exactly $m_x$ contain a simple pole at $w=\mathcal{W}_x$,
while other factors are regular and thus 
behave as $1+o(1)$ as $\RateLambda\to+\infty$.
We have for one such factor with a pole:
\begin{equation*}
	\frac
	{\bigl( 1+\frac{\tilde z}{\mathcal{W}_xd_{\tau,x}^{\mathrm{G}}}\RateLambda^{-\frac{1}{2}} ;q \bigr)_{\infty}}
	{\bigl( 1+\frac{\tilde w}{\mathcal{W}_xd_{\tau,x}^{\mathrm{G}}}\RateLambda^{-\frac{1}{2}} ;q \bigr)_{\infty}}
	=
	\frac{\tilde z}{\tilde w}\bigl( 1+o(1) \bigr).
\end{equation*}
Similarly to \Cref{sub:TW_contribution} one sees that the terms in 
\eqref{Gaussian_scaling_of_periodic_kernel}
corresponding to $k\ne 0$ vanish in the limit, and so
\begin{equation*}
	\lim_{\RateLambda\to+\infty}
	\tilde K_{r,m_x}^{\mathrm{G}}(\tilde w,\tilde w')
	=
		-
		\frac{1}{2\pi\mathbf{i}}
		\int_{\gamma_z^{\mathrm{G}}}
		\exp\left\{ 
			-\frac{\tilde w^2}{2}+\frac{\tilde z^2}{2}+
			\frac{r}{\mathcal{W}_xd_{\tau,x}^{\mathrm{G}}}
			(\tilde z-\tilde w)
		\right\}
		\left( \frac{\tilde z}{\tilde w} \right)^{m_x}
		\frac{d\tilde z}
		{
			(\tilde z-\tilde w)
			(\tilde z-\tilde w')
		},
\end{equation*}
where $\gamma_z^{\mathrm{G}}$ is the left contour in \Cref{fig:locally_G}.
Denote the kernel in the right-hand side above by $K^{\mathrm{G}}_{r,m_x}(\tilde w,\tilde w')$,
where 
$\tilde w,\tilde w'$ belong to $\gamma_w^{\mathrm{G}}$, the 
right contour in \Cref{fig:locally_G}.

Combining the above computation in a neighborhood of $\mathcal{W}_x$
with the results of \Cref{sub:steep_descent_contours,sub:extra_residues}, we conclude that
\begin{equation*}
	\lim_{\RateLambda\to\infty}
	\det 
	\left( 1+K_{\zeta(\RateLambda)} \right)_{L^2(C_{\mathcal{W}_x,\frac{\pi}{4}})}
	=
	\det 
	\left( 1+K_{r,m_x}^{\mathrm{G}} \right)_{L^2(\gamma_w^{\mathrm{G}})}.
\end{equation*}
The Fredholm determinant in the right-hand side can be identified with 
the distribution function of the largest eigenvalue
of an $m_x\times m_x$ GUE random matrix 
(cf. \cite{barraquand2015phase}):
\begin{equation*}
	\det 
	\left( 1+K_{r,m_x}^{\mathrm{G}} \right)_{L^2(\gamma_w^{\mathrm{G}})}
	=
	G_{m_x}\left( \frac{r}{\mathcal{W}_xd_{\tau,x}^{\mathrm{G}}} \right),
\end{equation*}
which completes the proof of the third part of 
\Cref{thm:main_theorem_on_fluctuations}.

\subsection{Contribution at a transition point}
\label{sub:BBP_contribution}

Finally, we consider the case when 
the double critical point coming from the function $G_{\tau,x}$ coincides
with a pole outside the exponent in \eqref{plugging_into_Fredholm_for_asymptotics}, 
that is,
$\upomega_{\mathsf{cr}}=\mathcal{W}_x=\upomega^\circ_{\tau,x}$.
We use the notation $m_x$ \eqref{m_x_definition}
for the multiplicity of this pole.
Take the power $\beta=\frac{1}{3}$, 
consider the same change of variables \eqref{TW_change_of_variables}
as in the Tracy--Widom phase, where $\tilde z\in\gamma_{z}^{\mathrm{BBP}}$
and $\tilde w,\tilde w'\in\gamma_{w}^{\mathrm{BBP}}$.
The only difference between the contours 
$\gamma_{z,w}^{\mathrm{BBP}}$
and 
$\gamma_{z,w}^{\mathrm{Ai}}$
in \Cref{fig:locally_TW}
is that the former contours should pass to the left of $\upomega^\circ_{\tau,x}$
to avoid the pole (in particular, $\gamma_z^{\mathrm{BBP}}$ is the same as $\gamma_z^{\mathrm{Ai}}$).
The regions where $\Re G_{\tau,x}(w)-\Re G_{\tau,x}(\upomega^\circ_{\tau,x})$
has constant sign in a neighborhood of $\upomega^\circ_{\tau,x}$
look exactly the same as in \Cref{fig:locally_TW}.
Thus, arguing similarly to \Cref{sub:TW_contribution,sub:Gaussian_contribution}, we see that 
\begin{equation}\label{BBP_Fredholm_convergence}
	\lim_{\RateLambda\to+\infty}\det
	\left( 1+K_{\zeta(\RateLambda)} \right)_{L^2(C_{\upomega^\circ_{\tau,x},\frac{\pi}{4}})}
	=
	\det
	\left( 1+K^{\mathrm{BBP}}_{r,m_x} \right)_{L^2(\gamma_w^{\mathrm{BBP}})},
\end{equation}
where the latter kernel has the form
\begin{equation*}
	K^{\mathrm{BBP}}_{r,m_x}(\tilde w,\tilde w')=
	-\frac{1}{2\pi\mathbf{i}}\int_{\gamma_{z}^{\mathrm{BBP}}}
	\exp\left\{ 
		\frac{\tilde w^3}{3}-\frac{\tilde z^3}{3}
		+\frac{r}{\upomega^\circ_{\tau,x}d^{\mathrm{TW}}_{\tau,x}}(\tilde z-\tilde w)
	\right\}
	\left( \frac{\tilde z}{\tilde w} \right)^{m_x}
	\frac{d\tilde z}{(\tilde z-\tilde w)(\tilde z-\tilde w')}.
\end{equation*}
The Fredholm determinant in the right-hand side of 
\eqref{BBP_Fredholm_convergence} can be identified with the 
BBP distribution function of \Cref{def:fluctuation_distributions}
(cf. \cite[Lemma C.2]{BorodinCorwinFerrari2012}):
\begin{equation*}
	\det
	\left( 1+K^{\mathrm{BBP}}_{r,m_x} \right)_{L^2(\gamma_w^{\mathrm{BBP}})}
	=
	F_{\mathrm{BBP},m_x,\mathbf{b}}
	\left( \frac{r}{\upomega^\circ_{\tau,x}d^{\mathrm{TW}}_{\tau,x}} \right).
\end{equation*}
The particular distribution we obtain 
in this limit regime has
order $m_x$ and $\mathbf{b}=(0,0,\ldots,0 )$.
This completes the proof of the second part of \Cref{thm:main_theorem_on_fluctuations}.

\appendix 

\section{$q$-polygamma functions}
\label{sec:appendix_q}

Here we list a number of formulas related to the $q$-gamma and $q$-polygamma
functions which are used throughout the paper.
The $q$-gamma function is defined by
(we always assume $q\in(0,1)$)
\begin{equation*}
	\Gamma_q(z) := (1-q)^{1-z} \frac{(q;q)_{\infty}}{(q^z;q)_{\infty}}.
\end{equation*}
We have $\lim_{q\nearrow 1}\Gamma_q(z)=\Gamma(z)$. The log-derivative of $\Gamma_q(z)$ (the $q$-digamma function) is denoted by
\begin{equation*}
	\psi_q(z) := \frac{1}{\Gamma_q(z)}\frac{\partial \Gamma_q(z)}{\partial z}.
\end{equation*}
It is straightforward that
\begin{equation}\label{q_digamma}
	\psi_q(z) = -\log(1-q) + \log q  \sum_{k=0}^{\infty} \frac{q^{k+z}}{1-q^{k+z}},
\end{equation}
which is a meromorphic function in $z$ having poles when $q^{z+k}=1$
(and the series converges for any $z$ except these poles
thanks to the factors $q^k$).

The following formula is an alternative series representation for derivatives
of $\psi_q(z)$ (the so-called $q$-polygamma functions):
\begin{equation}\label{q_polygamma}
	\psi_q^{(n)}(z) = (\log q)^{n+1} \sum_{k=1}^{\infty} \frac{k^n q^{kz}}{1-q^k},
	\qquad n\ge1,
\end{equation}
e.g., see \cite[Lemma 2.1]{barraquand2015q} for the computation. 
In contrast with \eqref{q_digamma}, this series converges only when
$|q^{z}|<1$, i.e., when $\Re z>0$.

It is convenient to replace $q^{z}$ by $w$, and define for any $n\ge0$:
\begin{equation}\label{Phi_n_function}
	\phi_{n}(w):=\sum_{k=1}^{\infty} \frac{k^n w^{k}}{1-q^k},\qquad |w|<1.
\end{equation}
We thus have
\begin{equation}\label{Phi_n_function_analytic_continuation}
	\phi_{n}(w) = \frac{\log(1-q)}{\log q}\,\mathbf{1}_{n=0}
	+ \frac{1}{(\log q)^{n+1}}\psi_q^{(n)}(\log_{q}w),\qquad n\ge0.
\end{equation}
The latter formula gives an analytic continuation of the series
\eqref{Phi_n_function} to a meromorphic function of $w\in\mathbb{C}$ having poles of
order $n+1$ at $w=q^{-k}$, $k\in\mathbb{Z}_{\ge0}$.

Several useful properties of the functions $\phi_n$ are summarized below:
\begin{proposition}
	\label{statement:Phi_n_behavior}
	We have
	\begin{equation}
		\label{little_phi_n_derivative}
		\frac{\partial}{\partial w}\phi_{n}(w)=\frac{1}{w}\phi_{n+1}(w),\qquad n\ge0.
	\end{equation}
	The functions $\phi_{0}(w)$ and $\phi_{1}(w)$ are negative for
	negative real $w$, and $\phi_{1}(w)$ and $\phi_{3}(w)$ are positive
	for positive real $w\notin q^{\mathbb{Z}_{\le 0}}$, while $\phi_{2}(w)$ is
	positive for $w\in(0,1)$. 
	Moreover, $\phi_{n}(0)=0$ for all $n$.
\end{proposition}
\begin{proof}
	The claim about the derivatives is straightforward from either
	\eqref{Phi_n_function} or \eqref{Phi_n_function_analytic_continuation}. 
	To check the signs of the $\phi_n$'s, let us use the series \eqref{q_digamma}
	and its derivatives to get formulas for $\phi_n(w)$ valid for all $w$. 
	We have
	\begin{align}
		\label{little_phi_series}
		\phi_{0}(w) & = \sum_{k = 0}^{\infty}\frac{q^kw}{1-q^kw},\qquad
		\phi_{1}(w)   = \sum_{k = 0}^{\infty}\frac{q^kw}{(1-q^kw)^{2}},\\
		\phi_{2}(w) & = \sum_{k = 0}^{\infty}\frac{q^{k}w(1+q^{k}w)}{(1-q^kw)^{3}},\qquad
		\phi_{3}(w)   = \sum_{k = 0}^{\infty}\frac{q^kw (1+4 q^kw +q^{2 k}w^2 )}{(1-q^kw)^4}.
		\label{little_phi_series2}
	\end{align}
	This immediately implies all the remaining claims.
\end{proof}

\section{Translation invariant stationary distributions}
\label{sec:appendix_stationary_distributions}

\subsection{Preliminaries}
\label{sub:appendix_preliminaries}

Here we perform computations related to translation invariant stationary
distributions of homogeneous versions of our particle systems on the whole
(discrete or continuous) line.
Classification of translation invariant stationary distributions for rather
general zero range processes (in the sense of \cite{Spitzer1970}) on
$\mathbb{Z}$ is well-known, e.g., see \cite{andjel1982invariant}.
In particular, under mild conditions on the process every translation invariant
stationary distribution is a mixture of product measures.
Here by a product measure we mean assigning random independent identically
distributed numbers of particles at each location in $\mathbb{Z}$ (such a random
configuration is clearly translation invariant).

While neither the half-continuous stochastic higher spin six vertex model nor
the exponential jump model are zero range, the existence (for suitable initial
configurations) of these processes on $\mathbb{Z}$ and $\mathbb{R}$,
respectively, can be established similarly to
\cite{Liggett1973infinite_zerorange}, \cite{andjel1982invariant}.
The main observation is that our process on $\mathbb{Z}$ (denote it by
$\VertexXProcess^{\mathbb{Z}}(t)$) is ``slower'' than the zero range process
(with the geometric jumping distribution) obtained by dropping the interaction
of the flying particles with the sitting ones. 
In the case of the continuous space, our process (denote it by
$\XProcess^{\mathbb{R}}(t)$) is ``slower'' than 
simply the process of independent particles each of which jumps to the 
right by an exponentially distributed random distance after exponentially distributed time intervals.
Therefore, to show the existence of both $\VertexXProcess^{\mathbb{Z}}$ and
$\XProcess^{\mathbb{R}}$ one can essentially repeat the estimates of
\cite{Liggett1973infinite_zerorange} or \cite{andjel1982invariant} (with
suitable modifications in the case of the continuous space).
This also implies that product measures on $\mathbb{Z}$ or their analogues on
$\mathbb{R}$, marked Poisson processes (see \Cref{def:continuous_product}
below), can serve as (random) initial configurations for
$\VertexXProcess^{\mathbb{Z}}$ and $\XProcess^{\mathbb{R}}$, respectively,
provided that the random number of points at a single location has, say, 
two finite first moments.

Here we do not attempt to classify all translation invariant stationary
distributions of $\VertexXProcess^{\mathbb{Z}}$ and $\XProcess^{\mathbb{R}}$,
but instead show that certain specific product measures or marked Poisson
processes are indeed stationary under our systems.
We also obtain formulas for particle density and particle current (sometimes
also called particle flux) for these measures which (in the case of
$\XProcess^{\mathbb{R}}$) are employed in the heuristic derivation of the
macroscopic limit shape for the height function in
\Cref{sub:intro_hydrodynamics} (%
	based on the assumption that these measures describe local behavior of the
	inhomogeneous exponential jump model%
).

\subsection{Stationary distributions for half-continuous vertex model}
\label{sub:appendix_half_cont}

Let $\VertexXProcess^{\mathbb{Z}}(t)$ be the homogeneous version of the
half-continuous stochastic higher spin six vertex model (described in
\Cref{sub:continuous_vertex_coordinate}) which is well-defined on a suitable
subset of $\mathrm{Conf}_{\bullet}^{\sim}(\mathbb{Z})$, the space of possibly
countably infinite particle configurations in $\mathbb{Z}$ with multiple
(but finitely many)
particles per location allowed.
This process depends on parameters $\upxi_i\equiv \upxi>0$ and
$\mathsf{s}_i\equiv \mathsf{s}\in(-1,0)$.

For any $c\ge0$, let $\HahnDistribution_{c,\mathsf{s}^2}$ be the probability
distribution on $\left\{ 0,1,2,\ldots \right\}$ defined as
\begin{equation}\label{infinite_q_binomial_distribution}
	\HahnDistribution_{c,\mathsf{s}^2}(j)
	:=
	c^j\frac{(\mathsf{s}^2;q)_j}{(q;q)_j}
	\frac{(c;q)_{\infty}}{(c\mathsf{s}^2;q)_{\infty}}
	,\qquad j\ge0.
\end{equation}
The fact that these quantities sum to one follows from the $q$-binomial theorem
\cite[(1.3.2)]{GasperRahman}. 
When $\mathsf{s}=0$, $\HahnDistribution_{c,\mathsf{s}^2}$ turns into the
distribution $(c;q)_{\infty}c^j/(q;q)_j$ which is often called the $q$-geometric
distribution. 
The latter in turn becomes the usual geometric distribution $(1-c)c^j$ when
$q=0$. 

Let
${\mathfrak{m}}^{\mathrm{hc}}_{c,\mathsf{s}^2}=(\HahnDistribution_{c,\mathsf{s}^2})^{\otimes\mathbb{Z}}$
be the product probability measure on $\left\{ 0,1,2,\ldots
\right\}^{\mathbb{Z}}$ corresponding to the number of particles at each
location distributed as
$\HahnDistribution_{c,\mathsf{s}^2}$, all of them being independent.
\begin{proposition}
	\label{statement:discrete_invariant}
	For any $c\ge0$, the measure ${\mathfrak{m}}^{\mathrm{hc}}_{c,\mathsf{s}^2}$ on
	$\mathrm{Conf}_{\bullet}^{\sim}(\mathbb{Z})$ is stationary under the process
	$\VertexXProcess^{\mathbb{Z}}(t)$ with the matching parameter $\mathsf{s}$
	and arbitrary $\upxi$.
\end{proposition}
\begin{proof}
	It suffices to consider the evolution of the distribution of the number of
	particles at a given location in $\mathbb{Z}$. 
	If there are $k\ge0$ particles at this location, then one particle leaves it
	at rate 
	\begin{equation*}
		\mathop{\mathrm{Rate}}(k\to k-1)=-\upxi\mathsf{s}(1-q^k).
	\end{equation*}
	Let us compute $\mathop{\mathrm{Rate}}(k\to k+1)$. 
	Let $Y^{\mathrm{hc}},Y_1^{\mathrm{hc}},Y_2^{\mathrm{hc}},\ldots $ be independent random variables distributed as
	$\HahnDistribution_{c,\mathsf{s}^2}$. 
	The rate at which a particle joins a stack of $k$ particles 
	has the following form:
	\begin{equation}
	\label{rate_k_k1_computation}
		\mathop{\mathrm{Rate}}(k\to k+1)
		=
		\mathop{\mathbb{E}}
		\Bigg(
			\sum_{n=0}^{\infty}
			\;
			\underbrace{-\upxi\mathsf{s}(1-q^{Y^{\mathrm{hc}}})}_{
				{\parbox{.17\textwidth}{\textnormal{\scriptsize{}rate at which a particle leaves a stack $n+1$ positions to the left}}}
			}
			\
			\cdot
			\
			\underbrace{q^{{Y}^{\mathrm{hc}}_1}\ldots q^{{Y}^{\mathrm{hc}}_n}(\mathsf{s}^2)^n}_{
				{\parbox{.17\textwidth}{\textnormal{\scriptsize{}probability that the flying particle travels distance $n$}}}
			}
			\
			\cdot
			\
			\underbrace{(1-\mathsf{s}^2q^k)}_{
				{\parbox{.17\textwidth}{\textnormal{\scriptsize{}probability that the flying particle stops at our location having $k$
				sitting particles}}}
			}
		\;\Bigg).
	\end{equation}
	One readily sees that 
	\begin{equation*}
		\mathop{\mathbb{E}}q^{Y^{\mathrm{hc}}}
		=
		\sum_{j=0}^{\infty}(cq)^j
		\frac{(q;q)_j}{(\mathsf{s^2};q)_j}
		\frac{(c;q)_{\infty}}{(c\mathsf{s}^2;q)_{\infty}}
		=
		\frac{(cq\mathsf{s}^2;q)_{\infty}}{(cq;q)_{\infty}}
		\frac{(c;q)_{\infty}}{(c\mathsf{s}^2;q)_{\infty}}
		=
		\frac{1-c}{1-c\mathsf{s}^2},
	\end{equation*}
	and thus the sum in \eqref{rate_k_k1_computation} simplifies to 
	\begin{equation*}
		\mathop{\mathrm{Rate}}(k\to k+1)
		=
		-\upxi\mathsf{s}c(1-\mathsf{s}^2q^k).
	\end{equation*}
	The desired stationarity now follows from the identity which can be readily
	verified:
	\begin{equation*}
		\left(
			\HahnDistribution_{c,\mathsf{s}^2}(k-1)
			-
			\HahnDistribution_{c,\mathsf{s}^2}(k)
		\right)
		\mathop{\mathrm{Rate}}(k-1\to k)
		+
		\left(
			\HahnDistribution_{c,\mathsf{s}^2}(k+1)
			-
			\HahnDistribution_{c,\mathsf{s}^2}(k)
		\right)
		\mathop{\mathrm{Rate}}(k+1\to k)=0
	\end{equation*}
	(%
		if $k=0$, then by agreement
		$
			\mathop{\mathrm{Rate}}(0\to -1)
			=
			\mathop{\mathrm{Rate}}(-1\to 0)
			=
			0
		$%
	).
\end{proof}

Let us now compute the particle density ${\rho}^{\mathrm{hc}}(c)$ and the
particle current ${\jmath}^{\mathrm{hc}}(c)$ associated with the product measure
${\mathfrak{m}}^{\mathrm{hc}}_{c,\mathsf{s}^2}$ (a similar computation appears
in \cite{Veto2014qhahn}).
The particle current is the average number of particles jumping over any given
location per unit time.
This quantity is given by the sum in the right-hand side of
\eqref{rate_k_k1_computation} without the last factor $(1-\mathsf{s}^2q^k)$,
and hence 
\begin{equation*}
	{\jmath}^{\mathrm{hc}}(c)=-\upxi\mathsf{s}c.
\end{equation*}
The particle density is the average number of particles per location, it is
equal to
\begin{multline*}
	{\rho}^{\mathrm{hc}}(c)
	=
	\mathbb{E}\,({Y^{\mathrm{hc}}})
	=
	\sum_{j=0}^{\infty}
	j
	c^j\frac{(\mathsf{s}^2;q)_j}{(q;q)_j}
	\frac{(c;q)_{\infty}}{(c\mathsf{s}^2;q)_{\infty}}
	=
	\frac{(c;q)_{\infty}}{(c\mathsf{s}^2;q)_{\infty}}
	\,
	c\frac{\partial}{\partial c}
	\left( 
		\frac{(c\mathsf{s}^2;q)_{\infty}}{(c;q)_{\infty}}
	\right)
	\\=
	\frac{c\frac{\partial}{\partial c}(c\mathsf{s}^2;q)_{\infty}}{(c\mathsf{s}^2;q)_{\infty}}
	-
	\frac{c\frac{\partial}{\partial c}(c;q)_{\infty}}{(c;q)_{\infty}}
	=
	\sum_{i=0}^{\infty}\frac{-c\mathsf{s}^2q^i}{1-c\mathsf{s}^2q^i}
	+
	\sum_{i=0}^{\infty}\frac{cq^i}{1-cq^i}
	=
	\phi_0(c)-\phi_0(c\mathsf{s}^2),
\end{multline*}
where in the last equality we used \eqref{little_phi_series}.
We will not further pursue formulas for the discrete model, and instead in the
next two subsections turn to the exponential jump model as it is the main
object of the present paper.

\subsection{Stationary distributions for the exponential jump model}
\label{sub:appendix_expon_model}

The homogeneous exponential jump model $\XProcess^{\mathbb{R}}(t)$ on the whole
line depends on parameters $\Speed(x)\equiv \Speed>0$ and $\RateLambda>0$,
and we assume that there are no roadblocks.\footnote{%
	The assumption that there are no roadblocks makes sense because we expect
	that $\XProcess^{\mathbb{R}}$ should describe local behavior
	under the inhomogeneous exponential jump model from \Cref{sub:definition_of_the_model},
	and roadblocks cannot accumulate locally.%
}
We will use the following analogue of product measures on the continuous line:
\begin{definition}
	\label{def:continuous_product}
	A \emph{marked Poisson process} with marks following a probability
	distribution $\ContinuousHahnDistribution$ on $\mathbb{Z}_{\ge1}$ is a 
	probability measure on $\mathrm{Conf}_{\bullet}^{\sim}(\mathbb{R})$
	(i.e., a random particle configuration in $\mathbb{R}$)
	obtained as follows.
	Take a (homogeneous) Poisson process on $\mathbb{R}$ with some rate $\mu>0$,
	and at each point of this Poisson process put a random number of particles
	according to the distribution $\ContinuousHahnDistribution$ on
	$\mathbb{Z}_{\ge1}$, independently for each point of the Poisson process.
\end{definition}

As explained in \Cref{sub:limit_to_continuous_space}, the exponential jump model
arises from the half-continuous vertex model as $\mathsf{s}^2=e^{- \RateLambda
\varepsilon}\to1$, $-\upxi\mathsf{s}=\Speed$ is fixed, and the discrete space
is rescaled by $\varepsilon$ to become continuous.
In this limit regime the distribution $\HahnDistribution_{c,\mathsf{s}^2}$
\eqref{infinite_q_binomial_distribution} behaves as follows:
\begin{equation*}
	\HahnDistribution_{c,\mathsf{s}^2}(0)
	=
	\frac{(c;q)_{\infty}}{(ce^{- \RateLambda\varepsilon};q)_{\infty}}
	=
	\exp \biggl\{ 
		\sum_{i=0}^{\infty}
		(-\varepsilon \RateLambda)\frac{c q^i}{1-c q^i}+O(\varepsilon^2)
	\biggr\}
	=
	1-\varepsilon \RateLambda \phi_0(c)+O(\varepsilon^2),
\end{equation*}
and for any $j\ge1$ we have
\begin{equation*}
	\HahnDistribution_{c,\mathsf{s}^2}\bigl(j\mid j\ge1\bigr)
	=
	c^j\frac{(e^{- \RateLambda\varepsilon};q)_j}{(q;q)_j}
	\frac{1}{\frac{(ce^{- \RateLambda\varepsilon};q)_\infty}{(c;q)_{\infty}}-1}
	=
	c^j\frac{(1-e^{- \RateLambda\varepsilon})}{\varepsilon \RateLambda \phi_0(c)+O(\varepsilon^2)}
	\frac{(qe^{- \RateLambda\varepsilon};q)_{j-1}}{(q;q)_j}
	\to
	\frac{1}{\phi_0(c)}\frac{c^j}{1-q^j}.
\end{equation*}
In other words, the product measure ${\mathfrak{m}}^{\mathrm{hc}}_{c,\mathsf{s}^2}$ from
\Cref{sub:appendix_half_cont} turns into a marked Poisson process with rate 
$\RateLambda\phi_0(c)$ and marking distribution
\begin{equation}
	\ContinuousHahnDistribution_{c}(j)
	:=
	\frac{1}{\phi_0(c)}\frac{c^j}{1-q^j},\qquad j\ge1.
	\label{ContinuousHahnDistribution_def}
\end{equation}
Denote this marked Poisson process by $\mathfrak{m}_{c,\RateLambda}$.
\begin{proposition}
	\label{statement:cont_invariant}
	For any $c\ge0$ the measure $\mathfrak{m}_{c,\RateLambda}$ on
	$\mathrm{Conf}_{\bullet}^{\sim}(\mathbb{R})$ is stationary under the process
	$\XProcess^{\mathbb{R}}(t)$ with the matching parameter $\RateLambda$ and
	arbitrary parameter $\Speed$.
\end{proposition}
\begin{proof}
	The proof is a continuous space modification of the one of
	\Cref{statement:discrete_invariant}. 
	However, as the continuous space computations are somewhat more involved, let
	us give some details.
	Fix $x\in\mathbb{R}$, and consider the evolution of the distribution
	of the number of particles in $(x,x+dx)$.
	First, if there are $k\ge1$ particles there, then the rate 
	at which one particle leaves is 
	\begin{equation*}
		\mathop{\mathrm{Rate}}(k\to k-1)=(1-q^k)\Speed.
	\end{equation*}
	Let $Y,Y_1,Y_2,\ldots $ be independent random variables distributed as
	$\ContinuousHahnDistribution_{c}$, and also let $x>p_1>p_2>\ldots $
	be all points of the Poisson process of rate $\RateLambda\phi_0(c)$ 
	to the left of our $x$.
	We have, similarly to \eqref{rate_k_k1_computation}:
	\begin{equation}
		\mathop{\mathrm{Rate}}(k\to k+1)
		=
		\mathop{\mathbb{E}}\Bigg( 
			\sum_{i=1}^{\infty}
			\
			\underbrace{
			(1-q^{Y})\Speed}_{
				{\parbox{.17\textwidth}{\textnormal{\scriptsize{}rate at which a particle wakes up at the stack at $p_i$}}}
			}
			\cdot
			\ 
			\underbrace{
				q^{Y_{1}}\ldots q^{Y_{i-1}}e^{- \RateLambda(x-p_i)}
			}_{
				{\parbox{.17\textwidth}{\textnormal{\scriptsize{}probability that the flying particle travels past location $x$}}}
			}
			\
			\cdot
			\underbrace{
				(1-q^k e^{- \RateLambda dx})
			}_{
				{\parbox{.18\textwidth}{\textnormal{\scriptsize{}probability that the flying particle stops in $(x,x+dx)$,
				see below}}}
			}
		\Bigg).
		\label{continuous_computation_of_rate_k_k+1}
	\end{equation}
	The last factor $1-q^k e^{- \RateLambda dx}$ serves two cases: for $k=0$, the
	probability that the flying particle stops in $(x,x+dx)$ is $\RateLambda
	dx+O(dx^2)$ (which is small), while for $k\ge1$ it is $1-q^k+O(dx)$.
	We have
	\begin{equation*}
		\mathop{\mathbb{E}}q^{Y}
		=
		\frac{1}{\phi_0(c)}\sum_{j=1}^{\infty}\frac{(cq)^j}{1-q^j}
		=
		\frac{\phi_0(cq)}{\phi_0(c)}
		=
		1-\frac{c}{(1-c)\phi_0(c)},
	\end{equation*}
	where in the last equality we used \eqref{little_phi_series}.
	Next, observe that $x-p_i$ is a sum of $i$ independent exponential random
	variables with parameter $\RateLambda\phi_0(c)$ (%
		denote one such variable by $Y'$%
	), and so 
	\begin{equation*}
		\mathop{\mathbb{E}}e^{- \RateLambda(x-p_i)}
		=
		\left( \mathop{\mathbb{E}}e^{- \RateLambda Y'} \right)^i
		=
		\left( \frac{\phi_0(c)}{1+\phi_0(c)} \right)^i.
	\end{equation*}
	Thus, \eqref{continuous_computation_of_rate_k_k+1} turns into
	\begin{equation*}
		\mathop{\mathrm{Rate}}(k\to k+1)
		=
		\Speed
		\frac{
			c
			(1-q^k e^{- \RateLambda dx})
		}
		{(1-c)\phi_0(c)}
		\sum_{i=1}^{\infty}
		\left( \frac{\phi_0(c)}{1+\phi_0(c)} \right)^i
		\left( 1-\frac{c}{(1-c)\phi_0(c)} \right)^{i-1}
		=\Speed c
		(1-q^k e^{- \RateLambda dx}).
	\end{equation*}
	The desired stationarity follows from the identities
	\begin{align*}
		k=0:\qquad &
		\displaystyle
		-e^{- \RateLambda\phi_0(c) dx}\Speed c(1-e^{- \RateLambda dx})
		+
		(1-e^{- \RateLambda\phi_0(c) dx})\frac{c}{(1-q)\phi_0(c)}(1-q)\Speed=O(dx^2)
		;
		\\[12pt]
		k=1:\qquad &
		\displaystyle
		-(1-e^{- \RateLambda\phi_0(c) dx})\frac{c}{(1-q)\phi_0(c)}
		\bigl[  
			\Speed c(1-q e^{- \RateLambda dx})
			+
			(1-q)\Speed
		\bigr]
		\\[12pt]
		&\displaystyle
		\hspace{40pt}
		+
		(1-e^{- \RateLambda\phi_0(c) dx})\frac{c^{2}}{(1-q^{2})\phi_0(c)}(1-q^{2})\Speed
		\\[12pt]
		&\displaystyle
		\hspace{80pt}
		+
		e^{- \RateLambda\phi_0(c)dx}\Speed c(1-e^{- \RateLambda dx})=O(dx^2)
		;
		\\[12pt]
		k\ge2:\qquad &
		\displaystyle
		-(1-e^{- \RateLambda\phi_0(c) dx})\frac{c^k}{(1-q^k)\phi_0(c)}
		\bigl[  
			\Speed c(1-q^k e^{- \RateLambda dx})
			+
			(1-q^k)\Speed
		\bigr]
		\\[12pt]
		&\displaystyle
		\hspace{40pt}
		+
		(1-e^{- \RateLambda\phi_0(c) dx})\frac{c^{k+1}}{(1-q^{k+1})\phi_0(c)}(1-q^{k+1})\Speed
		\\[12pt]
		&\displaystyle
		\hspace{80pt}
		+
		(1-e^{- \RateLambda\phi_0(c) dx})\frac{c^{k-1}}{(1-q^{k-1})\phi_0(c)}\Speed c(1-q^{k-1}e^{- \RateLambda dx})
		=O(dx^2),
	\end{align*}
	which are readily verified.
\end{proof}

Let us now write down the particle density $\rho(c)$ and the particle current
$j(c)$ under the measure $\mathfrak{m}_{c,\RateLambda}$. We have
\begin{equation}
	\label{rho_c_expon}
	\rho(c)
	=
	\RateLambda\phi_0(c)\,{\mathbb{E}}\,(Y)
	=
	\RateLambda\sum_{j=1}^{\infty}\frac{jc^j}{1-q^j}
	=
	\RateLambda c\frac{\partial}{\partial c}\phi_0(c)
	=
	\RateLambda \phi_1(c),
\end{equation}
where we employed \eqref{little_phi_n_derivative}.
The particle current is given by the sum in 
the right-hand side of \eqref{continuous_computation_of_rate_k_k+1}
without the last factor $1-q^k e^{- \RateLambda dx}$, 
and so is equal to 
\begin{equation}
	\label{j_c_expon}
	j(c)=\Speed c.
\end{equation}

From \Cref{statement:Phi_n_behavior} it readily follows that
$\phi_1\colon[0,1]\to[0,+\infty]$ is one to one and increasing. 
Let $\phi_1^{-1}$ denote the inverse. 
From \eqref{rho_c_expon} and \eqref{j_c_expon} we see that the dependence
of the current on the density is
\begin{equation}
	j(\rho)=\Speed\phi_1^{-1}(\rho/\RateLambda).
	\label{j_of_rho_expon}
\end{equation}

\subsection{Verification of the macroscopic limit shape}
\label{sub:appendix_verification}

Recall the inhomogeneous exponential jump model on $\mathbb{R}_{>0}$ defined in
\Cref{sub:definition_of_the_model} depending on the speed function $\Speed(x)$
and the jump distance parameter $\RateLambda>0$. 
Let us assume that there are no roadblocks, and, moreover, that
$\Speed(x)$ is continuous at $x=0$.
Let us take a slightly more general model in which $\RateLambda$ also depends
on the location $x\in \mathbb{R}_{>0}$ as $\tilde\RateLambda(x)L$, where $L$ is a
large parameter and $\tilde\RateLambda(x)$ is a positive continuous function
bounded away from zero and infinity.
We also rescale the continuous time as $t=\tau L$.

Assume that in the limit as $L\to+\infty$ there is a limiting density of
particles $\rho(\tau,x)\in[0,+\infty]$. 
Moreover, assume that locally at each $x\in\mathbb{R}_{>0}$ where 
$\rho(\tau,x)< +\infty$ the behavior of the
particle system is described by the translation invariant stationary
distribution $\mathfrak{m}_{c,\tilde\RateLambda(x)}$ defined in
\Cref{sub:appendix_expon_model}.
Under these assumptions and using \eqref{j_of_rho_expon}, one naturally expects
(%
	following the hydrodynamic treatment of driven interacting particle systems
	in, e.g.,
	\cite{Andjel1984},
	\cite{Rezakhanlou1991hydrodynamics}, 
	\cite{Landim1996hydrodynamics},
	\cite{Seppalainen_Discont_TASEP_2010}%
)
that the limiting density satisfies the following partial differential
equation:
\begin{equation}
	\label{PDE_for_density}
	\frac{\partial}{\partial\tau}\rho(\tau,x)
	+
	\frac{\partial}{\partial x}
	\Bigl[ \Speed(x)\phi_1^{-1}\bigl(\rho(\tau,x)/\tilde\RateLambda(x)\bigr) \Bigr]=0,
\end{equation}
with the initial condition $\rho(0,x)=0$ ($x>0$) and the boundary condition $\rho(\tau,0)=+\infty$ ($\tau\ge0$).

\begin{remark}
	\label{rmk:enough_lambda_constant}
	The case $\tilde\RateLambda(x)\equiv1$ considered in the main part of the
	paper does not restrict the generality. 
	Indeed, let $\Lambda(x):=\int_0^x \RateLambda(u)du$. 
	This is a strictly increasing function, and let $\Lambda^{-1}$ denote its inverse.
	If $\rho(\tau,x)$ satisfies \eqref{PDE_for_density} then a straightforward
	computation shows that 
	\begin{equation*}
		\check\rho(\tau,y):=\frac{\rho(\tau,\Lambda^{-1}(y))}{\RateLambda(\Lambda^{-1}(y))}
	\end{equation*}
	satisfies the same equation \eqref{PDE_for_density} in the variables $(\tau,y)$,
	with $\tilde\RateLambda$ replaced by $1$, and with the modified
	speed function $\check\Speed(y):=\Speed(\Lambda^{-1}(y))$.
\end{remark}
By virtue of \Cref{rmk:enough_lambda_constant}, we assume that
$\tilde\RateLambda(x)\equiv1$, and consider the following equation for the
density:
\begin{equation}
	\frac{\partial}{\partial\tau}\rho(\tau,x)
	+
	\frac{\partial}{\partial x}
	\Bigl[ \Speed(x)\phi_1^{-1}\bigl(\rho(\tau,x)\bigr) \Bigr]=0,
	\label{PDE_for_density_improved}
\end{equation}
with initial condition 
$\rho(0,x)=0$ ($x>0$)
and boundary condition
$\rho(\tau,0)=+\infty$
($\tau\ge0$).

\medskip

Let us verify that the limiting density coming from the asymptotic
analysis of the Fredholm determinant in \Cref{sec:asymptotic_analysis}
satisfies \eqref{PDE_for_density_improved}. 
Recall that in the absence of roadblocks 
the macroscopic limit shape $\LimitShape(\tau,x)$ for the height
function is given as follows. 
First, let $x_e=x_e(\tau)\ge0$ be the edge point, i.e., the unique solution to
$\tau=\int_{0}^{x_e}\frac{dy}{(1-q)\Speed(y)}$.
Let also $w=\upomega^{\circ}_{\tau,x}$ for $0<x<x_e$ 
be the root of $\tau w=\int_0^x\phi_2\bigl(w/\Speed(y)\bigr)dy$ on the segment
$\bigl(0,\mathop{\mathrm{ess\,inf}}_{y\in[0,x)}\Speed(y)\bigr)$ 
(%
	by \Cref{statement:TW_root_exists_and_is_unique} this root exists and is unique%
). 
The limit shape is
\begin{equation*}
	\LimitShape(\tau,x)
	=
	\begin{cases}
		\displaystyle\tau \upomega^{\circ}_{\tau,x}
		-\int_0^x \phi_1\bigl(\upomega^{\circ}_{\tau,x}/\Speed(y)\bigr)dy,
		&
		0<x<x_e;
		\\
		0,
		&
		x\ge x_e.
	\end{cases}
\end{equation*}

\begin{proposition}
	\label{statement:PDE_verification}
	The density defined as $\rho(\tau,x):=-\frac{\partial}{\partial x}\LimitShape(\tau,x)$
	satisfies equation \eqref{PDE_for_density_improved}
	(whenever all derivatives involved make sense).
\end{proposition}
\begin{proof}
	Assume that $0<x<x_e(\tau)$, otherwise the density is zero and thus trivially
	satisfies the equation.
	Differentiating $\LimitShape(\tau,x)$ in $\tau$
	and $x$ and using \Cref{statement:Phi_n_behavior} together with the definition of
	$\upomega^{\circ}_{\tau,x}$, we obtain
	\begin{align*}
		\frac{\partial}{\partial\tau}\LimitShape(\tau,x)
		&=
		\upomega^\circ_{\tau,x}+\tau \frac{\partial\upomega^\circ_{\tau,x}}{\partial\tau}
		-
		\frac{1}{\upomega^\circ_{\tau,x}}\frac{\partial\upomega^\circ_{\tau,x}}{\partial\tau}
		\int_0^x\phi_2\bigl(\upomega^{\circ}_{\tau,x}/\Speed(y)\bigr)dy=\upomega^{\circ}_{\tau,x};
		\\
		\frac{\partial}{\partial x}\LimitShape(\tau,x)
		&=
		\tau\frac{\partial\upomega^{\circ}_{\tau,x}}{\partial x}
		-
		\phi_1\bigl(\upomega^{\circ}_{\tau,x}/\Speed(x)\bigr)
		-
		\frac{1}{\upomega^{\circ}_{\tau,x}}
		\frac{\partial \upomega^{\circ}_{\tau,x}}{\partial x}
		\int_0^x\phi_2\bigl(\upomega^{\circ}_{\tau,x}/\Speed(y)\bigr)dy
		=
		-
		\phi_1\bigl(\upomega^{\circ}_{\tau,x}/\Speed(x)\bigr).
	\end{align*}
	Therefore, we can write
	\begin{equation*}
		\rho(\tau,x)
		=
		-
		\frac{\partial\LimitShape(\tau,x)}{\partial x}
		=
		\phi_1(\upomega^{\circ}_{\tau,x}/\Speed(x))
		=
		\phi_1\left( \frac{1}{\Speed(x)}
		\frac{\partial\LimitShape(\tau,x)}{\partial \tau}\right),
	\end{equation*}
	or, inverting $\phi_1$:
	\begin{equation*}
		\frac{\partial\LimitShape(\tau,x)}{\partial \tau}
		=
		\Speed(x)\phi_1^{-1}\bigl(\rho(\tau,x)\bigr).
	\end{equation*}
	Differentiating the last equality in $x$ we arrive at equation
	\eqref{PDE_for_density_improved} for $\rho(\tau,x)$.
\end{proof}

\printbibliography

\end{document}